\documentclass[oneside,english]{amsart}
\usepackage[T1]{fontenc}
\usepackage[utf8]{inputenc}
\usepackage{amsthm}
\usepackage{amssymb}
\usepackage{esint}
\usepackage[all]{xy}
\usepackage{lmodern}
\usepackage{color}

\makeatletter

\numberwithin{equation}{section}
\numberwithin{figure}{section}
\theoremstyle{plain}
\newtheorem{thm}{\protect\theoremname}[section]
  \theoremstyle{plain}
  \newtheorem{prop}[thm]{\protect\propositionname}
  \theoremstyle{definition}
  \newtheorem{defn}[thm]{\protect\definitionname}
  \theoremstyle{remark}
  \newtheorem{rem}[thm]{\protect\remarkname}
  \theoremstyle{plain}
  \newtheorem{lem}[thm]{\protect\lemmaname}
  \theoremstyle{definition}
  \newtheorem{example}[thm]{\protect\examplename}
  \newtheorem{cor}[thm]{\protect\corollaryname}
  \theoremstyle{plain}

\makeatother

\usepackage{babel}
  \providecommand{\definitionname}{Definition}
  \providecommand{\examplename}{Example}
  \providecommand{\lemmaname}{Lemma}
  \providecommand{\propositionname}{Proposition}
  \providecommand{\remarkname}{Remark}
  \providecommand{\corollaryname}{Corollary}
\providecommand{\theoremname}{Theorem}

\begin{document}

\title{Moduli stacks of algebraic structures and deformation theory}

\address{University of Luxembourg, Mathematics Research Unit,
6 rue Richard Coudenhove-Kalergi L-1359 Luxembourg}

\email{sinan.yalin@uni.lu}

\author{Sinan Yalin}
\begin{abstract}
We connect the homotopy type of simplicial moduli spaces of algebraic structures to the cohomology of their deformation complexes.
Then we prove that under several assumptions, mapping spaces of algebras over a monad in an appropriate diagram category form
affine stacks in the sense of Toen-Vezzosi's homotopical algebraic geometry. This includes simplicial moduli spaces of algebraic
structures over a given object (for instance a cochain complex). When these algebraic structures are parametrised by properads,
the tangent complexes give the known cohomology theory for such structures and there is an associated obstruction theory
for infinitesimal, higher order and formal deformations. The methods are general enough to be adapted for more general
kinds of algebraic structures.


\noindent
\textit{AMS.}18D50 ; 18D10 ; 55U10 ; 18G55 ; 14D23
\end{abstract}
\maketitle

\tableofcontents{}

\section*{Introduction}

\subsection*{Motivations and background}

Algebraic structures occurred in a fundamental way in many branches of mathematics during the last decades. For instance,
associative algebras, commutative algebras, Lie algebras, and Poisson algebras play a key role in algebra, topology (cohomology operations on topological spaces), geometry (Lie groups, polyvector fields) and mathematical physics (quantization). They all share the common feature of being defined by operations with several inputs and one single output (the associative product, the Lie bracket, the Poisson bracket). A powerful device to handle such kind of algebraic structures is the notion of operad, which appeared initially in a topological setting to encompass the structure of iterated loop spaces. Since then, operads shifted to the algebraic world and have proven to be a fundamental tool to study
algebras such as the aforementioned examples (and many others), feeding back important outcomes in various fields of mathematics: algebra, topology, deformation theory, category theory, differential, algebraic and enumerative geometry, mathematical physics...
We refer the reader to the books \cite{MSS} and \cite{LV} for a thorough study of the various aspects of this theory.

However, the work of Drinfeld in quantum group theory (\cite{Dri1}, \cite{Dri2}) shed a light on the importance of another sort of algebraic structures generalizing the previous ones, namely Lie bialgebras and non (co)commutative Hopf algebras. In such a situation, one deals not only with products but also with coproducts, as well as composition of products with coproducts. That is, one considers operations with several inputs and several
outputs. A convenient way to handle such kind of structures is to use the formalism of props, which actually goes back to \cite{MLa} (operads appeared a few years later as a special case of props fully determined by the operations with one single output). The formalism of props plays a crucial role in the deformation quantization
process for Lie bialgebras, as shown by Etingof-Kazdhan (\cite{EK1}, \cite{EK2}), and more generally in the theory of quantization functors \cite{EE}. Merkulov developped a propic formulation of deformation quantization of Poisson structures on graded manifolds, using the more general notion of wheeled prop (\cite{Mer3},
\cite{Mer4}). Props also appear naturally in topology. Two leading examples are the Frobenius bialgebra structure on the cohomology of compact oriented manifolds, reflecting
the Poincaré duality phenomenon, and the involutive Lie bialgebra structure on the equivariant homology of loop spaces on manifolds, which lies at the heart of string topology (\cite{CS1},\cite{CS2}). In geometry, some local geometric structures on formal manifolds (e. g. local Poisson structures) carry a structure of algebra over a prop \cite{Mer1}. Props also provide a concise way to encode various field theories such as topological quantum field theories
(symmetric Frobenius bialgebras) and conformal field theories (algebras over the chains of the Segal prop).

A meaningful idea to understand the behavior of such structures and, accordingly, to get more information about the
underlying objects, is to determine how one structure can be deformed (or not) into another, and organize them into
equivalence classes. That is, one wants to build a deformation theory and an associated obstruction theory
to study a given moduli problem.
Such ideas where first organized and written in a systematic way in the pioneering work of Kodaira-Spencer, who set up a theory of deformations of complex manifolds.
Algebraic deformation theory began a few years later with the work of Gerstenhaber on associative algebras and the
Gerstenhaber algebra structure on Hochschild cohomology.
In the eighties, a convergence of ideas coming from Grothendieck, Schlessinger, Deligne and Drinfeld among others,
led to the principle that any moduli problem can be interpreted as the deformation functor of a certain dg Lie algebra.
Equivalence classes of structures are then determined by equivalences classes of Maurer-Cartan elements under
the action of a gauge group.
This approach became particularly successful in the work of Goldman-Millson on representations of fundamental groups
of compact Kähler manifolds, and later in Kontsevich's deformation quantization of Poisson manifolds.
In algebraic deformation theory, the notion of deformation complex of a given associative algebra (its Hochschild complex)
was extended to algebras over operads \cite{LV} and properads(\cite{MV1},\cite{MV2}).

A geometric approach to moduli problems is to build a "space" (algebraic variety, scheme, stack)
parameterizing a given type of structures or objects (famous examples being moduli spaces of algebraic curves of fixed genus
with marked points, or moduli stacks of vector bundles of fixed rank on an algebraic variety).
Infinitesimal deformations of a point are then read on the tangent complex.
However, the usual stacks theory shows its limits when one wants to study families of objects related by an equivalence notion weaker than isomorphisms (for example, quasi-isomorphisms between complexes of vector bundles),
to gather all higher and derived data in a single object or to read the full deformation theory on the tangent spaces.
Derived algebraic geometry, and more generally homotopical algebraic geometry, is a conceptual setting
in which one can solve such problems (\cite{TV1},\cite{TV2},\cite{Toe2}).

A topological approach is to organize the structures we are interested in as points of a topological space
or a simplicial set. This idea was first shaped for algebras over operads in Rezk's thesis \cite{Rez}.
The main ingredients are the following: first, the objects encoding algebraic structures of a certain sort
(operads, properads, props...) form a model category. Second, given an operad/properad/prop... $P$,
a $P$-algebra structure on an object $X$ (a chain complex, a topological space...) can be defined as a morphism
in this model category, from $P$ to an "endomorphism object" $End_X$ built from $X$.
Third, one defines the moduli space of $P$-algebra structures on $X$ as a simplicial mapping space
of maps $P\rightarrow End_X$. Such a construction exists abstractly in any model category.
The points are the algebraic structures, the paths or $1$-simplices are homotopies between these structures, defining
an equivalence relation, and the higher simplices encode the higher homotopical data.
As we can see, such a method can be carried out in various contexts.

We end up with two deformation theories of algebraic structures: an algebraic deformation theory
through deformation complexes and their deformation functors,
and a homotopical deformation theory through homotopies between maps.
The main purpose of this article is to connect these deformation theories and gather them in a derived geometry interpretation of moduli
spaces of algebraic structures.

\subsection*{Presentation of the results}

We work with a ground field $\mathbb{K}$ of characteristic zero. We denote by  $Ch_{\mathbb{K}}$ the category of $\mathbb{Z}$-graded cochain complexes over a field $\mathbb{K}$.
In particular, all differentials are of degree $1$.

\subsubsection*{Cohomology of deformation complexes and higher homotopy groups of mapping spaces}

Properads were introduced in \cite{Val} as an intermediate object between operads and props,
parameterizing a large class of bialgebra structures, and in the same time behaving well
enough to extend several well-known constructions available for operads.
To any cochain complex $X$ one associates a properad $End_X$ called the endomorphism properad of $X$ and defined by $End_X(m,n)=Hom_{\mathbb{K}}(X^{\otimes m},
X^{\otimes n})$, where $Hom_{\mathbb{K}}$ is the differential graded hom of cochain complexes.
A $P$-algebra structure on $X$ is then a properad morphism $P\rightarrow End_X$.
Dg properads form a cofibrantly generated model category (Appendix of \cite{MV2}).
Cofibrant resolutions of a properad $P$ can always be obtained as a cobar construction $\Omega(C)$
on some coproperad $C$ (the bar construction or the Koszul dual).
Given a cofibrant resolution $\Omega(C)\stackrel{\sim}{\rightarrow}P$ of $P$ and another properad $Q$,
one considers the convolution dg Lie algebra $Hom_{\Sigma}(\overline{C},Q)$ consisting in morphisms
of $\Sigma$-biobjects from the augmentation ideal of $C$ to $Q$. The Lie bracket is the antisymmetrization
of the convolution product. This convolution product is defined similarly to the convolution product of
morphisms from a coalgebra to an algebra, using the infinitesimal coproduct of $C$ and the infinitesimal
product of $Q$. This Lie algebra enjoys the following properties:
\begin{thm}
Let $P$ be a properad with cofibrant resolution $P_{\infty}:=\Omega(C)\stackrel{\sim}{\rightarrow}P$
and $Q$ be any properad. Suppose that the augmentation ideal $\overline{C}$ is of finite dimension in each arity.
The total complex $Hom_{\Sigma}(\overline{C},Q)$ is a complete dg Lie algebra whose simplicial Maurer-Cartan set
$MC_{\bullet}(Hom_{\Sigma}(\overline{C},Q))$ is isomorphic to $Map_{\mathcal{P}}(P_{\infty},Q)$.
\end{thm}
The fact that Maurer-Cartan elements of $Hom_{\Sigma}(\overline{C},Q)$ are properad morphisms $P_{\infty}\rightarrow Q$ was already proved in \cite{MV2},
but the identification between the simplicial mapping space of such maps and the simplicial Maurer-Cartan set is a new and non trivial extension of this result.

An important construction for complete dg Lie algebras is the twisting by a Maurer-Cartan element.
The twisting of $Hom_{\Sigma}(\overline{C},Q)$ by a properad morphism $\varphi:P_{\infty}\rightarrow Q$
gives a new dg Lie algebra $Hom_{\Sigma}(\overline{C},Q)^{\varphi}$ with the same underlying graded vector space,
the same Lie bracket, but a new differential obtained from the former by adding the bracket $[\varphi,-]$.
This dg Lie algebra is the deformation complex of $\varphi$, and we have an isomorphism
\[
Hom_{\Sigma}(\overline{C},Q)^{\varphi} \cong Der_{\varphi}(\Omega(C),Q)
\]
where the right-hand term is the complex of derivations with respect to $\varphi$.
\begin{cor}
Let $\phi:P_{\infty}\rightarrow Q$ be a Maurer-Cartan element of $Hom_{\Sigma}(\overline{C},Q)$
and $Hom_{\Sigma}(\overline{C},Q)^{\phi}$ the corresponding twisted Lie algebra.
For every integer $n\geq 0$, we have a bijection
\[
H^{-n}(Hom_{\Sigma}(\overline{C},Q)^{\phi})\cong \pi_{n+1}(Map_{\mathcal{P}}(P_{\infty},Q),\phi)
\]
which is an isomorphism of abelian groups for $n\geq 1$, and an isomorphism of groups for $n=0$
where $H^0(Hom_{\Sigma}(\overline{C},Q)^{\phi})$ is equipped with the group structure given by the Hausdorff-Campbell formula.
\end{cor}
It turns out that the same results hold more generally if $P$ is a properad with minimal model $(\mathcal{F}(s^{-1}C),\partial)\stackrel{\sim}{\rightarrow}P$ for a certain homotopy coproperad $C$, and $Q$ is any properad. In this case, the complex $Hom_{\Sigma}(\overline{C},Q)$ is a complete dg $L_{\infty}$ algebra.

Moreover, there is also a version of Theorem 0.1 for algebras over operads whose proof proceeds similarly.
For any operad $P$ with cofibrant resolution $\Omega(C)$, and any $P$-algebra $A$, there is a quasi-free $P$-algebra $A_{\infty}=((P\circ C)(A),\partial)$ over a quasi-free
$C$-coalgebra $C(A)$ which gives a cofibrant resolution of $A$ (see \cite{Fre1} for a detailed construction).
\begin{thm} Let $P$ be an operad and $C$ be a cooperad such that $\Omega(C)$ is a cofibrant resolution of $P$. Let $A,B$ be any $P$-algebras.
The total complex $Hom_{dg}(C(A),B)$ is a complete dg Lie algebra whose simplicial Maurer-Cartan set $MC_{\bullet}(Hom_{dg}(C(A),B))$
is isomorphic to $Map(A_{\infty},B)$.
\end{thm}
\begin{cor}
Let $\phi:A_{\infty}\rightarrow B$ be a Maurer-Cartan element of $Hom_{dg}(C(A),B)$
and $Hom_{dg}(C(A),B)^{\phi}$ the corresponding twisted Lie algebra.
For every integer $n\geq 0$, we have a bijection
\[
H^{-n}(Hom_{dg}(C(A),B)^{\phi})\cong \pi_{n+1}(Map(A_{\infty},B),\phi)
\]
which is an isomorphism of abelian groups for $n\geq 1$, and an isomorphism of groups for $n=0$
where $H^0(Hom_{dg}(C(A),B)^{\phi})$ is equipped with the group structure given by the Hausdorff-Campbell formula.
\end{cor}
The twisted complex $Hom_{dg}(A_{\infty},B)^{\phi}$ actually gives the André-Quillen cohomology
of $A_{\infty}$ with coefficients in $B$.
For instance, we obtain
\[
H^n_{AQ}(A_{\infty},A_{\infty})\cong \pi_{n+1}(haut(A_{\infty}),\phi)
\]
where $H^n_{AQ}$ is the André-Quillen cohomology and $A_{\infty}$ is seen as an $A_{\infty}$-module via $\phi$.

\subsubsection*{Deformation functors from properad morphisms and geometric interpretation}

Let $Art_{\mathbb{K}}$ be the category of local commutative artinian $\mathbb{K}$-algebras.
To any complete dg Lie algebra $g$ one associates its deformation functor defined by
\begin{eqnarray*}
Def_g:Art_{\mathbb{K}} & \rightarrow & Set \\
A & \mapsto & \mathcal{MC}(g\otimes A),
\end{eqnarray*}
where $\mathcal{MC}(g\otimes A)$ is the moduli set of Maurer-Cartan elements of the dg Lie algebra
$g\otimes A$ (wich is still complete, see Corollary 2.4 of \cite{Yal3}).
Our main observation here is that we can replace $Def_{Hom_{\Sigma}(\overline{C},Q)}$ by a more convenient
and naturally isomorphic functor
\[
\pi_0P_{\infty}\{Q\}(-):Art_{\mathbb{K}}\rightarrow Set
\]
defined by
\[
\pi_0P_{\infty}\{Q\}(R)=\pi_0P_{\infty}\{Q\otimes_e R\}
\]
where $\otimes_e$ is the external tensor product defined by $(Q\otimes_e R)(m,n)=Q(m,n)\otimes R$.
We prove that the homotopical $R$-deformations of $\varphi$ correspond to its algebraic $R$-deformations:
\begin{thm}
For every artinian algebra $R$ and any morphism $\varphi:P_{\infty}\rightarrow Q$, there are group isomorphisms
\[
\pi_{*+1}(Map_{Prop(R-Mod)}(P_{\infty}\otimes_e R,Q\otimes_e R),\varphi\otimes_e id_R)\cong H^{-*}(Hom_{\Sigma}(\overline{C},Q)^{\varphi}\otimes R).
\]
\end{thm}

These moduli spaces also admit, consequently, an algebraic geometry interpretation.
Indeed, the deformation functor
\[
Def_{Hom_{\Sigma}(\overline{C},Q)}:Art_{\mathbb{K}}\rightarrow Set
\]
extends to a pseudo-functor of groupoids
\[
\underline{Def}_{Hom_{\Sigma}(\overline{C},Q)}:Alg_{\mathbb{K}}\rightarrow Grpd
\]
where $Alg_{\mathbb{K}}$ is the category of commutative $\mathbb{K}$-algebras and $Grpd$
the $2$-category of groupoids. This pseudo-functor is defined by sending any algebra $A$
to the Deligne groupoid of $Hom_{\Sigma}(C,Q)\otimes A$, that is, the groupoid whose objects are Maurer-Cartan
elements of $Hom_{\Sigma}(C,Q)\otimes A$ and morphisms are given by the action of the gauge group
$exp(Hom_{\Sigma}(C,Q)^0\otimes A)$. Such a pseudo-functor forms actually a prestack, whose stackification
gives the quotient stack
\[
[MC(Hom_{\Sigma}(C,Q))/exp(Hom_{\Sigma}(C,Q)^0)]
\]
of the Maurer-Cartan scheme $MC(Hom_{\Sigma}(C,Q))$ by the action of the prounipotent algebraic group $exp(Hom_{\Sigma}(C,Q)^0)]$.
It turns out that the $0^{th}$ cohomology group of the tangent complex of such a stack, encoding equivalences classes of infinitesimal deformations of a $\mathbb{K}$-point of this stack, is exactly
\[
t_{Def_{Hom_{\Sigma}(C,Q)^{\phi}}}=H^{-1}(Hom_{\Sigma}(C,Q)^{\phi}).
\]
We refer the reader to \cite{Yal3} for a proof of these results.
However, this geometric structure does not capture the whole deformation theory of the points.
For this, the next part develops such a geometric interpretation in the context of homotopical algebraic geometry.

\subsubsection*{Higher stacks from mapping spaces of algebras over monads}

Now we give a geometric interpretation of our moduli spaces of algebraic structures.
Homotopical algebraic geometry is a wide subject for which we refer the reader to \cite{TV1} and \cite{TV2}.
In this paper, we just outline in 4.1 the key ideas in the construction of higher stacks.
We fix two Grothendieck universes $\mathbb{U}\in\mathbb{V}$ such that $\mathbb{N}\in\mathbb{V}$,
as well as a HAG context $(\mathcal{C},\mathcal{C}_0,\mathcal{A},\tau,P)$ (see \cite{TV2} for a
definition). In particular, the category
$\mathcal{C}$ is a $\mathbb{V}$-small $\mathbb{U}$-combinatorial symmetric monoidal model category.
We fix also a regular cardinal $\kappa$ so that (acyclic) cofibrations of $\mathcal{C}$ are generated by
$\kappa$-filtered colimits of generating (acyclic) cofibrations.
For technical reasons, we also need to suppose that the tensor product preserves fibrations.
This assumption is satisfied in particular by topological spaces, simplicial sets, simplicial modules over a ring,
cochain complexes over a ring.
We consider a $\mathbb{U}$-small category $I$ and the associated category of diagrams $\mathcal{C}^I$,
as well as a monad $T:\mathcal{C}^I\rightarrow\mathcal{C}^I$ preserving $\kappa$-filtered colimits.
We also assume some technical assumptions about the compatibility between $T$ and modules over commutative monoids in $\mathcal{C}$.
Any homotopy mapping space
$Map(X_{\infty},Y)$ between a cofibrant $T$-algebra $X_{\infty}$ and a fibrant $T$-algebra $Y$ (here these are diagrams
with values in fibrant objects of $\mathcal{C}$) gives rise to the functor
\[
\underline{Map}(X_{\infty},Y): Comm(\mathcal{C}) \rightarrow sSet
\]
which sends any $A\in Comm(\mathcal{C})$ to $Mor_{T-Alg}(X_{\infty},(Y\otimes A^{cf})^{\Delta^{\bullet}})$,
where $(-)^{cf}$ is a functorial fibrant-cofibrant replacement in $\mathcal{C}$ and $(-)^{\Delta^{\bullet}}$
a simplicial resolution in $T$-algebras.
\begin{thm}
Let $X_{\infty}$ be a cofibrant $T$-algebra and $Y$ be a $T$-algebra
with values in fibrations of fibrant perfect objects of $\mathcal{C}$. Then $\underline{Map}(X_{\infty},Y)$ is a representable stack.
\end{thm}
Theorem 0.6 admits the following variation, whose proof is completely similar:
\begin{thm}
Suppose that all the objects of $\mathcal{C}$ are cofibrant.
Let $X_{\infty}$ be a cofibrant $T$-algebra with values in fibrant objects and $Y$ be a $T$-algebra with values in fibrations between
fibrant dualizable objects of $\mathcal{C}$. Then $\underline{Map}(X_{\infty},Y)$ is a representable stack.
\end{thm}
This holds for instance when $\mathcal{C}=sSet$ is the category of simplicial sets.
When $\mathcal{C}=Ch_{\mathbb{K}}$, all the objects are fibrant and cofibrant so we just have to suppose that
$Y$ take its values in dualizable cochain complexes, which are the bounded cochain complexes of finite dimension in each degree.
\begin{cor}
(1) Let $P_{\infty}=\Omega(C)\stackrel{\sim}{\rightarrow} P$ be a cofibrant resolution of a dg properad $P$
and $Q$ be any dg properad such that each $Q(m,n)$ is a bounded complex of finite dimension in each degree.
The functor
\[
\underline{Map}(P_{\infty},Q):A\in CDGA_{\mathbb{K}}\mapsto Map_{Prop}(P_{\infty},Q\otimes A)
\]
is an affine stack in the setting of complicial algebraic geometry of \cite{TV2}.

(2) Let $P_{\infty}=\Omega(C)\stackrel{\sim}{\rightarrow} P$ be a cofibrant resolution of a dg properad $P$ in non positively graded cochain complexes,
and $Q$ be any properad such that each $Q(m,n)$ is a finite dimensional vector space.
The functor
\[
\underline{Map}(P_{\infty},Q):A\in CDGA_{\mathbb{K}}\mapsto Map_{Prop}(P_{\infty},Q\otimes A)
\]
is an affine stack in the setting of derived algebraic geometry of \cite{TV2}, that is, an affine derived scheme.
\end{cor}

\begin{rem}
In the derived algebraic geometry context, the derived stack $\underline{Map}(P_{\infty},Q)$ is not affine anymore wether the $Q(m,n)$ are not finite dimensional vector spaces.
However, we expect these stacks to be derived $n$-Artin stacks for the $Q(m,n)$ being perfect complexes with a given finite amplitude.
We refer the reader to Section 4.4 for an idea of why such a result seems reasonable, based on the characterization of derived $n$-Artin stacks via resolutions by Artin $n$-hypergroupoids
given in \cite{Pri}.
\end{rem}

We obtain in this way moduli stacks arising from mapping spaces of operads, cyclic operads, modular operads, $1/2$-props, dioperads, properads, props, their colored versions
and their wheeled versions. An exhaustive list of examples is impossible to write down here, but we decided in Section 5 to focus on several examples
which are of great interest in topology, geometry and mathematical physics:
\begin{itemize}
\item the moduli stack of realizations of Poincaré duality on the cochains of a compact oriented manifold;
\item the moduli stack of chain-level string topology operations on the $S^1$-equivariant chains on the loop space of a smooth manifold;
\item the moduli stack of Poisson structures on formal graded manifolds (in particular, formal Poisson structures on
$\mathbb{R}^d$);
\item the moduli stack of complex structures on formal manifolds;
\item the moduli stack of deformation quantizations of graded Poisson manifolds;
\item the moduli stack of quantization functors (in particular for Lie bialgebras);
\item the moduli stack of homotopy operad structures on a given dg operad.
\end{itemize}

\subsubsection*{Tangent complexes, higher automorphisms and obstruction theory}

We denote by $\mathbb{T}_{\underline{Map}(P_{\infty},Q),x_{\varphi}}$ the tangent complex of $\underline{Map}(P_{\infty},Q)$ at an $A$-point $x_{\varphi}$ associated
to a properad morphism $\varphi:P_{\infty}\rightarrow Q\otimes_e A$.
\begin{thm}
Let $P_{\infty}=\Omega(C)\stackrel{\sim}{\rightarrow} P$ be a cofibrant resolution of a dg properad $P$
and $Q$ be any dg properad such that each $Q(m,n)$ is a bounded complex of finite dimension in each degree.
Let $A$ be a perfect cdga. We have isomorphisms
\[
H^{-*}(\mathbb{T}_{\underline{Map}(P_{\infty},Q),x_{\varphi}}) \cong
H^{-*+1}(Hom_{\Sigma}(\overline{C},Q)^{\varphi}\otimes A)
\]
for any $*\geq 1$.
\end{thm}
To put it in words, non-positive cohomology groups of the deformation complex correspond to negative groups of
the tangent complex, which computes the higher automorphisms of the given point (homotopy groups of its homotopy
automorphisms.
\begin{thm}
(1) If $H^2(\mathbb{T}_{\underline{Map}(X_{\infty},Y),x_{\varphi}}[-1])=H^1(\mathbb{T}_{\underline{Map}(X_{\infty},Y),x_{\varphi}})=0$ then for every integer $n$,
every deformation of order $n$ lifts to a deformation of order $n+1$.
Thus any infinitesimal deformation of $\varphi$ can be extended to a formal deformation.

(2) Suppose that a given deformation of order $n$ lifts to a deformation of order $n+1$. Then the set of such lifts forms a torsor under the action
of the cohomology group $H^1(\mathbb{T}_{\underline{Map}(X_{\infty},Y),x_{\varphi}}[-1])=H^0(\mathbb{T}_{\underline{Map}(X_{\infty},Y),x_{\varphi}})$. In particular, if
$H^0(\mathbb{T}_{\underline{Map}(X_{\infty},Y),x_{\varphi}})=0$ then such a lift is unique up to equivalence.
\end{thm}
\begin{rem}
These results work as well for $A$-deformations, where $A$ is a cdga which is perfect as a cochain complex.
\end{rem}
Under the assumptions of Theorem 0.11 we have
\[
\underline{Map}_{T-Alg}(X_{\infty},Y) \sim \mathbb{R}\underline{Spec}_{C(X_{\infty},Y)},
\]
and for a given $A$-point $x_{\varphi}:\mathbb{R}\underline{Spec}_A\rightarrow \underline{Map}_{T-Alg}(X_{\infty},Y)$
corresponding to a morphism $\varphi\otimes A:X_{\infty}\otimes_e A\rightarrow Y\otimes_e A$,
the tangent complex is
\[
\mathbb{T}_{\underline{Map}(X_{\infty},Y),x_{\varphi}} \cong Der_A(C(X_{\infty},Y),A)
\]
where $\mathbb{R}\underline{Hom}_{A-Mod}$ is the derived hom of the internal hom in the category of
dg $A$-modules.
The complex $\mathbb{T}_{\underline{Map}(X_{\infty},Y),x_{\varphi}}$ is a dg Lie algebra, hence
the following definition:
\begin{defn}
The $A$-deformation complex of the $T$-algebras morphism $X_{\infty}\rightarrow Y$ is
$\mathbb{T}_{\underline{Map}(X_{\infty},Y),x_{\varphi}}$.
\end{defn}
In the particular case where $I=\mathbb{S}$, we consider $T$-algebras in the category of $\Sigma$-biobjects of
$\mathcal{C}$. When all the objects of $\mathcal{C}$ are fibrant, all the assumptions of Theorem 0.7 are
satisfied. Consequently, we have a well-defined and meaningful notion of deformation complex of any morphism
of $T$-algebras with cofibrant source, with the associated obstruction theory.
This gives a deformation complex, for instance, for morphisms of
cyclic operads, modular operads, wheeled properads...more generally, morphisms of polynomial monads as defined
in \cite{BB}. In the situation where it makes sense to define an algebraic structure via a morphism towards
an "endomorphism object" (operads, properads, props, cyclic and modular operads, wheeled prop...) this gives
a deformation complex of algebraic structures and its obstruction theory.

\subsubsection*{Another description}

When the resolution $P_{\infty}$ satisfies some finiteness assumptions, the function ring of the moduli
stack $\underline{Map}(P_{\infty},Q)$ can be made explicit in both complicial and derived algebraic geometry contexts:
it is the Chevalley-Eilenberg complex of the convolution Lie algebra $Hom_{\Sigma}(\overline{C},Q)$.
Moreover, we can then prove that the whole cohomology of the tangent complex at a $\mathbb{K}$-point $\varphi$
is isomorphic to the cohomology of $Hom_{\Sigma}(\overline{C},Q)^{\varphi}$ up to degree shift, giving
an explicit description of the obstruction groups of this moduli stack (not only the groups computing higher automorphisms):
\begin{thm}
Let $P$ be a dg properad equipped with a cofibrant resolution $P_{\infty}:=\Omega(C)\stackrel{\sim}{\rightarrow}P$,
where $C$ admits a presentation $C=\mathcal{F}(E)/(R)$, and $Q$ be a dg properad such that each $Q(m,n)$ is a bounded
complex of finite dimension in each degree. Let us suppose that each $E(m,n)$ is of finite dimension, and that there exists an integer $N$ such that $E(m,n)=0$ for $m+n>N$. Then

(1) The moduli stack $\underline{Map}(P_{\infty},Q)$ is isomorphic to $\mathbb{R}Spec_{C^*(Hom_{\Sigma}(\overline{C},Q))}$,
where $C^*(Hom_{\Sigma}(\overline{C},Q))$ is the Chevalley-Eilenberg algebra of $Hom_{\Sigma}(\overline{C},Q)$.

(2) The cohomology of the tangent dg Lie algebra at a $\mathbb{K}$-point $\varphi:P_{\infty}\rightarrow Q$
is explicitely determined by
\[
H^*(\mathbb{T}_{\underline{Map}(P_{\infty},Q),x_{\varphi}}[-1]) \cong
H^*(Hom_{\Sigma}(\overline{C},Q)^{\varphi}).
\]
\end{thm}
\begin{rem}
The proof relies actually on a more general characterization of derived Maurer-Cartan stacks
for profinite complete dg Lie algebras and their cotangent complexes, which is of independent interest to study more
general formal moduli problems.
\end{rem}
This theorem applies in particular to the case of a Koszul properad, which includes for instance Frobenius algebras, Lie bialgebras and their variants such as involutive Lie bialgebras in string topology or shifted Lie bialgebras defining
Poisson structures on formal manifolds.

\subsection*{Perspectives}

Shifted symplectic structures on derived Artin stacks have been introduced in \cite{PTVV}.
The main result of the paper asserts the existence of shifted symplectic structures on a wide class of derived mapping
stacks, including the derived moduli stack of perfect complexes on Calabi-Yau varieties and the derived moduli
stack of perfect complexes of local systems on a compact oriented topological manifold for instance.
There is also a notion of shifted Poisson structure, and it is strongly conjectured that a shifted symplectic
structure gives rise to a shifted Poisson structure. The final goal is to set up a general framework for
$n$-quantization of $n$-shifted symplectic structures, hence a deformation quantization of these various derived moduli
stacks. Here the word $n$-quantization means a deformation of the derived category of quasi-coherent sheaves as an
$E_n$-monoidal dg category. This process covers for instance usual representations of quantum groups, the skein algebras
quantizing character varieties of surfaces, and is related to Donaldson-Thomas invariants. It also gives rise to
a bunch of unknown quantizations of moduli stacks.

An interesting perspective for a future work is to go further in the derived geometric study of moduli spaces of algebraic structures, by determining under which conditions moduli stacks of algebraic structures can be
equipped with shifted symplectic structures. An ultimate goal is to set up a general theory of deformation quantization
of such moduli stacks and analyze the outcomes of such quantizations in the various topics where such structures occur
(for instance, formal Poisson and formal complex structures).

\noindent
\textit{Organization of the paper.}
Section 1 is devoted to recollections about algebras over operads, props, properads, and the key notion of homotopy
algebra. Section 2 relates higher homotopy groups of
mapping spaces to cohomology groups of deformation complexes. After explaining how to use simplicial mapping spaces
to define a meaningful notion of moduli space of algebraic structures, we prove this result in 2.2 and 2.3.
Sections 2.4, 2.5 and 2.6 point out links with André-Quillen cohomology, rational homotopy groups of these moduli spaces
and long exact sequences between cohomology theories of algebraic structures on a given complex.
Section 3 details the deformation functors viewpoint.
We explain that the usual Maurer-Cartan deformation functor can be replaced by the connected components of our moduli
spaces and that homotopical deformations over an artinian algebra encodes algebraic deformations. Moreover, such
connected components can be enhanced with the structure of an algebraic stack encoding infinitesimal algebraic
deformations in its tangent complexes. In Section 4, we go further in the geometrical intepretation, by proving in 4.2
that any mapping space of algebras over a monad in diagrams over a suitable symmetric monoidal model category
carries the structure of a representable higher stack in the complicial algebraic geometry setting.
This applies to all the known objects encoding algebraic structures, that is, not only operads, properads and props
but also algebras over polynomial monads (hence modular operads and wheeled props for instance).
We detail in Section 5 several motivating examples coming from topology, geometry and mathematical physics.
In Section 6 we prove that in the case of operads, properads and props, the tangent complexes of such stacks are the well known deformation complexes studied for instance in \cite{MV2}.
We also provide the obstruction theory associated to these stacks for infinitesimal, finite order
and formal deformations.
Section 7 briefly shows how to get explicitely the function ring of our moduli stacks and the associated obstruction groups
under some mild finiteness assumtions.
In Section 8, we emphasize that the methods used could be general enough
to be explicitely transposed to more general algebras over monads (polynomial monads for instance).
We conclude by some perspectives of future work about the existence of shifted symplectic structures on such moduli stacks.

\section{Algebras over operads and prop(erad)s}

We work with a ground field $\mathbb{K}$ of characteristic zero. We denote by  $Ch_{\mathbb{K}}$ the category of $\mathbb{Z}$-graded cochain complexes over a field $\mathbb{K}$.
In particular, all differentials are of degree $1$.

\subsection{On $\Sigma$-bimodules, props and algebras over a prop}

Let $\mathbb{S}$ be the
category having the pairs $(m,n)\in\mathbb{N}^{2}$ as objects
together with morphisms sets such that:
\[
Mor_{\mathbb{S}}((m,n),(p,q))=\begin{cases}
\Sigma_{m}^{op}\times\Sigma_{n}, & \text{if $(p,q)=(m,n)$},\\
\emptyset & \text{otherwise}.
\end{cases}
\]
The (differential graded) $\Sigma$-biobjects in $Ch_{\mathbb{K}}$ are the $\mathbb{S}$-diagrams
in $Ch_{\mathbb{K}}$. So a $\Sigma$-biobject is a double sequence $\{M(m,n)\in Ch_{\mathbb{K}}\}_{(m,n)\in\mathbb{N}^{2}}$
where each $M(m,n)$ is equipped with a right action of $\Sigma_{m}$
and a left action of $\Sigma_{n}$ commuting with each other.
\begin{defn}
A dg prop is a $\Sigma$-biobject endowed with associative horizontal products
\[
\circ_{h}:P(m_{1},n_{1})\otimes P(m_{2},n_{2})\rightarrow P(m_{1}+m_{2},n_{1}+n_{2}),
\]
vertical associative composition products
\[
\circ_{v}:P(k,n)\otimes P(m,k)\rightarrow P(m,n)
\]
and units $1\rightarrow P(n,n)$ neutral for both composition products.
These products satisfy the exchange law
\[
(f_1\circ_h f_2)\circ_v(g_1\circ_h g_2) = (f_1\circ_v g_1)\circ_h(f_2\circ_v g_2)
\]
and are compatible with the actions of symmetric groups.
Morphisms of props are equivariant morphisms of collections compatible with the composition products.
We denote by $\mathcal{P}$ the category of props.
\end{defn}
Let us note that Appendix A of \cite{Fre1} provides a construction of the free prop
on a $\Sigma$-biobject. The free prop functor is left adjoint to
the forgetful functor:
\[
F:Ch_{\mathbb{K}}^{\mathbb{S}}\rightleftarrows\mathcal{P}:U.
\]

The following definition shows how a given prop encodes algebraic operations on the tensor powers
of a complex:
\begin{defn}

(1) The endomorphism prop of a complex $X$ is given by $End_X(m,n)=Hom_{dg}(X^{\otimes m},X^{\otimes n})$
where $Hom_{dg}(-,-)$ is the internal hom bifunctor of $Ch_{\mathbb{K}}$.

(2) Let $P$ be a dg prop. A $P$-algebra
is a complex $X$ equipped with a prop morphism $P\rightarrow End_X$.
\end{defn}
Hence any ``abstract'' operation of $P$ is send to an operation on $X$, and the way abstract operations
compose under the composition products of $P$ tells us the relations satisfied by the corresponding
algebraic operations on $X$.

One can perform similar constructions in the category of colored $\Sigma$-biobjects
in order to define colored props and their algebras:
\begin{defn}
Let $C$ be a non-empty set, called the \emph{set of colors}.

(1) A \emph{$C$-colored $\Sigma$-biobject} $M$ is a double sequence of
complexes $\{M(m,n)\}_{(m,n)\in\mathbb{N}^{2}}$ where
each $M(m,n)$ admits commuting left $\Sigma_{m}$ action and right
$\Sigma_{n}$ action as well as a decomposition
\[
M(m,n)=colim_{c_{i},d_{i}\in C}M(c_{1},...,c_{m};d_{1},...,d_{n})
\]
compatible with these actions. The objects $M(c_{1},...,c_{m};d_{1},...,d_{n})$
should be thought as spaces of operations with colors $c_{1},...,c_{m}$
indexing the $m$ inputs and colors $d_{1},...,d_{n}$ indexing the
$n$ outputs.

(2) A \emph{$C$-colored prop $P$} is a $C$-colored $\Sigma$-biobject
endowed with a horizontal composition
\begin{align*}
\circ_{h}:P(c_{11},...,c_{1m_{1}};d_{11},...,d_{1n_{1}})\otimes...\otimes P(c_{k1},...,c_{km_{k}};d_{k1},...,d_{kn_{1}}) & \rightarrow\\
P(c_{11},...,c_{km_{k}};d_{k1},...,d_{kn_{k}})\subseteq P(m_{1}+...+m_{k},n_{1}+...+n_{k})\\
\end{align*}
and a vertical composition
\[
\circ_{v}:P(c_{1},...,c_{k};d_{1},...,d_{n})\otimes P(a_{1},...,a_{m};b_{1},...,b_{k})\rightarrow P(a_{1},...,a_{m};d_{1},...,d_{n})\subseteq P(m,n)
\]
which is equal to zero unless $b_{i}=c_{i}$ for $1\leq i\leq k$.
These two compositions satisfy associativity axioms (we refer the
reader to \cite{JY} for details).
\end{defn}
\medskip{}

\begin{defn}
(1) Let $\{X_{c}\}_{C}$ be a collection of complexes.
The \emph{$C$-colored endomorphism prop} $End_{\{X_{c}\}_{C}}$ is defined
by
\[
End_{\{X_{c}\}_{C}}(c_{1},...,c_{m};d_{1},...,d_{n})=Hom_{dg}(X_{c_{1}}\otimes...\otimes X_{c_{m}},X_{d_{1}}\otimes...\otimes X_{d_{n}})
\]
with a horizontal composition given by the tensor product of homomorphisms
and a vertical composition given by the composition of homomorphisms
with matching colors.

(2) Let $P$ be a $C$-colored prop. A $P$-algebra is the data of
a collection of complexes $\{X_{c}\}_{C}$ and a $C$-colored prop morphism
$P\rightarrow End_{\{X_{c}\}_{C}}$.
\end{defn}

\begin{example}
Let $I$ be a small category and $P$ a prop. We can
build an $ob(I)$-colored prop $P_{I}$ such that the $P_{I}$-algebras
are the $I$-diagrams of $P$-algebras in the same
way as that of \cite{Mar1}.
\end{example}

The category of $\Sigma$-biobjects $Ch_{\mathbb{K}}^{\mathbb{S}}$
is a diagram category over $Ch_{\mathbb{K}}$, so it inherits a cofibrantly
generated model category structure in which weak equivalences and fibrations are
defined componentwise.
The adjunction $F:Ch_{\mathbb{K}}^{\mathbb{S}}\rightleftarrows\mathcal{P}:U$
transfer this model category structure to the props:
\begin{thm}
(1) (cf. \cite{Fre1}, theorem 4.9) The category of dg props $\mathcal{P}$ equipped with the
classes of componentwise weak equivalences and componentwise fibrations forms a cofibrantly generated model category.

(3) (cf. \cite{JY}, theorem 1.1) Let $C$ be a non-empty set.
Then the category $\mathcal{P}_{C}$ of dg $C$-colored props forms a cofibrantly generated model category with fibrations
and weak equivalences defined componentwise.
\end{thm}

\subsection{Properads}

Composing operations of two $\Sigma$-biobjects $M$ and $N$ amounts to consider $2$-levelled directed graphs
(with no loops) with the first level indexed by operations of $M$ and the second level by operations of $N$.
Vertical composition by grafting and horizontal composition by concatenation allows one to define props as before.
The idea of properads is to mimick operads (for operations with several outputs), which are defined as monoids in $\Sigma$-objects,
by restricting the vertical composition product to connected graphs.
We denote by $\boxtimes_c$ this connected composition product of $\Sigma$-biobjects, whose explicit formula is given in \cite{Val}.
The unit for this product is the $\Sigma$-biobject $I$ given by $I(1,1)=\mathbb{K}$ and $I(m,n)=0$ otherwise.
The category of $\Sigma$-biobjects then forms a symmetric monoidal category $(Ch_{\mathbb{K}}^{\mathbb{S}},\boxtimes_c,I)$.
\begin{defn}
A dg properad $(P,\mu,\eta)$ is a monoid in $(Ch_{\mathbb{K}}^{\mathbb{S}}\boxtimes_c,I)$,
where $\mu$ denotes the product and $\eta$ the unit.
It is augmented if there exists a morphism of properads $\epsilon:P\rightarrow I$.
In this case, there is a canonical isomorphism $P\cong I\oplus\overline{P}$
where $\overline{P}=ker(\epsilon)$ is called the augmentation ideal of $P$.
\end{defn}
Properads form a category noted $\mathcal{P}$. Morphisms of properads are collections of equivariant cochain maps compatible with respect to
the monoid structures. Properads have also their dual notion, namely coproperads:
\begin{defn}
A dg coproperad $(C,\Delta,\epsilon)$ is a comonoid in $(Ch_{\mathbb{K}}^{\mathbb{S}},\boxtimes_c,I)$.
\end{defn}
\begin{rem}
When constructing deformations complexes of algebras over properads we will need a weaker notion of
``homotopy coproperad" for which we refer the reader to \cite{MV1}.
\end{rem}
As in the prop case, there exists a free properad functor $\mathcal{F}$ forming an adjunction
\[
\mathcal{F}:Ch_{\mathbb{K}}^{\mathbb{S}}\rightleftarrows\mathcal{P} :U
\]
with the forgetful functor $U$.
There is an explicit construction of the free properad in terms of direct sums of labelled graphs for which
we refer the reader to \cite{Val}. The free properad $\mathcal{F}(M)$ on a $\Sigma$-biobject $M$
admits a filtration
\[
\mathcal{F}(M)=\bigoplus_n \mathcal{F}^{(n)}(M)
\]
by the number of vertices. Precisely, each component $\mathcal{F}^{(n)}(M)$ is formed by graphs with $n$ vertices
decorated by $M$.
Let us note that, according to \cite{BoM}, properads (as well as props and operads) can alternatively be defined as algebras
over the monad $U\circ\mathcal{F}$ induced by the adjunction above.
Dually, there exists a cofree coproperad functor denoted $\mathcal{F}_c(-)$ having the same underlying
$\Sigma$-biobject.
Moreover, according to \cite{MV2}, this adjunction equips dg properads with a cofibrantly generated model category structure with componentwise fibrations and weak equivalences.
There is a whole theory of explicit resolutions in this model category, based on the
bar-cobar construction and the Koszul duality \cite{Val}.

Finally, let us note that we have a notion of algebra over a properad analogous to an algebra over a prop,
since the endomorphism prop restricts to an endomorphism properad.

\subsection{Algebras over operads}

Operads are used to parametrize various kind of algebraic structures consisting of operations with
one single output.
Fundamental examples of operads include the operad $As$ encoding associative algebras,
the operad $Com$ of commutative algebras, the operad $Lie$ of Lie algebras and the operad
$Pois$ of Poisson algebras.
Dg operads form a model category with bar-cobar resolutions and Koszul duality \cite{LV}.
There exists several equivalent approaches for the
definition of an algebra over an operad. We will use the following one which we recall for convenience:
\begin{defn}
Let $(P,\gamma,\iota)$ be a dg operad, where $\gamma$ is the composition product and $\iota$ the unit.
A $P$-algebra is a complex $A$ endowed
with a morphism $\gamma_{A}:P(A)\rightarrow A$ such that
the following diagrams commute
\[
\xymatrix{(P\circ P)(A)\ar[r]^{P(\gamma_{A})}\ar[d]_{\gamma(A)} & P(A)\ar[d]^{\gamma_{A}}\\
P(A)\ar[r]_{\gamma_{A}} & A
}
\]

\[
\xymatrix{A\ar[r]^{\iota(A)}\ar[rd]_{=} & P(A)\ar[d]^{\gamma_{A}}\\
 & A
}
.
\]
\end{defn}
For every complex $V$, we can equip $P(V)$ with a $P$-algebra structure
by setting $\gamma_{P(V)}=\gamma (V):P(P(V)) \rightarrow P(V)$. The $P$-algebra $(P(V),\gamma(V))$ equipped with the map
$\iota(V):I(V)=V \rightarrow P(V)$ is the free $P$-algebra on $V$ (see \cite{LV}, Proposition 5.2.6).

The category of $P$-algebras satisfies good homotopical properties:
\begin{thm}(see \cite{Fre2})
The category of dg $P$-algebras inherits a cofibrantly generated model category structure such that
a morphism $f$ of $P$-algebras is

(i)a weak equivalence if $U(f)$ is a quasi-isomorphism, where $U$ is the forgetful functor;

(ii)a fibration if $U(f)$ is a fibration of cochain complexes, thus a surjection;

(iii)a cofibration if it has the left lifting property with respect to acyclic fibrations.
\end{thm}
We can also say that cofibrations are relative cell complexes with respect to the generating cofibrations,
where the generating cofibrations and generating acyclic cofibrations are, as expected,
the images of the generating (acyclic) cofibrations
of $Ch_{\mathbb{K}}$ under the free $P$-algebra functor $P$.
\begin{rem}
In general, algebras over properads and props do not inherit such a model category structure, since there is no free algebra functor.
\end{rem}

\subsection{Homotopy (bi)algebras}

Given a prop, properad or operad $P$, a homotopy $P$-algebra, or $P$-algebra up to homotopy, is
an algebra for which the relations are relaxed up to a coherent system of higher homotopies.
More precisely:
\begin{defn}
A homotopy $P$-algebra is an algebra over a cofibrant resolution $P_{\infty}$ of $P$.
\end{defn}
In the properadic setting, the bar-cobar resolution gives a functorial cofibrant resolution,
and Koszul duality is used to built smaller resolutions.
These resolutions are of the form $P_{\infty}=(\mathcal{F}(V),\partial)$ where $\partial$
is a differential obtained by summing the differential induced by $V$ with a certain derivation.
When $P$ is concentrated in degree zero, all the higher homology groups of $P_{\infty}$ vanish.
The system of higher homotopies defining a homotopy $P$-algebra then consists in the generators of $V$,
and the coboundary conditions give the relations between these higher operations.

Such a notion depends a priori on the choice of a resolution, so one naturally wants the categories
of algebras over two resolutions of the same prop to be homotopy equivalent.
This is a well known result in the operadic case, using the model structure on algebras over operads and
the machinery of Quillen equivalences.
The problem is more subtle in the general propic case, which requires different methods, and has been
solved in \cite{Yal1} and \cite{Yal2}.

Homotopy algebras are of crucial interest in deformation theory and appear in various contexts.
In particular, they appear each time one wants to transfer an algebraic structure along a quasi-isomorphism,
or when one realizes an algebraic structure on the homology of a complex into a finer homotopy algebra
structure on this complex.

\begin{rem}
The resolution $P_{\infty}$ is usually considered as a non-negatively graded properad with a differential $\partial$ of degree $-1$.
This is equivalent to consider a non-positively graded resolution with a differential of degree $1$.
\end{rem}

\section{Mapping spaces, moduli spaces and cohomology of (bi)algebras}

Recall that in a model category $M$, one can define homotopy mapping spaces $Map(-,-)$, which are simplicial sets
equipped with a composition law defined up to homotopy. There are two possible definitions:
the expression
\[
Map(X,Y)=Mor(X\otimes\Delta^{\bullet},Y)
\]
where $(-)\otimes\Delta^{\bullet}$ is a cosimplicial resolution,
and
\[Map(X,Y)=Mor(X,Y^{\Delta^{\bullet}})
\]
where $(-)^{\Delta^{\bullet}}$ is a simplicial resolution.
For the sake of brevity and clarity, we refer the reader to the chapter 16 in \cite{Hir} for a complete
treatment of the notions of simplicial resolutions, cosimplicial resolutions and Reedy model categories.
When $X$ is cofibrant and $Y$ is fibrant, these two definitions give the same homotopy type of mapping space
and have also the homotopy type of Dwyer-Kan's hammock localization $L^H(M,wM)(X,Y)$ where $wM$ is the subcategory
of weak equivalences of $M$ (see \cite{DK3}). Moreover, the set of connected components $\pi_0Map(X,Y)$
is the set of homotopy classes $[X,Y]$ in $Ho(M)$.

The space of homotopy automorphisms of an object $X$ is the simplicial sub-monoid
$haut(X^{cf})\subset Map(X^{cf},X^{cf})$ of invertible connected components, where $X^{cf}$ is a cofibrant-fibrant resolution
of $X$,  i.e
\[
haut(X^{cf})=\coprod_{\overline{\phi}\in [X,X]_{Ho(M)}^{\times}} Map(X^{cf},X^{cf})_{\phi}
\]
where the $\overline{\phi}\in [X,X]_{Ho(M)}^{\times}$ are the automorphisms in the homotopy category of $M$
and $Map(X^{cf},X^{cf})_{\phi}$ the connected component of $\phi$ in the standard homotopy mapping space.

Similarly, the space of homotopy isomorphisms from $X$ to $Y$ is the simplicial sub-monoid $hiso(X^{cf},Y^{cf})$
of $Map(X^{cf},Y^{cf})$ of invertible connected components,  i.e
\[
hiso(X^{cf},Y^{cf})=\coprod_{\overline{\phi}\in [X,Y]_{Ho(M)}^{\times}} Map(X^{cf},Y^{cf})_{\phi}
\]
where the $\overline{\phi}\in [X,Y]_{Ho(M)}^{\times}$ are the isomorphisms in the homotopy category of $M$
and $Map(X^{cf},Y^{cf})_{\phi}$ the connected component of $\phi$.

Before going to the heart of the matter, let us provide the following adjunction result, which will
be used at several places in the remaining part of this paper:
\begin{prop}
Let $F:\mathcal{C}\rightleftarrows\mathcal{D}:G$ be a Quillen adjunction. It induces natural isomorphisms
\[
Map_{\mathcal{D}}(F(X),Y) \cong Map_{\mathcal{C}}(X,G(Y))
\]
where $X$ is a cofibrant object of $\mathcal{C}$ and $Y$ a fibrant object of $\mathcal{D}$.
\end{prop}
\begin{proof}
We will use the definition of mapping spaces via cosimplicial frames.
The proof works as well with simplicial frames.
The adjunction $(F,G)$ induces an adjunction at the level of diagram categories
\[
F:\mathcal{C}^{\Delta}\rightleftarrows\mathcal{D}^{\Delta}:G.
\]
Now let $\phi:A^{\bullet}\rightarrowtail B^{\bullet}$ be a Reedy cofibration between Reedy cofibrant objects
of $\mathcal{C}^{\Delta}$. This is equivalent, by definition, to say that for every integer $r$ the map
\[
(\lambda,\phi)_r:L^rB\coprod_{L^rA}A^r\rightarrowtail B^r
\]
induced by $\phi$ and the latching object construction $L^{\bullet}A$ is a cofibration in $\mathcal{C}$.
Let us consider the pushout
\[
\xymatrix{
L^rA\ar[r]\ar[d]_{L^r\phi} & A^r\ar[d] \\
L^rB\ar[r] & L^rB\coprod_{L^rA}A^r
}.
\]
The fact that $\phi$ is a Reedy cofibration implies that for every $r$, the map $L^r\phi$ is a cofibration.
Since cofibrations are stable under pushouts, the map $A^r\rightarrow L^rB\coprod_{L^rA}A^r$ is also
a cofibration. By assumption, the cosimplicial object $A^{\bullet}$ is Reedy cofibrant, so it is in particular
pointwise cofibrant. We deduce that $L^rB\coprod_{L^rA}A^r$ is cofibrant. Similarly, each $B^r$ is cofibrant
since $B^{\bullet}$ is Reedy cofibrant. The map $(\lambda,\phi)_r$ is a cofibration between cofibrant objects
and $F$ is a left Quillen functor, so $F((\lambda,\phi)_r)$ is a cofibration of $\mathcal{D}$ between
cofibrant objects. Recall that the $r^{th}$ latching object construction is defined by a colimit.
As a left adjoint, the functor $F$ commutes with colimits so we get a cofibration
\[
L^rF(B^{\bullet})\coprod_{L^rF(A^{\bullet})}F(A^r)\rightarrowtail F(B^r).
\]
This means that $F(\phi)$ is a Reedy cofibration in $\mathcal{D}^{\Delta}$.
Now, given that Reedy weak equivalences are the pointwise equivalences, if $\phi$ is a Reedy weak equivalence
between Reedy cofibrant objects then it is in particular a pointwise weak equivalence between pointwise cofibrant
objects, hence $F(\phi)$ is a Reedy weak equivalence in $\mathcal{D}^{\Delta}$.
We conclude that $F$ induces a left Quillen functor between cosimplicial objects for the Reedy model structures.
In particular, it sends any cosimplicial frame of a cofibrant object $X$ of $\mathcal{C}$ to a cosimplicial
frame of $F(X)$.
\end{proof}
\begin{rem}
The isomorphism above holds if the cosimplicial frame for the left-hand mapping space is chosen to be
the image under $F$ of the cosimplicial frame of the right-hand mapping space.
But recall that cosimplicial frames on a given object are all weakly equivalent, so that for any choice
of cosimplicial frame we get at least weakly equivalent mapping spaces.
\end{rem}

\subsection{Moduli spaces of algebra structures over a prop}

A moduli space of algebra structures over a prop $P$, on a given complex $X$, is a simplicial set whose points are
the prop morphisms $P\rightarrow End_X$. Such a moduli space can
be more generally defined on diagrams of cochain complexes.
\begin{defn}
Let $P$ a cofibrant prop and $X$ be a cochain complex. The moduli space of $P$-algebra structures on $X$
is the simplicial set alternatively defined by
\[
P\{X\}=Mor_{\mathcal{P}}(P\otimes\Delta^{\bullet},End_X)
\]
where $P\otimes\Delta^{\bullet}$ is a cosimplicial resolution of $P$,
or
\[
P\{X\}=Mor_{\mathcal{P}}(P,End_X^{\Delta^{\bullet}}).
\]
where $End_X^{\Delta^{\bullet}}$ is a simplicial resolution of $End_X$.
Since every cochain complex over a field is fibrant and cofibrant, every dg prop is fibrant:
the fact that $P$ is cofibrant and $End_X$ is fibrant implies that these
two formulae give the same moduli space up to homotopy.
\end{defn}
We can already get two interesting properties of these moduli
spaces:
\begin{prop}
(1) The simplicial set $P\{X\}$ is a Kan complex and its connected components give the equivalences classes
of $P$-algebra structures on $X$, i.e
\[
\pi_0P\{X\}\cong [P,End_X]_{Ho(\mathcal{P})}.
\]

(2) Every weak equivalence of cofibrant props $P\stackrel{\sim}{\rightarrow}Q$ gives
rise to a weak equivalence of fibrant simplicial sets $Q\{X\}\stackrel{\sim}{\rightarrow}P\{X\}$.
\end{prop}

These properties directly follows from the properties of simplicial mapping spaces in model categories \cite{Hir}.
The higher simplices of these moduli spaces encode higher simplicial homotopies between homotopies.

For dg props we can actually get explicit simplicial resolutions:
\begin{prop}
Let $P$ be a prop in $Ch_{\mathbb{K}}$. Let us define $P^{\Delta^{\bullet}}$ by
\[
P^{\Delta^{\bullet}}(m,n)=P(m,n)\otimes A_{PL}(\Delta^{\bullet}),
\]
where $A_{PL}$ denotes Sullivan's functor of piecewise linear forms (see \cite{Sul}).
Then $P^{\Delta^{\bullet}}$ is a simplicial resolution of $P$ in the category of dg props.
\end{prop}

\begin{rem}
Since every dg prop is fibrant, according to \cite{Hir}, any simplicial frame on a dg prop is thus a simplicial resolution.
We just have to prove that $(-)^{\Delta^{\bullet}}$ is a simplicial frame.
\end{rem}

\begin{proof}
Let us first show that $P^{\Delta^{\bullet}}$ is a simplicial object in props.
We equip each $P^{\Delta^k}$, $k\in\mathbb{N}$, with the following vertical composition products
\begin{eqnarray*}
\circ_v^{P^{\Delta^k}}:P^{\Delta^k}(l,n)\otimes P^{\Delta^k}(m,l) & \cong & P(l,n)\otimes P(m,l)\otimes A_{PL}(\Delta^k)
\otimes A_{PL}(\Delta^k) \\
 & \stackrel{\circ_v^P\otimes\mu_k}{\rightarrow} & P(m,n)\otimes A_{PL}(\Delta^k)=P^{\Delta^k}(m,n)
\end{eqnarray*}
where $\circ_v^P$ is the vertical composition product of $P$ and $\mu_k$ the product of the dg commutative algebra $A_{PL}(\Delta^k)$,
and the following horizontal composition products
\begin{eqnarray*}
\circ_h^{P^{\Delta^k}}:P^{\Delta^k}(m_1,n_1)\otimes P^{\Delta^k}(m_2,n_2) & \cong & P(m_1,n_1)\otimes P(m_2,n_2)\otimes A_{PL}(\Delta^k)
\otimes A_{PL}(\Delta^k) \\
 & \stackrel{\circ_h^P\otimes\mu_k}{\rightarrow} & P(m_1+m_2,n_1+n_2)\otimes A_{PL}(\Delta^k)=P^{\Delta^k}(m,n)
\end{eqnarray*}
where $\circ_h^P$ is the horizontal composition product of $P$.
The associativity of these products follows from the associativity of $\circ_v^P$, $\circ_h^P$ and $\mu_k$.
The right and left actions of symmetric groups on $P^{\Delta^k}(m,n)$ are induced by those on $P(m,n)$.
The compatibility between these actions and the composition products comes directly from this compatibility in the
prop structure of $P$, since the product $\mu_k$ is commutative.
The interchange law of props is obviously still satisfied.
The units $\mathbb{K}\rightarrow P(n,n)$ are defined by
\[
\mathbb{K}\cong\mathbb{K}\otimes\mathbb{K}\stackrel{\eta_P\otimes\eta_k}{\rightarrow} P(m,n)\otimes A_{PL}(\Delta^k)
\]
where $\eta_P$ is the unit of $P$ and $\eta_k$ the unit of $A_{PL}(\Delta^k)$.
The faces and degeneracies are induced by those of $\Delta^k$ via the functoriality of $A_{PL}$, and form
props morphisms by definition of the prop structure on the $P^{\Delta^k}$
(recall that $A_{PL}$ maps simplicial applications to morphisms of commutative dg algebras, which commute
with products).

Now we have to check the necessary conditions to obtain a simplicial frame on $P$.
Recall that in a model category $\mathcal{M}$, a \emph{simplicial frame} on an object $X$ of $\mathcal{M}$
is a simplicial object $X^{\Delta^{\bullet}}$ satisfying the following properties:

(1) The identity $X^{\Delta^0}=X$;

(2) The morphisms $X^{\Delta^n} \rightarrow X^{\Delta^0}$ induced by the embeddings
${i}\hookrightarrow {0<...<n}$ of the category $\Delta$ form a Reedy fibration
\[
X^{\Delta^{\bullet}} \twoheadrightarrow r^{\bullet}X,
\]
where $r^{\bullet}X$ is a simplicial object such that $r_n X=\prod_{i=0}^n X$;

(3) The morphism $X^{\Delta^n}\rightarrow X^{\Delta^0}$ induced by the constant map
${0<...<n}\rightarrow {0}$ is a weak equivalence of $\mathcal{M}$.

The condition (1) is obvious by definition.
The condition (3) follows from the acyclicity of $A_{PL}(\Delta^k)$ and the Kunneth formula.
To check the condition (2) we have to prove that the associated matching map (see \cite{Hir})
is a fibration of props. By definition of such fibrations, we have to show that we get
a fibration of cochain complexes (i.e a surjection) in each biarity.
The forgetful functor from the category of props
$\mathcal{P}$ to the category of collections of cochain complexes $Ch_{\mathbb{K}}^{\mathbb{N}\times\mathbb{N}}$
is a right adjoint and thus preserves limits. Given that the matching object construction is
defined by a limit, it commutes with the forgetful functor. Moreover, the collections of cochain complexes
$Ch_{\mathbb{K}}^{\mathbb{N}\times\mathbb{N}}$ form a diagram category, and limits in diagram categories
are defined pointwise. Finally, it remains to prove that the matching map associated to
$P(m,n)\otimes A_{PL}(\Delta^{\bullet})\rightarrow r^{\bullet}P(m,n)$ is a fibration in cochain complexes
for every $(m,n)\in \mathbb{N}\times\mathbb{N}$.
This follows from the well known fact that $(-)\otimes A_{PL}(\Delta^{\bullet})$ is a simplicial frame in
$Ch_{\mathbb{K}}$ (see \cite{BG}).
\end{proof}
The same arguments prove that such a formula defines also a functorial simplicial resolution for operads
and properads.

Now let us see how to study homotopy structures with respect to a fixed one.
Let $O\rightarrow P$ be a morphism of props inducing a morphism of cofibrant props
$O_{\infty}\rightarrow P_{\infty}$ between the corresponding cofibrant resolutions.
We can form the homotopy cofiber
\[
O_{\infty}\rightarrow P_{\infty}\rightarrow (P,O)_{\infty}
\]
in the model category of dg props and thus consider the moduli space $(P,O)_{\infty}\{X\}$.
This moduli space encodes the homotopy classes (and higher simplicial homotopies) of the $P_{\infty}$-structures
on $X$ "up to $O_{\infty}$", i.e for a fixed $O_{\infty}$-structure on $X$.
A rigorous way to justify this idea is the following.
\begin{prop}
Let $\mathcal{M}$ be a model category, endowed with a functorial simplicial mapping space
\[
Map(-,-)=Mor_{\mathcal{M}}(-,(-)^{\Delta^{\bullet}})
\]
(which always exists, by existence of functorial simplicial resolutions \cite{Hir}).
Let $Y$ be a fibrant object of $\mathcal{M}$. Then the functor $Map(-,Y)$ sends cofiber sequences
induced by cofibrations between cofibrant objects to fiber sequences of Kan complexes.
\end{prop}
It follows from general properties of simplicial mapping spaces for which we refer the reader to \cite{Hir}.
We apply this result to obtain the particular homotopy fiber
\[
\xymatrix{ (P,O)_{\infty}\{X\} \ar[r] \ar[d] & P_{\infty}\{X\} \ar@{->>}[d] \\
\{X\} \ar@{^{(}->}[r] & O_{\infty}\{X\}
}.
\]
\begin{example}
We know that the $E_{\infty}$-algebra structures on the singular
cochains classify the rational homotopy type of the considered topological space.
For a Poincaré duality space, whose cochains form a unitary and counitary Frobenius bialgebra,
a structure corresponding to the notion of Frobenius algebra and encoded by a prop $ucFrob$,
a way to understand the homotopy Frobenius structures up to the rational homotopy type of this space is
to analyze $\pi_*(ucFrob,E)_{\infty}\{C_*(X;\mathbb{Q})\}$.
\end{example}

\subsection{Higher homotopy groups of mapping spaces of properads}

We first need some preliminary results about complete $L_{\infty}$-algebras.
There are two equivalent definitions of a $L_{\infty}$ algebra:
\begin{defn}
(1) A $L_{\infty}$ algebra is a graded vector space $g=\{g_n\}_{n\in\mathbb{Z}}$ equipped with maps
$l_k:g^{\otimes k}\rightarrow g$ of degree $2-k$, for $k\geq 1$, satisfying the following properties:
\begin{itemize}
\item $[...,x_i,x_{i+1},...]=-(-1)^{|x_i||x_{i+1}|}[...,x_{i+1},x_i,...]$
\item for every $k\geq 1$, the generalized Jacobi identities
\[
\sum_{i=1}^k\sum_{\sigma\in Sh(i,k-i)}(-1)^{\epsilon(i)}[[x_{\sigma(1)},...,x_{\sigma(i)}],x_{\sigma(i+1)},...,x_{\sigma(k)}]=0
\]
where $\sigma$ ranges over the $(i,k-i)$-shuffles and
\[
\epsilon(i) = i+\sum_{j_1<j_2,\sigma(j_1)>\sigma(j_2)}(|x_{j_1}||x_{j_2}|+1).
\]
\end{itemize}

(2) A $L_{\infty}$ algebra structure on a graded vector space $g=\{g_n\}_{n\in\mathbb{Z}}$ is a
coderivation of cofree coalgebras $Q:\Lambda^{\bullet\geq 1}sg\rightarrow \Lambda^{\bullet\geq 1}sg$ of degree $1$ such that $Q^2=0$.
\end{defn}
The bracket $l_1$ is actually the differential of $g$ as a cochain complex. When the brackets $l_k$ vanish
for $k\geq 3$, then one gets a dg Lie algebra.
The dg algebra $C^*(g)$ obtained by dualizing the dg coalgebra of (2) is called the Chevalley-Eilenberg algebra of $g$.

A $L_{\infty}$ algebra $g$ is filtered if it admits a decreasing filtration
\[
g=F_1g\supseteq F_2g\supseteq...\supseteq F_rg\supseteq ...
\]
compatible with the brackets: for every $l\geq 1$,
\[
[F_rg,g,...,g]\in F_rg.
\]
We suppose moreover that for every $r$, there exists an integer $N(r)$ such that $[g^{\wedge l}]\subseteq F_rg$
for every $l>N(r)$.
A filtered $L_{\infty}$ algebra $g$ is complete if the canonical map $g\rightarrow lim_rg/F_rg$ is an isomorphism.

The completeness of a $L_{\infty}$ algebra allows to define properly the notion of Maurer-Cartan element:
\begin{defn}
(1) Let $g$ be a dg $L_{\infty}$-algebra and $\tau\in g^1$, we say that $\tau$ is a Maurer-Cartan element of $g$ if
\[
\sum_{k\geq 1} \frac{1}{k!} [\tau^{\wedge k}]=0.
\]
The set of Maurer-Cartan elements of $g$ is noted $MC(g)$.

(2) The simplicial Maurer-Cartan set is then defined by
\[
MC_{\bullet}(g)=MC(g\hat{\otimes}\Omega_{\bullet}),
\],
where $\Omega_{\bullet}$ is the Sullivan cdga of de Rham polynomial forms on the standard simplex $\Delta^{\bullet}$ (see \cite{Sul})
and $\hat{\otimes}$ is the completed tensor product with respect to the filtration induced by $g$.
\end{defn}
The simplicial Maurer-Cartan set is a Kan complex, functorial in $g$ and preserves quasi-isomorphisms of complete $L_{\infty}$-algebras.
The Maurer-Cartan moduli set of $g$ is $\mathcal{MC}(g)=\pi_0MC_{\bullet}(g)$: it is the quotient of the set
of Maurer-Cartan elements of $g$ by the homotopy relation defined by the $1$-simplices.
When $g$ is a complete dg Lie algebra, it turns out that this homotopy relation is equivalent to the action of the gauge
group $exp(g^0)$ (a prounipotent algebraic group acting on Maurer-Cartan elements), so in this case
this moduli set coincides with the one usually known for Lie algebras.
We refer the reader to \cite{Yal3} for more details about all these results.

We also recall briefly the notion of twisting by a Maurer-Cartan element.
The twisting of a complete $L_{\infty}$ algebra $g$ by a Maurer-Cartan element $\tau$ is the complete $L_{\infty}$ algebra $g^{\tau}$
with the same underlying graded vector space and new brackets $l_k^{\tau}$ defined by
\[
l_k^{\tau}(x_1,...,x_k)=\sum_{i\geq 0}\frac{1}{i!}l_{k+i}(\tau^{\wedge i},x_1,...,x_k)
\]
where the $l_k$ are the brackets of $g$.
Thereafter we will need the following theorem:
\begin{thm}(Berglund \cite{Ber})
Let $g$ be a complete $L_{\infty}$-algebra and $\tau$ be a Maurer-Cartan element of $g$.
There is an isomorphism of abelian groups
\[
H^{-n}(g^{\tau})\cong \pi_{n+1}(MC_{\bullet}(g),\tau)
\]
for $n\geq 1$, and an isomorphism of groups
\[
exp(H^0(g^{\tau}))\cong \pi_1(MC_{\bullet}(g),\tau)
\]
for $n=0$ where $exp(H^0(g^{\tau}))$ is equipped with the group structure given by the Hausdorff-Campbell formula
(it is the prounipotent algebraic group associated to the pronilpotent Lie algebra $H^0(g^{\tau})$).
\end{thm}
Recall that we use a cohomological grading convention instead of the homological grading convention \cite{Ber}. This is the reason why homotopy groups of Maurer-Cartan simplicial
sets are related here to non-positive cohomology groups.

Let $P$ be a properad and $C$ a coproperad equipped with a twisting morphism $C\rightarrow P$,
such that $\Omega(C)\stackrel{\sim}{\rightarrow} P$ is a cofibrant resolution of $P$,
where $\Omega(-)$ is the properadic cobar construction (see \cite{Val}).
One can always produce such a resolution by taking for $C$ the bar contruction on $P$, or the Koszul dual of $P$ if $P$ is Koszul.

We consider a $\mathbb{Z}$-graded cochain complex $Hom_{\Sigma}(\overline{C},Q)$, where $Q$ is any dg properad, with the degree defined by
\[
f\in Hom_{\Sigma}(\overline{C},Q)_k\Leftrightarrow \forall n, f(\overline{C}_n)\subset Q_{n+k}
\]
and the differential defined by
\[
\delta(f)=d_Q\circ f - (-1)^{|f|}f\circ (d_C)
\]
where $d_Q$ is the differential of $Q$, $|f|$ the degree of $f$,
and $d_C$ the differential of $C$ restricted to the augmentation ideal $\overline{C}$.
This is actually the external differential graded hom of $\Sigma$-biobjects, given by the infinite product
\[
Hom_{\Sigma}(\overline{C},Q)=\prod_{m,n} Hom_{dg}(\overline{C}(m,n),Q(m,n))^{\Sigma_m\times\Sigma_n}.
\]
Let us note that according to Theorem 12 of \cite{MV2}, for every properad morphism $\varphi:P_{\infty}\rightarrow Q$ we have an isomorphism
\[
Hom_{\Sigma}(\overline{C},Q)^{\varphi} \cong Der_{\varphi}(P_{\infty},Q),
\]
where $Der_{\varphi}(P_{\infty},Q)$ is the complex of properadic derivations from $P_{\infty}$ to $Q$
with $Q$ equipped with the $P_{\infty}$-module structure induced by $\varphi$.

The principal theorem of this section is the following:
\begin{thm} Let $P$ be a properad with cofibrant resolution $P_{\infty}=\Omega(C)\stackrel{\sim}{\rightarrow}P$
and $Q$ be any properad. Suppose that the augmentation ideal $\overline{C}$ is of finite dimension in each arity.

(i) The total complex $Hom_{\Sigma}(\overline{C},Q)$ is a complete dg Lie algebra.

(ii) The Maurer-Cartan elements of $Hom_{\Sigma}(\overline{C},Q)$ are exactly the properad morphisms $P_{\infty}
\rightarrow Q$.

(iii) The simplicial sets $MC_{\bullet}(Hom_{\Sigma}(\overline{C},Q))$
and $Map_{\mathcal{P}}(P_{\infty},Q)$ are isomorphic.
\end{thm}

\begin{cor}
Let $\phi:P_{\infty}\rightarrow Q$ be a Maurer-Cartan element of $Hom_{\Sigma}(\overline{C},Q)$
and $Hom_{\Sigma}(\overline{C},Q)^{\phi}$ the corresponding twisted Lie algebra.

(1) For every integer $n\geq 0$, we have a bijection
\[
H^{-n}(Hom_{\Sigma}(\overline{C},Q)^{\phi})\cong \pi_{n+1}(Map_{\mathcal{P}}(P_{\infty},Q),\phi)
\]
which is an isomorphism of abelian groups for $n\geq 1$, and an isomorphism of groups for $n=0$
where $H^0(Hom_{\Sigma}(\overline{C},Q)^{\phi})$ is equipped with the group structure given by the Hausdorff-Campbell formula.

(2) When $\phi$ is a weak equivalence we obtain
\[
H^{-n}(Hom_{\Sigma}(\overline{C},Q)^{\phi})\cong \pi_{n+1}(hiso(P_{\infty},Q),\phi).
\]

(3) When $\phi$ is a weak equivalence and $Q=P_{\infty}$ we obtain
\[
H^{-n}(Hom_{\Sigma}(\overline{C},P_{\infty})^{\phi})\cong \pi_{n+1}(haut(P_{\infty}),\phi).
\]
\end{cor}
\begin{proof}
This corollary follows directly from Theorems 2.12 and 2.13.
Concerning the homotopy automorphisms, we just note that
\[
\pi_{n+1}(Map(A_{\infty},A_{\infty}),\phi)=\pi_{n+1}(Map(A_{\infty},A_{\infty})_{\phi},\phi)
\]
and $Map(A_{\infty},A_{\infty})_{\phi}=haut(A_{\infty})_{\phi}$ when $\phi$ is a weak equivalence because of the two-out-of-three property
(just draw the diagram of a homotopy between a weak equivalence and a map).
\end{proof}

Let us prove Theorem 2.13. We first need the following preliminary lemma:
\begin{lem}
Let
\[
\varphi_{\bullet}:g_{\bullet}\stackrel{\sim}{\rightarrow}h_{\bullet}
\]
be an isomorphism of simplicial dg Lie algebras.
Suppose these Lie algebras are filtered and that for every integer $r$, the map $\varphi_{\bullet}$ induces an isomorphism of simplicial dg Lie algebras
\[
F_r\varphi_{\bullet}:F_rg_{\bullet}\stackrel{\sim}{\rightarrow}F_rh_{\bullet}.
\]
Then its completion
\[
\hat{\varphi}_{\bullet}:\hat{g}_{\bullet}\stackrel{\sim}{\rightarrow}\hat{h}_{\bullet}
\]
is an isomorphism of simplicial complete dg Lie algebras.
\end{lem}
\begin{proof}
The faces and the degeneracies of $g_{\bullet}$ are compatible with its filtration and the same holds for $h_{\bullet}$.
Moreover, by assumption $\varphi_{\bullet}$ and $F_r\varphi_{\bullet}$ are isomorphisms,
so for every integer $r$ the induced map
\[
g_{\bullet}/F_rg_{\bullet}\stackrel{\cong}{\rightarrow}h_{\bullet}/F_rh_{\bullet}
\]
is also an isomorphism of simplicial algebras.
Simplicial objects in a category form a diagram category, so the limits are determined pointwise. This implies
that $\hat{g}_{\bullet}=lim_r g_{\bullet}/F_rg_{\bullet}$ as a simplicial dg Lie algebra and the same holds for
$h_{\bullet}$, so finally $\hat{\varphi}_{\bullet}$ is an isomorphism.
\end{proof}

\begin{proof}[Proof of Theorem 2.12]
Let us recall that for every properad $Q$, the complex $Hom_{\Sigma}(\overline{C},Q)$ forms a dg Lie
algebra. The bracket is defined by
\[
[f,g]=f\bullet g - (-1)^{|f||g|}g\bullet f
\]
where $g\bullet f$ is the convolution product of $g$ and $f$ obtained by using the infinitesimal coproduct of
the coproperad structure of $C$ and the infinitesimal product of the properad structure of $Q$.
The idea is the following : an element of $\overline{C}$ is represented by a directed graph with one vertex.
The infinitesimal coproduct expands it into a sum of graphs with two vertices, i.e $2$-levelled directed graphs
with one vertex on each level indexed by a certain element of $\overline{C}$. One then applies, for each such graph,
$f$ to the element of the first level and $g$ to the element of the second level, thus obtaining a sum
of $2$-levelled directed graphs of operations of $Q$. One finally applies the composition product of $Q$.
We refer the reader to \cite{DR} and \cite{MV2} for more details.
Its explicit formula is given by
\[
\overline{C}\stackrel{\Delta_{(1)}}{\rightarrow}\overline{C}\boxtimes_{(1)}\overline{C}\stackrel{f\boxtimes_{(1)}g}{\rightarrow}Q\boxtimes_{(1)}Q
\stackrel{\mu_{(1)}}{\rightarrow} Q
\]
where $\Delta_{(1)}$ is the infinitesimal coproduct of the coproperad $\mathcal{C}$ restricted to the augmentation ideal $\overline{C}$ and $\mu_{(1)}$ the infinitesimal
product of the properad $Q$.
This convolution product equips $Hom_{\Sigma}(\overline{C},Q)$
with a pre-Lie algebra structure. There is an associated Lie structure given by
\[
[f,g]=f\bullet g - (-1)^{|f||g|}g\bullet f.
\]

One can then show that this dg Lie algebra is filtered. There are several possible filtrations compatible with
the Lie bracket, in particular one given by the arity in \cite{DR} and one given by the differential of $P_{\infty}$
in \cite{MV2}. Moreover, $Hom_{\Sigma}(\overline{C},Q)$ is complete for both of these filtrations
(for the filtration given in \cite{MV2}, the reader can check that this filtration becomes $0$ from $r=3$).

Now let us consider the cochain complex isomorphism
\[
\varphi_{\bullet}:Hom_{\Sigma}(\overline{C},Q)\otimes\Omega_{\bullet}
\stackrel{\cong}{\rightarrow} Hom_{\Sigma}(\overline{C},Q\otimes_e\Omega_{\bullet})
\]
defined by $\varphi_{\bullet}(f\otimes \omega)=[c\in\overline{C}\mapsto f(c)\otimes_e\omega]$ for every $f\in Hom_{\Sigma}(\overline{C},Q)$ and $\omega\in\Omega_{\bullet}$.
We claim that this is actually an isomorphism of simplicial pre-Lie algebras, that is, the map $\varphi_{\bullet}$ commutes with the convolution product $\bullet$
and preserves the simplicial structure. This implies that $\varphi_{\bullet}$ is an isomorphism of simplicial dg Lie algebras for the Lie structure induced by
the convolution product. To see this, let $f\otimes\omega_1,g\otimes\omega_2\in Hom_{\Sigma}(\overline{C},Q)
\otimes\Omega_{\bullet}$ and $c\in\overline{C}$. By definition of the convolution $\bullet$,
the map $\varphi(f\otimes\omega_1)\bullet\varphi(f\otimes\omega_2)$ is evaluated on $c$ as follows.
First, we apply the infinitesimal coproduct $\Delta_{(1)}$ of the coproperad $C$ (restricted to the augmentation ideal
$\overline{C}$) to get
\[
\Delta_{(1)}(c)=\sum \pm Gr(c',c'')\in\mathcal{F}(\overline{C})^{(2)}
\]
which is a sum of connected $2$-levelled graphs $Gr(c',c'')$ with one vertex on each level. The upper one is indexed by $c'$ and the lower one by $c''$.
Then we apply $\varphi(f\otimes\omega_1)$ to the upper vertex and $\varphi(f\otimes\omega_1)$ to the lower vertex to get
\[
\sum \pm Gr(f(c')\otimes_e\omega_1,g(c'')\otimes_e\omega_2).
\]
Finally we apply the infinitesimal composition product $\mu_{(1)}^{Q\otimes_e\Omega_{\bullet}}$ to get
\begin{eqnarray*}
\mu_{(1)}^{Q\otimes_e\Omega_{\bullet}}(\sum \pm Gr(f(c')\otimes_e\omega_1,g(c'')\otimes_e\omega_2))
& = & \sum \pm \mu_{(1)}^Q(Gr(f(c'),f(c'')))\otimes\omega_1\omega_2 \\
& = & (f\bullet g)(c)\otimes\omega_1\omega_2 \\
& = & \varphi(f\otimes\omega_1\bullet g\otimes\omega_2)(c)
\end{eqnarray*}
by definition of the properad structure on $Q\otimes_e\Omega_{\bullet}$.

The simplicial structure induced by $\Omega_{\bullet}$ is obviously compatible with the filtrations,
since it is defined by the simplicial structure of $\Omega_{\bullet}$ and the filtration does not act on
$\Omega_{\bullet}$.
Moreover, the map $\varphi_{\bullet}$ induces an isomorphism at each stage of the filtrations.
Since $F_r(Hom_{\Sigma}(\overline{C},Q)\otimes\Omega_{\bullet})=F_rHom_{\Sigma}(\overline{C},Q)\otimes\Omega_{\bullet}$,
for any $f\otimes\omega\in F_rHom_{\Sigma}(\overline{C},Q)\otimes\Omega_{\bullet}$,
if $c\in F_{r-1}\overline{C}$ then $f(c)=0$ so
\begin{eqnarray*}
\varphi(f\otimes\omega)(x) & = & f(x)\otimes_e\omega \\
 & = & 0.
\end{eqnarray*}
By Lemma 2.15 we get an isomorphism between the completions
\[
\hat{\varphi}_{\bullet}:Hom_{\Sigma}(\overline{C},Q)\hat{\otimes}\Omega_{\bullet}
\stackrel{\cong}{\rightarrow} \hat{Hom_{\Sigma}(\overline{C},Q\otimes_e\Omega_{\bullet})}.
\]
The dg Lie algebra $Hom_{\Sigma}(\overline{C},Q\otimes_e\Omega_{\bullet})$ is complete, so there is a canonical isomorphism
$Hom_{\Sigma}(\overline{C},Q\otimes_e\Omega_{\bullet}) \cong \hat{Hom_{\Sigma}(\overline{C},Q\otimes_e\Omega_{\bullet})}$.
Given that
\begin{eqnarray*}
MC(Hom_{\Sigma}(\overline{C},Q\otimes\Omega_{\bullet})) & = & Mor_{\mathcal{P}}(P_{\infty},Q\otimes\Omega_{\bullet})\\
 & = & Map(P_{\infty},Q) \\
 \end{eqnarray*}
we finally have an isomorphism of simplicial sets
\begin{eqnarray*}
MC_{\bullet}(Hom_{\Sigma}(\overline{C},Q)) & = & MC(Hom_{\Sigma}(\overline{C},Q)\hat{\otimes}\Omega_{\bullet})\\
 & \cong & Map(P_{\infty},Q).
\end{eqnarray*}

\end{proof}

\begin{rem}
This gives alternatively a way to study the higher homotopy groups of moduli spaces via the derivations bicomplex,
for instance via homological spectral sequences arguments, or a way to study the homology groups of
the deformation complex via Bousfield-Kan spectral sequences arguments.
\end{rem}

\begin{example}
For $Q=End_X$ we get
\begin{eqnarray*}
\pi_{n+1}(P_{\infty}\{X\},\phi) & \cong & H^{-n}Hom_{\Sigma}(C,End_X)^{\phi} \\
 & \cong & H^{-n}Der_{\phi}(P_{\infty},End_X).
\end{eqnarray*}
\end{example}

\subsection{A more general case.}

The reader may have noticed that one can always provide a cofibrant resolution of $P$ formed by the cobar construction of
a certain coproperad (either the Koszul dual or the bar construction of $P$), hence getting a dg Lie deformation complex.
But in general, for properads which are not Koszul, the bar-cobar resolution is huge
and one could be interested in a smaller deformation complex. The price to pay for this size reduction is an
increased complexity of the $L_{\infty}$-structure, encoded by the differential of $P_{\infty}$.
For this, we now consider the more general situation when $P_{\infty}=(\mathcal{F}(s^{-1}C),\partial)$
is a quasi-free properad generated by the desuspension of a homotopy coproperad $C$. A homotopy coproperad is exactly
a $\Sigma$-biobject $C$ such that $\mathcal{F}(s^{-1}C)$ is equipped with a derivation $\partial$ of degree $-1$
such that $\partial^2=0$. Moreover, we suppose that $C$ is equipped with an increasing filtration
$\{C_r\}$ compatible with the differential, that is, for every $r$ we have
\[
\partial(C_r)\subset \mathcal{F}(s^{-1}C_r).
\]
This condition holds for minimal models in the sense of \cite{Mar2}.
A leading example is the properad of non unitary and non-counitary associative and coassociative bialgebras
$BiAss$, which is not Koszul but only homotopy Koszul in the sense of \cite{MV2}.
Explicit computations of certain terms of the differential of the minimal model $BiAss_{\infty}$ have been done in \cite{Mar2}. This non-quadratic and highly non trivial differential induces a fully-fledged filtered $L_{\infty}$-algebra structure on the deformation complex.
This deformation complex has been, in turn, be proven to be quasi-isomorphic to the well known Gerstenhaber-Schack
complex \cite{GS} used in deformation theory of bialgebras (see \cite{MV2} for a proof), proving that the
Gerstenhaber-Schack complex possesses a $L_{\infty}$-structure controlling the deformation theory of bialgebras
up to homotopy.

In \cite{MV2}, the $L_{\infty}$ algebras are supposed to be only filtered and one have to consider Maurer-Cartan elements
in their completion for this filtration. We explain here why these algebras are actually complete for the
appropriate filtration, so that we do not need to complete them.
\begin{thm} Let $P$ be a properad with minimal model $P_{\infty}=(\mathcal{F}(s^{-1}C),\partial)\stackrel{\sim}{\rightarrow}P$
for a certain homotopy coproperad $C$, and $Q$ be any properad.

(i) The total complex $Hom_{\Sigma}(\overline{C},Q)$ is a complete dg $L_{\infty}$ algebra.

(ii) The Maurer-Cartan elements of $Hom_{\Sigma}(\overline{C},Q)$ are exactly the properad morphisms $P_{\infty}
\rightarrow Q$.

(iii) The simplicial sets $MC_{\bullet}(Hom_{\Sigma}(\overline{C},Q))$
and $Map_{\mathcal{P}}(P_{\infty},Q)$ are isomorphic.
\end{thm}

\begin{cor}
Let $\phi:P_{\infty}\rightarrow Q$ be a Maurer-Cartan element of $Hom_{\Sigma}(\overline{C},Q)$
and $Hom_{\Sigma}(\overline{C},Q)^{\phi}$ the corresponding twisted Lie algebra.

(1) For every integer $n\geq 0$, we have a bijection
\[
H^{-n}(Hom_{\Sigma}(\overline{C},Q)^{\phi})\cong \pi_{n+1}(Map_{\mathcal{P}}(P_{\infty},Q),\phi)
\]
which is an isomorphism of abelian groups for $n\geq 1$, and an isomorphism of groups for $n=0$
where $H^0(Hom_{\Sigma}(\overline{C},Q)^{\phi})$ is equipped with the group structure given by the Hausdorff-Campbell formula.

(2) When $\phi$ is a weak equivalence we obtain
\[
H^{-n}(Hom_{\Sigma}(\overline{C},Q)^{\phi})\cong \pi_{n+1}(hiso(P_{\infty},Q),\phi).
\]

(3) When $\phi$ is a weak equivalence and $Q=P_{\infty}$ we obtain
\[
H^{-n}(Hom_{\Sigma}(\overline{C},P_{\infty})^{\phi})\cong \pi_{n+1}(haut(P_{\infty}),\phi).
\]
\end{cor}

To prove Theorem 2.17, let us first check that $Hom_{\Sigma}(\overline{C},Q)$ is a complete
$L_{\infty}$-algebra.
We use the decomposition proposed in the proof of Proposition 15 in \cite{Mar3}
\[
Hom_{\Sigma}(\overline{C},Q) = \prod_{l\geq 1} Hom_{\Sigma}(\overline{C},Q)_l
\]
which satisfies
\[
[Hom_{\Sigma}(\overline{C},Q)^k]_l=0
\]
for $k>l$. With such a decomposition we define a decreasing filtration
\[
F_rHom_{\Sigma}(\overline{C},Q) = \prod_{l\geq r} Hom_{\Sigma}(\overline{C},Q)_l.
\]
For every integer $k$ we have
\[
[Hom_{\Sigma}(\overline{C},Q)^k]\in\prod_{l\geq k}Hom_{\Sigma}(\overline{C},Q)_l=F_k Hom_{\Sigma}(\overline{C},Q)
\]
because $[Hom_{\Sigma}(\overline{C},Q)^k]_l=0$ for $k>l$.
Moreover, the isomorphisms
\[
Hom_{\Sigma}(\overline{C},Q)/F_rHom_{\Sigma}(\overline{C},Q)\cong \prod_{l=1}^{r-1}Hom_{\Sigma}(\overline{C},Q)_l
\]
imply that
\begin{eqnarray*}
lim_r Hom_{\Sigma}(\overline{C},Q)/F_rHom_{\Sigma}(\overline{C},Q) & \cong & \prod_{l\geq 1}Hom_{\Sigma}(\overline{C},Q)_l\\
 & = & Hom_{\Sigma}(\overline{C},Q).
\end{eqnarray*}
We know check that this filtration is compatible with the $L_{\infty}$ structure defined in \cite{MV2}.
We can actually rewrite our filtration as
\begin{eqnarray*}
F_rHom_{\Sigma}(\overline{C},Q)  & = & \prod_{l\geq r} Hom_{\Sigma}(\overline{C},Q)_l \\
 & \cong & \prod_{l\geq 1} Hom_{\Sigma}(\overline{C},Q)_l/\prod_{l\leq r-1} Hom_{\Sigma}(\overline{C},Q)_l \\
 & \cong & Hom_{\Sigma}(\overline{C},Q)/\oplus_{l\leq r-1}Hom_{\Sigma}(\overline{C},Q)_l \\
 & \cong & Hom_{\Sigma}(\overline{C},Q)/Hom_{\Sigma}(\overline{C}_{r-1},Q) \\
 & \cong & \{f\in Hom_{\Sigma}(\overline{C},Q)|f(\overline{C}_{r-1})=0\}
\end{eqnarray*}
where $\{\overline{C}_r\}$ is an exhaustive filtration of $\overline{C}$ defined by
\[
\overline{C}_r=\{\xi\in\overline{C}|\partial_{P_{\infty}}(\xi)\in \mathcal{F}(\overline{C})^{(\leq r)}\}.
\]
Now, given that for any integer $r$ we have $\partial_{P_{\infty}}(\overline{C}_{r-1})\in\mathcal{F}(\overline{C}_{r-1})$,
for any integer $n$ the inclusion
\[
\partial_{P_{\infty}}^{(n)}(\overline{C}_{r-1})\in\mathcal{F}(\overline{C}_{r-1})^{(n)}
\]
and the formula defining $Q^{(n)}$ imply that $Q^{(n)}(f_1\wedge...\wedge f_n)(\xi)=0$ for
$f_1\wedge...\wedge f_n\in\Lambda^nHom_{\Sigma}(\overline{C},Q)$ such that $f_1\in F_rHom_{\Sigma}(\overline{C},Q)$, $\xi\in\overline{C}_{r-1}$
and $n>r-1$. This means that
\[
Q^{(n)}(f_1\wedge...\wedge f_n)\in F_rHom_{\Sigma}(\overline{C},Q)
\]
for $n>r-1$.

The remaining part of the proof relies on
\begin{prop}
The cochain complex isomorphism
\[
\varphi:Hom_{\Sigma}(\overline{C},Q)\otimes\Omega_{\bullet}\stackrel{\cong}{\rightarrow}
Hom_{\Sigma}(\overline{C},Q\otimes_e\Omega_{\bullet})
\]
is an isomorphism of simplicial dg $L_{\infty}$-algebras.
\end{prop}
\begin{proof}
We prove a cdga isomorphism between the associated Chevalley-Eilenberg algebras.
This is equivalent to a strict isomorphism of $L_{\infty}$-algebras.
Let us consider
\[
C^*(\varphi):(\Lambda sHom_{\Sigma}(\overline{C},Q)\otimes\Omega_{\bullet},\tilde{Q})\rightarrow
(\Lambda sHom_{\Sigma}(\overline{C},Q\otimes_e\Omega_{\bullet}),Q_{\Omega})
\]
where $Q_{\Omega}$ is the coderivation defined in the proof of Theorem 5 (ii) in \cite{MV2}:
\[
Q_{\Omega}^{(1)}(f)=\partial_{Q\otimes_e\Omega_{\bullet}}\circ f - (-1)^{|f|}f\circ\partial_{P_{\infty}}^{(1)}
\]
and
\[
Q_{\Omega}^{(n)}(f_1\wedge...\wedge f_n)=- (-1)^{|f_1\wedge...\wedge f_n|}\mu_{Q\otimes_e\Omega_{\bullet}}(f_1\wedge...\wedge f_n\circ'\partial_{P_{\infty}}^{(n)})
\]
for $n>1$, where $\partial_{P_{\infty}}^{(n)}$ is the part of the differential of $P_{\infty}$ consisting of
graphs with $n$ vertices, that is
\[
\partial_{P_{\infty}}^{(n)}:s^{-1}\overline{C}\stackrel{\partial_{P_{\infty}}}{\rightarrow}\mathcal{F}(s^{-1}\overline{C})
\twoheadrightarrow \mathcal{F}(s^{-1}\overline{C})^{(n)}.
\]
In the notation $\mu_{Q\otimes_e\Omega_{\bullet}}(f_1\wedge...\wedge f_n\circ'\partial_{P_{\infty}}^{(n)})$, the $\circ'$ stands for a
composition defined by applying, for each graph with $n$ vertices appearing in $\partial_{P_{\infty}}^{(n)}$,
the map $f_i$ to the operation of $s^{-1}\overline{C}$ indexing the $i^{th}$ vertex. Then one applies the
composition product $\mu_{Q\otimes_e\Omega_{\bullet}}$ of the properad $Q\otimes_e\Omega_{\bullet}$.
The coderivation $\tilde{Q}$ is obtained by extending the $L_{\infty}$-algebra structure of
$Hom_{\Sigma}(\overline{C},Q)$ by the cdga $\Omega_{\bullet}$, that is
\[
\tilde{Q}^{(1)}(f\otimes_e \omega)=Q^{(1)}(f)\otimes_e\omega + (-1)^{|f|}f\otimes_e d(\omega)
\]
and
\[
\tilde{Q}^{(n)}(f_1\otimes_e\omega_1\wedge... \wedge f_n\otimes_e\omega_n)
=Q^{(n)}(f_1\wedge...\wedge f_n)\otimes\omega_1 ...\omega_n
\]
for $n>1$.
For $c\in\overline{C}$ and $f\otimes\omega\in Hom_{\Sigma}(\overline{C},Q)\otimes\Omega_{\bullet}$, we have
\begin{eqnarray*}
Q^{(1)}(\varphi(f\otimes\omega))(c) & = & \partial_{Q\otimes_e\Omega_{\bullet}}\circ\varphi(f\otimes\omega)(c)
-(-1)^{|f|}\varphi(f\otimes\omega)\circ\partial_{P_{\infty}}^{(1)}(c) \\
 & = & \partial_{Q\otimes_e\Omega_{\bullet}}(f(c)\otimes_e\omega)-(-1)^{|f|}f(\partial_{P_{\infty}}^{(1)}(c))\otimes_e
 \omega \\
 & = & \partial_Q(f(c))\otimes_e\omega + (-1)^{|f|}f(c)\otimes_e d(\omega)-(-1)^{|f|}f(\partial_{P_{\infty}}^{(1)}(c))\otimes_e
 \omega \\
 & = & ((\partial_Q\circ f)(c)-(-1)^{|f|}(f\circ\partial_{P_{\infty}}^{(1)})(c))\otimes_e\omega + (-1)^{|f|}f(c)\otimes_e d(\omega) \\
 &  = & \varphi(\tilde{Q}^{(1)}(f\otimes\omega))(c).
\end{eqnarray*}
For $n>1$, $c\in\overline{C}$ and $f_1\otimes\omega_1\wedge...\wedge f_n\otimes\omega_n\in
\Lambda^n sHom_{\Sigma}(\overline{C},Q)\otimes\Omega_{\bullet}$, we have
\begin{eqnarray*}
& & Q^{(n)}(C^*(\varphi)(f_1\otimes\omega_1\wedge...\wedge f_n\otimes\omega_n))(c) \\
& = &
-(-1)^{|f_1\otimes\omega_1\wedge...\wedge f_n\otimes\omega_n|}
\mu_{Q\otimes_e\Omega_{\bullet}}(Gr(C^*(\varphi)(f_1\otimes\omega_1\wedge...\wedge f_n\otimes\omega_n))\circ'
\partial_{P_{\infty}}^{(n)}(c)) \\
 & = & -(-1)^{|f_1\otimes\omega_1\wedge...\wedge f_n\otimes\omega_n|}
\mu_{Q\otimes_e\Omega_{\bullet}}(Gr(f_1(-)\otimes\omega_1,...,f_n(-)\otimes\omega_n)\circ'\partial_{P_{\infty}}^{(n)}(c)) \\
 & = & -(-1)^{|f_1\otimes\omega_1\wedge...\wedge f_n\otimes\omega_n|}
\mu_Q(Gr(f_1,...,f_n)\circ'\partial_{P_{\infty}}^{(n)}(c))\otimes_e\omega_1 ...\omega_n \\
 & = & Q^{(n)}(f_1,...,f_n)(c)\otimes_e\omega_1 ...\omega_n \\
 & = & C^*(\varphi)(\tilde{Q}^{(n)}(f_1\otimes\omega_1,...,f_n\otimes\omega_n))(c).
\end{eqnarray*}
\end{proof}
Moreover, for every integer $r$ the induced map $F_r\varphi$ is also an isomorphism.
By Lemma 2.15 we get an isomorphism between the completions and we conclude the proof.

\subsection{Higher homotopy groups of mapping spaces of algebras over operads}

Let $P$ be an operad and $C$ a cooperad equipped with a twisting morphism $C\rightarrow P$,
such that $\Omega(C)\stackrel{\sim}{\rightarrow} P$ is a cofibrant resolution of $P$,
where $\Omega(-)$ is the operadic cobar construction (see \cite{LV} for instance).
One can always produce such a resolution by taking for $C$ the bar contruction on $P$, or the Koszul dual of $P$ if $P$ is Koszul.
Let $A$ be a $P$-algebra. There is a quasi-free $P$-algebra $A_{\infty}=((P\circ C)(A),\partial)$ over a quasi-free
$C$-coalgebra $C(A)$ which gives a cofibrant resolution of $A$. We refer the reader to \cite{Fre1} for a detailed construction of such $C$-coalgebras and cofibrant
resolutions.

We consider a $\mathbb{Z}$-graded cochain complex $Hom(C(A),B)$ with the degree defined by
\[
f\in Hom(C(A),B)_k\Leftrightarrow \forall n, f(C(A)_n)\subset B_{n+k}
\]
and the differential defined by
\[
\delta(f)=d_B\circ f - (-1)^{|f|}f\circ (d_C+d_A)
\]
where $d_B$ is the differential of $B$, $|f|$ the degree of $f$,
and $d_C+d_A$ the differential of the total complex of $C$, where $d_A$ comes
from the differential of $A$ and $d_C$ is the only derivation extending the infinitesimal
coproduct of the cooperad $C$.
This is a complete dg Lie algebra whose Maurer-Cartan elements consist in the set of twisting morphisms $Tw(C(A),B)$.
Recall that we have bijections
\[
Mor_{P-Alg}(A_{\infty},B) \cong Tw(C(A),B) \cong Mor_{C-Coalg}(C(A),C(B)).
\]

The principal theorem of this section is the following:
\begin{thm} Let $P$ be an operad and $A,B$ be any $P$-algebras.

(i) The total complex $Hom_{dg}(C(A),B)$ is a complete dg Lie algebra.

(ii) The Maurer-Cartan elements of $Hom_{dg}(C(A),B)$ are exactly the morphisms of $P$-algebras $A_{\infty}\rightarrow B$.

(iii) There exists an isomorphism between $MC_{\bullet}(Hom_{dg}(C(A),B))$
and $Map(A_{\infty},B)$.
\end{thm}

\begin{cor}
Let $\phi:A_{\infty}\rightarrow B$ be a Maurer-Cartan element of $Hom_{dg}(C(A),B)$
and $Hom_{dg}(C(A),B)^{\phi}$ the corresponding twisted Lie algebra.

(1) For every integer $n\geq 0$, we have a bijection
\[
H^{-n}(Hom_{dg}(C(A),B)^{\phi})\cong \pi_{n+1}(Map(A_{\infty},B),\phi)
\]
which is an isomorphism of abelian groups for $n\geq 1$, and an isomorphism of groups for $n=0$
where $H^0(Hom_{dg}(C(A),B)^{\phi})$ is equipped with the group structure given by the Hausdorff-Campbell formula.

(2) When $\phi$ is a weak equivalence we obtain
\[
H^{-n}(Hom_{dg}(C(A),B)^{\phi})\cong \pi_{n+1}(hiso(A_{\infty},B),\phi).
\]

(3) When $\phi$ is a weak equivalence and $B=A_{\infty}$ we obtain
\[
H^{-n}(Hom_{dg}(C(A),A_{\infty})^{\phi})\cong \pi_{n+1}(haut(A_{\infty}),\phi).
\]
\end{cor}
\begin{proof}
This corollary follows directly from Theorems 2.12 and 2.29.
Concerning the homotopy automorphisms, we just note that
\[
\pi_{n+1}(Map(A_{\infty},A_{\infty}),\phi)=\pi_{n+1}(Map(A_{\infty},A_{\infty})_{\phi},\phi)
\]
and $Map(A_{\infty},A_{\infty})_{\phi}=haut(A_{\infty})_{\phi}$ when $\phi$ is a weak equivalence because of the two-out-of-three property
(just draw the diagram of a homotopy between a weak equivalence and a map).
\end{proof}

The proof is completely equivalent to the properadic case, once we have checked that the tensor product
$(-)\otimes\Omega_{\bullet}$ defines a functorial simplicial resolution in the category of $P$-algebras.

\subsection{Link with André-Quillen homology}

When the operad $P$ is Koszul, we choose $C=P^{\textrm{!`}}$ the Koszul dual of $P$, and the Lie bracket of
the complete dg Lie algebra $Hom_{dg}(C(A),B)$ is given by Theorem 4.1.1 of \cite{Mil}.
In this case there is an isomorphism of cochain complexes
\[
Hom_{dg}(C(A),B)^{\phi}\cong Der_{\phi}(A_{\infty},B)=Der_{A_{\infty}}(A_{\infty},B)
\]
for any morphism of $P$-algebras $A_{\infty}\rightarrow B$,
where the derivations are defined for $A_{\infty}$-modules with the $A_{\infty}$-module structure
on $B$ induced by $\phi$.
This implies that for every $n\in\mathbb{N}$,
\begin{eqnarray*}
H^{-n}Hom_{dg}(A_{\infty},B)^{\phi} & \cong & H^{-n}Hom_{dg}(A_{\infty},B)^{\phi} \\
 & = & H^{-n}Der_{A_{\infty}}(A_{\infty},B)\\
 & = & H^n_{AQ}(A_{\infty},B)
\end{eqnarray*}
where $H^n_{AQ}$ is the André-Quillen cohomology.
For instance, we obtain
\[
H^n_{AQ}(A_{\infty},A_{\infty})\cong \pi_{n+1}(haut(A_{\infty}),\phi)
\]
where the right hand $A_{\infty}$ is seen as an $A_{\infty}$-module via $\phi$.

\begin{rem}
According to Theorem 8.3.1 of \cite{Mil}, we know that this cohomology can be expressed as an $Ext$ functor:
\[
H^n_{AQ}(A_{\infty},A_{\infty}) \cong Ext^n_{A_{\infty}\otimes^{P_{\infty}}\mathbb{K}}(\Omega_{P_{\infty}}A_{\infty},
A_{\infty}).
\]
\end{rem}

\subsection{Rational homotopy groups}

Let us suppose here that $\mathbb{K}=\mathbb{Q}$.
By Corollary 1.2 of \cite{Ber}, for any complete $L_{\infty}$-algebra $g$ of finite type,
the cdga $C^*(g^{\tau}_{\geq 0})$ is a Sullivan model of $MC_{\bullet}(g)_{\tau}$ (the connected component
of $\tau$ in $MC_{\bullet}(g)$). This implies that $C^*(g^{\tau}_{\geq 0})\sim A_{PL}(MC_{\bullet}(g)_{\tau})$
in $CDGA_{\mathbb{Q}}$, hence a rational equivalence
\[
MC_{\bullet}(g)_{\tau} \simeq_{\mathbb{Q}}\langle A_{PL}(MC_{\bullet}(g)_{\tau}) \rangle
\sim_{\mathbb{Q}} \langle C^*(g^{\tau}_{\geq 0}) \rangle
\]
and finally
\[
MC_{\bullet}(g) \simeq_{\mathbb{Q}} \coprod_{[\tau]\in\mathcal{MC}(g)} Mor_{CDGA_{\mathbb{Q}}}
(C^*(g^{\tau}_{\geq 0}),\Omega_{\bullet})
\]
Applying this to $g=Hom_{\Sigma}(C,Q)$ when it is of finite type and $\varphi:P_{\infty}\rightarrow Q$,
we get in particular a group isomorphism
\[
H^*_{CE}(g^{\varphi}_{\leq 0})\cong \pi_*(Map(P_{\infty},Q),\varphi)\otimes\mathbb{Q}
\]
for every $*>0$.
In the setting of algebras over an operad, for $g=Hom_{dg}(C(A),B)$ and $\varphi:A_{\infty}\rightarrow B$ we get
\[
H^*_{CE}(g^{\varphi}_{\leq 0})\cong \pi_*(Map(A_{\infty},B),\varphi)\otimes\mathbb{Q}
\]
for every $*>0$.

\subsection{Long exact sequences for cohomology theories of bialgebras}

An interesting feature of this relationship between homotopical and algebraic sides of deformation theory
is that one can use properties on one side to get new properties on the other side.
An example is the following: any fiber sequence of properads induces a long exact sequence of homotopy groups
of moduli spaces for a given complex $X$,
which in turn gives a long exact sequence of cohomology groups relating the corresponding cohomology theories
on $X$.

\subsubsection{Quasi-free properads}

The reference we use is \cite{Fre1}, to which we refer the reader for a more detailed study.
Though it is written for dg operads, all the notions
and results we need here have their equivalent in the properadic setting.
Let $M$ be a $\Sigma$-biobject in cochain complexes. The free properad $(\mathcal{F}(M),d_M)$
endowed with the differential induced by the one of $M$ is the free dg properad on $M$.
Let $\alpha:M\rightarrow Q$ be a morphism of $\Sigma$-biobjects, where the target $Q$
is a properad. It can be extended uniquely into a derivation
\[
\partial_{\alpha}:\mathcal{F}(M)\rightarrow Q
\]
by the derivation rules. Conversely, any derivation is obtained by this way.
Now let us consider a free properad $(\mathcal{F}(M),d_M)$. The quasi-free properads
are obtained by twisting the differential of free properads. It means that we endow
$\mathcal{F}(M)$ with a new differential $(\mathcal{F}(M),\partial)$ of the form
$\partial=d_M+\partial_{\alpha}$, where $\partial_{\alpha}:\mathcal{F}(M)\rightarrow\mathcal{F}(M)$
is a special derivation of degree $1$ called a twisting morphism. It has to satisfy the equation of
twisting morphisms
\begin{eqnarray*}
(d_M + \partial_{\alpha})^2 = 0 & = & d_M^2 + d_M\circ\partial_{\alpha} + \partial_{\alpha}\circ d_M + \partial_{\alpha}^2 \\
 & = & \partial_{\alpha}^2 + d_M(\partial_{\alpha}).
\end{eqnarray*}
where we use the notation
\[
d_M(f)=d_M\circ f - (-1)^{deg(f)}f\circ d_M.
\]
The proof of proposition 1.4.7 of \cite{Fre1} can be adapted readily to give this reformulation in the properadic context:
\begin{prop}
There is a bijective correspondence between the properad morphisms
\[
\varphi_f:(\mathcal{F}(M),\partial)\rightarrow Q
\]
with a quasi-free source and the degree $0$ homomorphisms of $\Sigma$-biobjects
\[
f\in(Hom_{\Sigma}(M,Q),\delta)
\]
such that $\delta(f)=\varphi_f\circ\alpha$. Here $\delta$ is the usual differential on the external hom
of $\Sigma$-biobjects, given by
\[
\delta(f)=d_Q\circ f - (-1)^{deg(f)}f\circ d_M.
\]
\end{prop}

Cofibrations of quasi-free properads induced by morphisms of $\Sigma$-biobjects at the level of generators
satisfies the following characterization, which is a properadic analogue of proposition 1.4.13 in \cite{Fre2}:
\begin{prop}
Suppose we have a morphism of reduced $\Sigma$-biobjects $f:M\rightarrow N$ which determines a morphism
of quasi-free properads
\[
\mathcal{F}(f):(\mathcal{F}(M),\partial_{\theta})\rightarrow (\mathcal{F}(M),\partial_{\psi}),
\]
where the differentials come from derivations $\theta$ and $\psi$ satisfying $\theta(M)\subset\mathcal{F}^{(\geq 2)}(M)$
and $\psi(N)\subset\mathcal{F}^{(\geq 2)}(N)$ (i.e the differentials are decomposable). The morphism
$\mathcal{F}(f)$ is a (acyclic) cofibration of $\mathcal{P}$
whenever $f$ is a (acyclic) cofibration of $(Ch_{\mathbb{K}}^{\mathbb{S}})$.
\end{prop}

\subsubsection{Cofiber sequences}

Let $i:O_{\infty}\rightarrow P_{\infty}$ be any fixed prop morphism of cofibrant properads (for the projective
model structure).
The map $i$ admits a factorization $O_{\infty}\stackrel{\tilde{i}}{\rightarrowtail}
\tilde{P_{\infty}}\stackrel{\sim}{\twoheadrightarrow_p}P_{\infty}$ into a cofibration $\tilde{i}$
followed by an acyclic fibration $p$. The pushout of $\tilde{i}$ along the final morphism of $O_{\infty}$
gives the cofiber of $\tilde{i}$, which is a model for the homotopy cofiber of $i$.
Let us note $(P,O)_{\infty}$ the homotopy cofiber of $i$, which represents the operations
of $P_{\infty}$ which are not in $O_{\infty}$. We know that $i$ induces a map of simplicial sets
$i^*:P_{\infty}\{X\}\rightarrow O_{\infty}\{X\}$ which factors through
$P_{\infty}\{X\}\stackrel{\sim}{\rightarrow_{p^*}}\tilde{P_{\infty}}\{X\}\stackrel{\tilde{i}^*}{\twoheadrightarrow}
O_{\infty}\{X\}$. One has the chain of homotopy equivalences
\[
hofib(i^*)\simeq fib(\tilde{i}^*)\simeq cofib(\tilde{i})\{X\}\simeq hocofib(i)\{X\}=(P,O)_{\infty}\{X\}
\]
relating the homotopy fiber of $i^*$ with the moduli space of $(P,O)_{\infty}$-algebra structures on $X$.
We get a long exact sequence of homotopy groups inducing a long exact sequence of cohomology groups.

\section{Deformation functors and quotient stacks}

\subsection{Deformation functors as connected components of moduli spaces}

Let $P$ be a properad admitting a cofibrant resolution of the form
$P_{\infty}=\Omega(C)\stackrel{\sim}{\rightarrow}P$ where $C$ is a dg coproperad. We have seen before
that $Hom_{\Sigma}(C,Q)$ is a complete dg Lie algebra with Maurer-Cartan elements the properad
morphisms $P_{\infty}\rightarrow Q$. Hence there is an associated deformation functor
\[
Def_{Hom_{\Sigma}(C,Q)}:Art_{\mathbb{K}}\rightarrow Set
\]
which associates to any local artinian commutative $\mathbb{K}$-algebra $R$ with residue field $\mathbb{K}$
the set $\mathcal{MC}(Hom_{\Sigma}(C,Q)\otimes m_R)$ of equivalences classes of $R$-extended
properad morphisms $P_{\infty}\otimes R\rightarrow Q\otimes R$.
If we fix a morphism $\varphi:P_{\infty}\rightarrow Q$, then the deformation functor
\[
Def_{Hom_{\Sigma}(C,Q)^{\phi}}:Art_{\mathbb{K}}\rightarrow Set
\]
associates to any $R$ the equivalence classes of $R$-deformations of $\varphi$.
The global tangent space of this deformation functor, in the sense of \cite{Sch}, is
\[
t_{Def_{Hom_{\Sigma}(C,Q)^{\phi}}}=H^{-1}(Hom_{\Sigma}(C,Q)^{\phi})
\]
(see \cite{Man}).

Our main observation here is that we can replace $Def_{Hom_{\Sigma}(C,Q)}$ by a more convenient
and naturally isomorphic functor
\[
P_{\infty}\{Q\}(-):Art_{\mathbb{K}}\rightarrow Set
\]
defined by
\[
P_{\infty}\{Q\}(R)=\pi_0P_{\infty}\{Q\otimes R\}.
\]
To prove this, let us just point out that $P_{\infty}\otimes R$ is still a properad because
of the commutative algebra structure on $R$, and that $P_{\infty}\otimes R=\Omega(C\otimes R)$,
which implies that $P_{\infty}\otimes R$ is still cofibrant and that
\begin{eqnarray*}
\pi_0P_{\infty}\otimes R\{Q\otimes R\} & \cong & \mathcal{MC}(Hom_{\Sigma}(C\otimes R,Q\otimes R))\\
 & \cong & \mathcal{MC}(Hom_{\Sigma}(C,Q)\otimes R).
\end{eqnarray*}
Note that in the first line of this computation, the external hom $Hom_{\Sigma}$
is taken in the category of $\Sigma$-biobjects in dg $R$-modules,
and in the second line it is the external hom of $\Sigma$-biobjects in cochain complexes.
The fact that the tensor product with $R$ commutes with this external hom follows from the two following arguments. First, the fact that $R$ is finitely generated (since it is artinian)
implies that it commutes with the dg hom of cochain complexes. Second, the external hom of $\Sigma$-biobjects is by definition a product of invariants in such dg homs
under the action of symmetric groups, and the action of symmetric groups on $R$ is trivial in the external tensor product of a $\Sigma$-biobject with $R$,
so this commutation holds for the external hom of $\Sigma$-biobjects as well.
Then we just apply Theorem 2.13.
In the special case of $Q=End_X$, we have
\begin{eqnarray*}
End_X\otimes R(m,n) & = & Hom_{dg}(X^{\otimes m},X^{\otimes n})\otimes R \\
 & \cong & Hom_{dg}(X^{\otimes m}\otimes R,X^{\otimes n}\otimes R) \\
 & \cong & Hom_{dg}((X\otimes_{\mathbb{K}}R)^{\otimes_R m},(X\otimes_{\mathbb{K}}R)^{\otimes_R n})\\
 & = & End_{X\otimes R}(m,n)
\end{eqnarray*}
where the last line is the endomorphism prop of $X\otimes R$ as a dg $R$-module.
Consequently, we get
\[
\pi_0P_{\infty}\otimes R\{X\otimes R\}\cong Def_{Hom_{\Sigma}(\overline{C},End_X)}(R)
\]
(the connected components of a mapping space of properads in dg $R$-modules).

\begin{thm}
For every artinian algebra $R$ and any morphism $\varphi:P_{\infty}\rightarrow Q$, there are group isomorphisms
\[
\pi_{*+1}(Map_{Prop(R-Mod)}(P_{\infty}\otimes_e R,Q\otimes_e R),\varphi\otimes_e id_R)\cong H^{-*}(Hom_{\Sigma}(\overline{C},Q)^{\varphi}\otimes R).
\]
\end{thm}
This means that the homotopical $R$-deformations of $\varphi$ correspond to its algebraic $R$-deformations.

This theorem is a corollary of Theorem 2.13 once we have checked the following lemmas:
\begin{lem}
Let $A$ be a cdga and $g$ be a dg Lie algebra. Let $\varphi$ be a Maurer-Cartan element of $g$,
then $\varphi\otimes 1_A$ is a Maurer-Cartan element of $g\otimes A$ and $(g\otimes A)^{\varphi\otimes 1_A}
=g^{\varphi}\otimes A$ as a dg Lie algebra.
\end{lem}
\begin{proof}
A tensor $x\otimes a\in g\otimes A$ is a Maurer-Cartan element if and only if
\[
d_{g\otimes A}(x\otimes a)+\frac{1}{2}[x\otimes a,x\otimes a]=0,
\]
that is,
\[
d_g(x)\otimes a +(-1)^{|x|}x\otimes d_A(a)+\frac{1}{2}[x,x]\otimes a.a = 0.
\]
For $a=1_A$ we get
\begin{eqnarray*}
d_{g\otimes A}(x\otimes 1_A)+\frac{1}{2}[x\otimes 1_A,x\otimes 1_A] & = &
d_g(x)\otimes 1_A +(-1)^{|x|}x\otimes 0+\frac{1}{2}[x,x]\otimes 1_A \\
 & = & (d_g(x)+ \frac{1}{2}[x,x])\otimes 1_A
\end{eqnarray*}
since $1_A.1_A=1_A$ and $d_A(1_A)=0$.
Hence $x\in MC(g)$ if and only if $x\otimes 1_A\in MC(g\otimes A)$.

Let $\varphi$ be a Maurer-Cartan element of $g$. There is, by definition, an equality
\[
g^{\varphi}\otimes A=(g\otimes A)^{\varphi\otimes 1_A}
\]
as graded Lie algebras. It remains to check that the differentials are the same:
for every $x\otimes a\in g\otimes A$,
\begin{eqnarray*}
d_{(g\otimes A)^{\varphi\otimes 1_A}}(x\otimes a) & = & d_{g\otimes A}(x\otimes a) + [\varphi\otimes 1_A,
x\otimes a] \\
 & = & d_g(x)\otimes a + (-1)^{|x|}x\otimes d_A(a) + [\varphi,x]\otimes 1_A.a \\
 & = & (d_g(x)\otimes a + [\varphi,x]\otimes a) + (-1)^{|x|}x\otimes d_A(a) \\
 & = & d_{g^{\varphi}}(x)\otimes a + (-1)^{|x|}x\otimes d_A(a) \\
 & = & d_{g^{\varphi}\otimes A}(x\otimes a).
\end{eqnarray*}
\end{proof}

\begin{lem}
Let $X$ and $Y$ be two cochain complexes and $A$ be a cdga. The adjunction isomorphism
\[
Mor_{A-Mod}(X\otimes A,Y\otimes A)\stackrel{\cong}{\rightarrow} Mor_{Ch_{\mathbb{K}}}(X,Y\otimes A)
\]
sends every map of the form $f\otimes id_A$ to the map $x\mapsto f(x)\otimes 1_A$.
\end{lem}
\begin{proof}
The left adjoint is the tensor product $-\otimes A$ and the right adjoint is the forgetful functor $U$.
By definition of the adjonction isomorphism,
the morphism $f\otimes id_A$ is sent to $U(f\otimes id_A)\circ\eta(X)$ where $\eta$ is the unit of this adjunction,
defined by $\eta(X)(x)=x\otimes 1_A$ for every $X\in Ch_{\mathbb{K}}$ and $x\in X$.
\end{proof}
We deduce that for any artinian cdga $A$, the composed isomorphism
\[
Hom_{A-Mod}(X\otimes A,Y\otimes A)\stackrel{\cong}{\rightarrow} Hom_{Ch_{\mathbb{K}}}(X,Y)\otimes A,
\]
where $Hom$ stands for the differential graded hom bifunctors, respectively for $A$-modules and cochain complexes,
sends any morphism of the form $f\otimes id_A\in Mor_{A-Mod}(X\otimes A,Y\otimes A)$ to $f\otimes 1_A$.
As a corollary, for $\varphi:P_{\infty}\rightarrow Q$, we get the following isomorphisms of dg Lie algebras:
\begin{eqnarray*}
Hom_{A-Mod^{\mathbb{S}}}(\overline{C}\otimes_e A,Q\otimes_e A)^{\varphi\otimes id_A} & \cong &
Hom_{\Sigma}(\overline{C},Q\otimes_e A)^{[x\mapsto \varphi(x)\otimes_e 1_A]} \\
 & \cong & (Hom_{\Sigma}(\overline{C},Q))^{\varphi\otimes 1_A} \\
 & \cong & Hom_{\Sigma}(\overline{C},Q)^{\varphi}\otimes A.
\end{eqnarray*}

Let $A$ be an artinian cdga. The category of $A$-modules is a symmetric monoidal category for the tensor product
defined by the coequalizer
\[
A\otimes M \otimes N \rightrightarrows M\otimes N\rightarrow M\otimes_A N
\]
for any $A$-modules $M$ and $N$, where the two arrows correspond to the $A$-module structures of $M$ and $N$.
There is also a Quillen adjunction defining a model category structure on $A$-modules:
\[
-\otimes A:Ch_{\mathbb{K}}\rightleftarrows A-Mod:U
\]
where $U$ is the forgetful functor. Weak equivalences and fibrations of $A$-modules  are quasi-isomorphisms
and surjections of cochain complexes. The model and tensor structures are compatible so that $A-Mod$ forms
a cofibrantly generated monoidal model category (\cite{SS} Theorem 4.1).

The left adjoint $-\otimes A$ obviously respects the monoidal structure, but this is not the case for $U$.
The forgetful functor is only lax monoidal, with natural maps $U(M)\otimes U(N)\rightarrow U(M\otimes_A N)$
provided by the map $M\otimes N\rightarrow M\otimes_A N$ in the coequalizer defining the tensor product
of $A$-modules. However, an adjunction with a strong symmetric monoidal left adjoint and a lax monoidal
right adjoint is sufficient to induce an adjunction at the level of props
\[
-\otimes A:Prop(Ch_{\mathbb{K}})\rightleftarrows Prop(A-Mod):U.
\]

Both categories of props possess a cofibrantly generated model category structure with fibrations
and weak equivalences defined by forgetting the prop structure, i.e componentwise.
In particular, the forgetful functor $U:Prop(A-Mod)\rightarrow Prop(Ch_{\mathbb{K}})$ preserves fibrations
and weak equivalences, so finally:
\begin{lem}
The componentwise tensor product with $A$ defines a Quillen adjunction
\[
-\otimes_e A:Prop(Ch_{\mathbb{K}})\rightleftarrows Prop(A-Mod):U
\]
where $\otimes_e$ stands for the external tensor product (the componentwise tensor product with a cochain complex).
\end{lem}
This Quillen adjunction gives an adjunction relation at the level of homotopy mapping spaces.
For any cofibrant dg prop $P_{\infty}$ and any dg prop $Q$, we get an isomorphism of simplicial sets
\[
Map_{Prop(A-Mod)}(P_{\infty}\otimes_e A,Q\otimes_e A) \cong Map_{Prop(Ch_{\mathbb{K}})}(P_{\infty},Q\otimes_e A).
\]
In particular, for every $\varphi\otimes_e id_A\in Mor_{Prop(A-Mod)}(P_{\infty}\otimes_e A,Q\otimes_e A)$,
we get group isomorphisms
\[
\pi_*(Map_{Prop(A-Mod)}(P_{\infty}\otimes_e A,Q\otimes_e A),\varphi\otimes_e id_A)\cong
\pi_*(Map_{Prop(Ch_{\mathbb{K}})}(P_{\infty},Q\otimes_e A),[x\mapsto \varphi(x)\otimes 1_A]).
\]
we also have a bijection between the connected components, so we can replace the deformation functor
\[
Def_{Hom_{\Sigma}(\overline{C},End_X)}
\]
by the naturally isomorphic deformation functor
\[
P_{\infty}\{X\}(R)=\pi_0P_{\infty}\{X\otimes R\}.
\]

We can do the same constructions as above with the dg Lie algebra $Hom_{dg}(C(A),B)$ with
$A,B$ being algebras over a dg operad $P$ with cofibrant resolution $\Omega(C)$.

\subsection{Quotient stacks}

These moduli spaces also admit, consequently, an algebraic geometry interpretation.
Indeed, the deformation functor
\[
Def_{Hom_{\Sigma}(C,Q)}:Art_{\mathbb{K}}\rightarrow Set
\]
extends to a pseudo-functor of groupoids
\[
\underline{Def}_{Hom_{\Sigma}(C,Q)}:Alg_{\mathbb{K}}\rightarrow Grpd
\]
where $Alg_{\mathbb{K}}$ is the category of commutative $\mathbb{K}$-algebras and $Grpd$
the $2$-category of groupoids. This pseudo-functor is defined by sending any algebra $A$
to the Deligne groupoid of $Hom_{\Sigma}(C,Q)\otimes A$, that is, the groupoid whose objects are Maurer-Cartan
elements of $Hom_{\Sigma}(C,Q)\otimes A$ and morphisms are given by the action of the gauge group
$exp(Hom_{\Sigma}(C,Q)^0\otimes A)$. Such a pseudo-functor forms actually a prestack, whose stackification
gives the quotient stack
\[
[MC(Hom_{\Sigma}(C,Q))/exp(Hom_{\Sigma}(C,Q)^0)]
\]
of the Maurer-Cartan scheme $MC(Hom_{\Sigma}(C,Q))$ by the action of the prounipotent algebraic group $exp(Hom_{\Sigma}(C,Q)^0)]$.
It turns out that the $0^{th}$ cohomology group of the tangent complex of such a stack, encoding equivalences classes of infinitesimal deformations
of a $\mathbb{K}$-point of this stack, is exactly
\[
t_{Def_{Hom_{\Sigma}(C,Q)^{\phi}}}=H^{-1}(Hom_{\Sigma}(C,Q)^{\phi}).
\]
We refer the reader to \cite{Yal3} for a proof of these results.
However, this geometric structure does not capture the whole deformation theory of the points.
For this, the next section will develop such a geometric interpretation in the context of homotopical algebraic geometry.

\section{Moduli spaces as higher stacks}

In order to get a geometrical interpretation of the whole deformation theory of algebras over properads,
we work in the setting of homotopical algebraic geometry as constructed by Toen-Vezzosi in \cite{TV1}
and \cite{TV2}. Our purpose is two-fold: firstly, to construct higher stacks out of moduli spaces.
Secondly, to prove that the tangent spaces of this stack give precisely the cohomology theory of these algebras.
These two goals lead us to:
\begin{thm}
(1) Let $P_{\infty}=\Omega(C)\stackrel{\sim}{\rightarrow} P$ be a cofibrant resolution of a dg properad $P$
and $Q$ be any dg properad such that each $Q(m,n)$ is a bounded complex of finite dimension in each degree.
The functor
\[
\underline{Map}(P_{\infty},Q):A\in CDGA_{\mathbb{K}}\mapsto Map_{Prop}(P_{\infty},Q\otimes A)
\]
is a representable stack in the setting of complicial algebraic geometry of \cite{TV2}.

(2) Let $P_{\infty}=\Omega(C)\stackrel{\sim}{\rightarrow} P$ be a cofibrant resolution of a dg properad $P$ in non positively graded cochain complexes,
and $Q$ be any properad such that each $Q(m,n)$ is a finite dimensional vector space.
The functor
\[
\underline{Map}(P_{\infty},Q):A\in CDGA_{\mathbb{K}}\mapsto Map_{Prop}(P_{\infty},Q\otimes A)
\]
is a representable stack in the setting of derived algebraic geometry of \cite{TV2}, that is, an affine derived scheme.

(3) The homology groups of the tangent complex of $\underline{Map}(P_{\infty},Q)$ at the point
$\varphi:P_{\infty}\rightarrow Q$ are isomorphic to the cohomology groups of the deformation complex
$Der_{\varphi}(P_{\infty},Q)$.
\end{thm}
\begin{cor}
(1) Let $X$ be a bounded complex of finite dimension in each degree.
The moduli space functor
\[
\underline{P_{\infty}\{X\}}:A\in CDGA_{\mathbb{K}}\mapsto Map_{Prop}(P_{\infty},End_{X\otimes A})
\]
is a representable stack in complicial algebraic geometry, whose tangent complexes encode the cohomology theory of $P_{\infty}$-algebra structures on $X$.

(2) The same statement holds in derived algebraic geometry if $X$ is a finite dimensional vector space.
\end{cor}
The second part of Corollary 4.2 follows from the second part of Theorem 4.1 because if $X$ is a finite dimensional vector space, then so is $End_X(m,n)$ for any integers $m,n$.

\subsection{Hotomopical algebraic geometry}

Homotopical algebraic geometry has been introduced in \cite{TV1} and \cite{TV2} as a unified way of doing algebraic
geometry in various homotopical contexts, including in particular usual algebraic geometry, derived
algebraic geometry, complicial algebraic geometry and algebraic geometry over ring spectra.
Derived algebraic geometry over ring spectra has been developed in parallel by Lurie in his series of papers \cite{Lur2}.

For any model site $(M,\tau)$, that is, a model category $M$ equipped with a pretopology $\tau$, there is a notion
of \emph{stacks} over $M$ \cite{TV1}. We briefly describe this general construction.
Let $SPr(M)$ be the category of simplicial  presheaves over $M$, that is,
of functors from $M^{op}$ to the category of simplicial sets $sSet$. We assume that $SPr(M)$ is endowed with
the projective model category structure (i.e pointwise weak equivalences and fibrations).
The \emph{category of prestacks} $M^{\wedge}$ is a left Bousfield localization of this projective structure along the
weak equivalences of $M$, that is, the natural transformations $h_u,u\in wM$ where
\[
h:M\rightarrow Pr(M)\hookrightarrow SPr(M)
\]
is the constant Yoneda embedding.
The \emph{model category of stacks} $M^{\sim,\tau}$ is then obtained as a left Bousfield localization of the model category
of prestacks $M^{\wedge}$ along the \emph{homotopy $\tau$-hypercovers}  \cite{TV1}. These are exactly the equivalences of
homotopy sheaves $\pi_0(-)$ over $(M,\tau)$ and $\pi_i(-),i>0$ over the comma model sites $(M/x,\tau)$ (where $x$ ranges over the fibrant objects of $M$), see \cite{TV1} Theorem 4.6.1.
This Bousfield localization of  $M^{\wedge}$ comes with a Quillen adjunction
\[
Id:M^{\wedge}\rightleftarrows M^{\sim,\tau}:Id
\]
inducing an adjunction between the homotopy categories
\[
\mathbb{L}Id:Ho(M^{\wedge})\rightleftarrows Ho(M^{\sim,\tau}):\mathbb{R}Id.
\]
\emph{Stacks} are then defined as simplicial presheaves whose image in $Ho(M^{\wedge})$ belongs to the essential
image of $\mathbb{R}Id$ (see \cite{TV2} Definitionb 1.3.1.2). This means that they are simplicial presheaves preserving weak equivalences and satisfying a descent condition in terms of hypercovers.
\begin{rem}
Morphisms of stacks are just morphisms in the homotopy category of simplicial presheaves.
Consequently, if there is a natural transformation $\eta:F\rightarrow G$ of simplicial presheaves such that $\eta(A)$ is a weak equivalence for every commutative algebra $A$, and if $F$ and $G$ represent stacks, then they represent the
same stack in $Ho(M^{\wedge})$.
\end{rem}
A first step to do homotopical algebraic geometry is to define a \emph{homotopical algebraic context} $(\mathcal{C},\mathcal{C}_0,\mathcal{A})$ (\cite{TV2} Definition 1.1.0.11).
The category $\mathcal{C}$ is a combinatorial symmetric monoidal model category,
and the categories $\mathcal{C}_0$ and $\mathcal{A}$ are full subcategories of $\mathcal{C}$ stable under weak equivalences.
These categories satisfy several assumptions, required firstly to transpose usual constructions of
linear and commutative algebra,
and secondly to define analogues of derivations, cotangent complexes, etale maps, smooth and infinitesimally smooth maps,
etc. In particular, the category of monoids $Comm(\mathcal{C})$ forms a model category with fibrations and weak equivalences defined in $\mathcal{C}$. The opposite model category $Comm(\mathcal{C})^{op}$ is denoted by $Aff_{\mathcal{C}}$.
A pretopology $\tau$ on $Aff_{\mathcal{C}}$ in the sense of \cite{TV1} (inducing a Grothendieck topology on its homotopy category)
enhances this category into a \emph{model site} $(Aff_{\mathcal{C}},\tau)$ (\cite{TV2} Definition 1.3.1.1).

In this setting, one can express more concretely the descent condition characterizing stacks among simplicial presheaves.
A stack is a simplicial presheaf $F:Comm(\mathcal{C})\rightarrow sSet$ such that
\begin{itemize}
\item $F$ preserves weak equivalences;

\item Let us denote by $\prod^h$ the finite homotopy products. Given a finite family of commutative algebras $\{A_i\}$, the map
\[
F(\prod_i^h A_i)\rightarrow \prod_i F(A_i)
\]
is an isomorphism in $Ho(sSet)$.

\item $F$ satisfies the descent condition: for any $B^{\bullet}\in(A/Comm(\mathcal{C}))^{\Delta}$ such that
$A\rightarrow B^{\bullet}$ defines a $\tau$-hypercover $\mathbb{R}\underline{Spec}_{B^{\bullet}}
\rightarrow\mathbb{R}\underline{Spec}_A$ (that is, satisfies the descent condition of \cite{TV2}, Definition 1.2.12.4), the map
\[
F(A)\rightarrow holim_{\Delta} F(B_n)
\]
is an isomorphism in $Ho(sSet)$.
\end{itemize}

\begin{rem}
A way to tackle the problem of checking the descent condition is to use Corollary B.0.8 of \cite{TV2}.
An interesting example for which one applies this result is to build stacks of algebras.
In \cite{Mur}, Muro builds a stack of algebras over a nonsymmetric operad $\underline{Alg_{\mathcal{C}}}(O)$,
whose evaluation at an algebra $A$ is the classification space of cofibrant $O$-algebras in $A$-modules (that is, the nerve
of the subcategory of weak equivalences between cofibrant objects).
For every algebra $A$, there is a forgetful functor from $O$-algebras
in $A$-modules to $A$-modules, inducing a morphism of simplicial sets $\mathcal{N}wAlg_{Mod(A)}(O)\rightarrow
\mathcal{N}wMod(A)$, hence a morphism of stacks
$\underline{Alg_{\mathcal{C}}}(O)\rightarrow \underline{QCoh}$ (where $\underline{QCoh}$ is the stack of
quasi-coherent modules).
The stack of $O$-algebra structures on an $A$-module $M$ is then defined as the homotopy pullback of this morphism
along the morphism $\mathbb{R}Spec(A)\rightarrow \underline{QCoh}$ representing $M$.
This gives a stack version of Rezk's homotopy pullback theorem \cite{Rez}.
In our context, such a description is not possible because of the absence of model category structure
on algebras over properads.
\end{rem}

In order to define a notion of geometric stacks, which are a certain kind of gluing of representable stacks,
one has to introduce a class of morphisms $\mathbf{P}$ satisfying the adequate properties with respect
to the model pretopology. One obtains the notion of a homotopical algebraic geometry context
$(\mathcal{C},\mathcal{C}_0,\mathcal{A},\tau,\mathbf{P})$, HAG for short, for which
we refer the reader to Definition 1.3.2.13 of \cite{TV2}. This gluing of representable stacks can be realized by the action of a Segal groupoid (see Definition 1.3.1.6 and Proposition 1.3.4.2 of \cite{TV2}).
HAG contexts include in particular:
\begin{itemize}
\item $\mathcal{C}=\mathbb{Z}-Mod$, giving a theory of geometric stacks in classical algebraic geometry;

\item $\mathcal{C}=sMod_{\mathbb{K}}$, the category of simplicial modules over a commutative ring $\mathbb{K}$, giving
a theory of derived or $D^-$-geometric stacks, that is \emph{derived algebraic geometry};

\item $\mathcal{C}=Ch_{\mathbb{K}}$, the category of unbounded cochain complexes of modules over $\mathbb{K}$
with $car(\mathbb{K})=0$, giving a theory of geometric stacks in \emph{complicial algebraic geometry}, or geometric
$D$-stacks;

\item $\mathcal{C}=Sp^{\Sigma}$, the category of symmetric spectra, giving a theory of geometric stacks in
\emph{brave new algebraic geometry}.
\end{itemize}

\subsection{Mapping spaces as higher stacks}

We give here a general representability result for higher stacks arising from mapping spaces in algebras
over monads in diagram categories. This result applies, as special cases, to mapping spaces of operads,
$\frac{1}{2}$-props, dioperads, properads, props, their colored and wheeled variants, as well as to algebras
over operads.

For this, we fix two Grothendieck universes $\mathbb{U}\in\mathbb{V}$ such that $\mathbb{N}\in\mathbb{V}$,
as well as a HAG context $(\mathcal{C},\mathcal{C}_0,\mathcal{A},\tau,P)$. In particular, the category
$\mathcal{C}$ is a $\mathbb{V}$-small $\mathbb{U}$-combinatorial symmetric monoidal model category.
We fix also a regular cardinal $\kappa$ so that (acyclic) cofibrations of $\mathcal{C}$ are generated by
$\kappa$-filtered colimits of generating (acyclic) cofibrations.
For technical reasons, we also need to suppose that the tensor product preserves fibrations.
This assumption is satisfied in particular by simplicial sets, simplicial modules over a ring, and
cochain complexes over a ring. This implies that we can take for instance as HAG contexts the derived algebraic
geometry context and the complicial algebraic geometry context (see \cite{TV2}).
We now consider a $\mathbb{U}$-small category $I$ and the associated category of diagrams $\mathcal{C}^I$,
as well as a monad $T:\mathcal{C}^I\rightarrow\mathcal{C}^I$ satisfying the assumptions of Lemma 2.3 of \cite{SS}
(in particular $T$ preserves $\kappa$-filtered colimits).

\subsubsection{Combinatoriality of $T$-algebras}

We refer the reader to \cite{AR} for a detailed study of presentable and accessible categories.
We just recall here some definitions we need.
\begin{defn}
Let $\kappa$ be a regular cardinal.

(1) An object $X$ in a category $\mathcal{C}$ is $\kappa$-compact in $\mathcal{C}$ if the functor
$Mor_{\mathcal{C}}(X,-)$ preserves $\kappa$-directed colimits. An object is small if it is
$\kappa$-compact for a certain $\kappa$.

(2) A locally small category is said to be $\kappa$-accessible if it admits $\kappa$-directed colimits and
if there exists a set of $\kappa$-compact objects generating the category under these colimits.

(3) A functor $F:\mathcal{C}\rightarrow\mathcal{D}$ is $\kappa$-accessible if $\mathcal{C}$ and $\mathcal{D}$
are $\kappa$-accessible categories and $F$ preserves $\kappa$-directed colimits.

(4) A category is locally presentable if it is an accessible category having all small colimits.
\end{defn}
\begin{defn}
A combinatorial model category is a cofibrantly generated model category which is also locally presentable.
\end{defn}
Such a model category structure can be transfered at the level of diagrams in two Quillent equivalent
ways:
\begin{prop}(Lurie \cite{Lur}, prop. A.2.8.2)
Let $\mathcal{C}$ be a combinatorial model category and $I$ a small category. There exists two Quillen
equivalent combinatorial
model structures on the category of $I$-diagrams $Func(I,\mathcal{C})$:

(i) the projective model structure with weak equivalences and fibrations determined componentwise;

(ii) the injective model structure with weak equivalences and cofibrations determined componentwise.
\end{prop}
We thus use:
\begin{thm}(see \cite{Bor}, theorem 5.5.9)
Let $\mathcal{C}$ be a locally presentable category and $T$ be a monad over $\mathcal{C}$. If $T$
is an accessible functor, then the category $\mathcal{C}^T$ of $T$-algebras is locally presentable.
\end{thm}
We deduce:
\begin{cor}
Let $\mathcal{C}$ be a $\mathbb{U}$-combinatorial model category and $T$ a monad over $\mathcal{C}$ satisfying
the assumptions of lemma 2.3 of \cite{SS}. Then $\mathcal{C}^T$ is a $\mathbb{U}$-combinatorial model category.
\end{cor}
Here we will use the projective model structure for diagrams.
Finally, we will also need the following characterization of cofibrant $T$-algebras:
\begin{lem}
Every cofibrant $T$-algebra is a $\mathbb{U}$-small homotopy colimit of free $T$-algebras which are images of cofibrant $I$-diagrams under the functor $T$.
\end{lem}

\subsubsection{Stability of $T$-algebras under the external tensor product}

The diagram category $\mathcal{C}^I$ is a monoidal model category over $\mathcal{C}$.
We refer the reader to \cite{Fre2} or Section 1 of \cite{Yal1} for the notion of symmetric
monoidal category $\mathcal{E}$ over a base symmetric monoidal category $\mathcal{C}$.
The internal tensor product is the pointwise tensor product of functors, and the external tensor product between a $I$-diagram $F$ and an object $C$ of $\mathcal{C}$
is defined by $(F\otimes_eC)(i)=F(i)\otimes C$ for every $i\in ob(I)$.
The external hom $Hom_{\mathcal{C}^{I}}(-,-):\mathcal{C}^{I}\times\mathcal{C}^{I}\rightarrow\mathcal{C}$
is given by
\[
Hom_{\mathcal{C}^{I}}(X,Y)=\int_{i\in I}Hom_{\mathcal{C}}(X(i),Y(i)).
\]
We deal here with a symmetric monoidal category equipped with a model category
structure. The notion of symmetric monoidal model category consists of axioms formalizing the interplay between the tensor and the model structure,
for which we refer to Hirschhorn \cite{Hir} and Hovey \cite{Hov} for a comprehensive
treatment. We refer to \cite{Fre2} for the notion of symmetric monoidal model category over a base category.

We need the following compatibility assumptions between the monad $T$ and the external tensor product $\otimes_e$.
To any commutative algebra $A$ of $\mathcal{C}$ one can associate a category $Mod_A$ of $A$-modules in $\mathcal{C}^I$
defined obviously like usual modules but with the external tensor product.
We suppose that for any commutative algebra $A$, there exists a monad $T_A:Mod_A\rightarrow Mod_A$ such that
for any $A$-module $M$
\begin{itemize}
\item[(i)] $T_A(M\otimes_e A) = T(M)\otimes_e A$
\item[(ii)]$\mu_{T_A}(M\otimes_e A)=\mu_T\otimes_e A$
\item[(iii)]$\eta_{T_A}(M\otimes_e A)=\eta_T\otimes_e A$
\end{itemize}
(i.e. the monad structure of $T_A$ is determined on free $A$-modules by the monad structure of $T$).
\begin{lem}
For any $T$-algebra $X$ and any commutative algebra $A$ of $\mathcal{C}$, the free $A$-module $X\otimes_e A$ is a $T_A$-algebra in $Mod_A$.
\end{lem}
\begin{proof}
The $T_A$-algebra structure is defined by
\[
T_A(X\otimes_eA)=T(X)\otimes_e A\stackrel{\gamma_X\otimes_e A}{\rightarrow} X\otimes_e A.
\]
\end{proof}
\begin{lem}
For any commutative algebra $A$, the functor $T(-)\otimes_e A$ is a monad on $\mathcal{C}^I$ and
and for any $T$-algebra $X$, the free $A$-module $X\otimes_e A$ is a $T(-)\otimes_e A$-algebra in $\mathcal{C}^I$.
\end{lem}
\begin{proof}
By the equality $T(-)\otimes_e A = T_A(-\otimes_eA)$, the monad structure of $T(-)\otimes_e A$ comes from
the monad structure of $T_A$ and the commutative algebra structure of $A$.
The $T(-)\otimes_e A$-algebra structure on $X\otimes_e A$ is given by
\begin{eqnarray*}
T(X\otimes_e A)\otimes_e A =  & T_A((X\otimes_e A)\otimes_e A) & = T_A(X\otimes_e (A\otimes A)) \\
 & \stackrel{T_A(X\otimes_e (\mu_A))}{\rightarrow} & T_A(X\otimes_e A) = T(X)\otimes_e A \\
 & \stackrel{\gamma_X\otimes_e A}{\rightarrow} & X\otimes_e A
\end{eqnarray*}
where $\mu_A$ is the product of $A$ and $\gamma_X$ the $T$-algebra structure of $X$.
\end{proof}
\begin{lem}
For any commutative algebra $A$, there exists a monad morphism $T\rightarrow T(-)\otimes_e A$.
Consequently, every free $A$-module $X\otimes_e A$ over a $T$-algebra $A$ is a $T$-algebra.
\end{lem}
\begin{proof}
This is just the composite of natural transformations given by
\[
T(X)\stackrel{\cong}{\rightarrow} T(X)\otimes_e 1_{\mathcal{C}}\stackrel{T(X)\otimes_e \eta_A}{\rightarrow}T(X)\otimes_e A
\]
where $1_{\mathcal{C}}$ is the unit of $\mathcal{C}$ and $\eta_A$ is the unit of $A$.
\end{proof}

\subsubsection{Representability of higher stacks obtained from mapping spaces}

Now we can state the main result. For this, let us first note that any homotopy mapping space
$Map(X_{\infty},Y)$ between a cofibrant $T$-algebra $X_{\infty}$ and a fibrant $T$-algebra $Y$ (here these are diagrams
with values in fibrant objects of $\mathcal{C}$) gives rise to the functor
\begin{eqnarray*}
\underline{Map}(X_{\infty},Y): & Comm(\mathcal{C}) \rightarrow & sSet\\
 & A \mapsto Mor_{T-Alg}(X_{\infty},(Y\otimes_e A^{cf})^{\Delta^{\bullet}})\\
\end{eqnarray*}
where $(-)^{cf}$ is a functorial fibrant-cofibrant replacement in $\mathcal{C}$ and $(-)^{\Delta^{\bullet}}$
a simplicial resolution in $T$-algebras.
\begin{thm}
Let $X_{\infty}$ be a cofibrant $T$-algebra such that $X_{\infty}\sim hocolim_iT(X_i)$, and $Y$ be a $T$-algebra with values in fibrations of fibrant perfect objects of $\mathcal{C}$. Then
\begin{eqnarray*}
\underline{Map}(X_{\infty},Y) & \simeq & holim_i\int^h_{k\in ob(I)}\mathbb{R}\underline{Spec}_{Com(X_i(k)
\otimes^{\mathbb{L}}Y(k)^{\vee})}\\
 & \simeq & \underline{Spec}_{hocolim_i\int^h_{k\in ob(I)}Com(X_i(k) \otimes^{\mathbb{L}}Y(k)^{\vee})}.
\end{eqnarray*}
In particular, $\underline{Map}(X_{\infty},Y)$ is a representable stack.
\end{thm}
We refer the reader to Definition 1.2.3.6 and Proposition 1.2.3.7 of \cite{TV2} for the definition and properties
of perfect objects in a given symmetric monoidal model category.
\begin{proof}
The external tensor product of diagram categories preserves weak equivalences between cofibrant objects
of the base category, and homotopy mapping spaces preserve weak equivalences between fibrant targets,
so the functor $\underline{Map}(X_{\infty},Y)$ defines a simplicial presheaf on $Aff_{\mathbb{C}}$
preserving weak equivalences, that is, a prestack. Recall that stacks are the essential image
of the right derived Quillen functor
\[
\mathbb{R}Id:Ho(Aff_{\mathcal{C}}^{\sim,\tau})\rightarrow Ho(Aff_{\mathcal{C}}^{\wedge}),
\]
hence every prestack weakly equivalent to a stack represents a stack.
Moreover, since we are working in a HAG context, Corollary 1.3.2.5 of \cite{TV2} ensures that the topology
$\tau$ is subcanonical, which implies that the model Yoneda embedding gives a fully faithful functor
\[
\mathbb{R}\underline{h}:Ho(Aff_{\mathcal{C}}^{\sim,\tau})\rightarrow Ho(Aff_{\mathcal{C}}^{\sim,\tau})
\]
whose essential image consists of representable stacks.
So we have to prove that the prestack $\underline{Map}(X_{\infty},Y)$ is weakly equivalent to a stack.

In particular, the functor $\mathbb{R}\underline{h}$
commutes with $\mathbb{U}$-small homotopy limits. Given that the model category of $T$-algebras
$T-Alg$ is $\mathbb{U}$-combinatorial, the cofibrant $T$-algebra $X_{\infty}$ is obtained as a $\mathbb{U}$-small
homotopy colimit of free $T$-algebras. This justifies the second line of the formula in the theorem above.

We have to prove the first line of this formula. Homotopy mapping spaces transform homotopy colimits
at the source into homotopy limits, so we have
\[
Map_{T-Alg}(X_{\infty},Y\otimes_e A^{cf}) \sim holim_i Map_{T-Alg}(T(X_i),Y\otimes_e A^{cf}).
\]
For every $i$ we have
\begin{eqnarray*}
Map_{T-Alg}(T(X_i),Y\otimes_e A^{cf}) & = & Mor_{T-Alg}(T(X_i),(Y\otimes_e A^{cf})^{\Delta^{\bullet}}) \\
 & \cong & Mor_{\mathcal{C}^I}(X_i,(Y\otimes_e A^{cf})^{\Delta^{\bullet}}).
\end{eqnarray*}
For any pair of functors $F,G:\mathcal{C}\rightarrow\mathcal{D}$ between two categories $\mathcal{C}$ and
$\mathcal{D}$, the set of natural transformations from $F$ to $G$ can be described by the end
\[
Nat(F,G)=\int_{c\in ob(\mathcal{C})}Mor_{\mathcal{D}}(F(c),G(c)).
\]
We deduce the isomorphism of simplicial sets
\[
Map_{T-Alg}(T(X_i),Y\otimes_e A^{cf}) \cong \int_{k\in ob(I)} Mor_{\mathcal{C}}(X_i(k),(Y(k)\otimes A^{cf})^{\Delta^{\bullet}}).
\]
Let us observe that, since fibrations and limits in diagram categories are defined pointwise,
the fibration condition on matching spaces defining Reedy fibrations (see for instance \cite{GJ}) holds pointwise, thus any Reedy fibration of $(\mathcal{C}^I)^{\Delta^{\bullet}}$ is in particular a Reedy fibration of $\mathcal{C}^{\Delta^{\bullet}}$ at each point. Moreover, weak equivalences of $(\mathcal{C}^I)^{\Delta^{\bullet}}$
are weak equivalences of $(\mathcal{C}^I)$ in each simplicial dimension, which are, in turn, weak equivalences
of $\mathcal{C}$ at each point. This means that weak equivalences of $(\mathcal{C}^I)^{\Delta^{\bullet}}$
are pointwise weak equivalences of $\mathcal{C}^{\Delta^{\bullet}}$. Given that simplicial resolutions
are Reedy fibrant replacements of constant simplicial objects, all of this implies that a simplicial resolution of a diagram defines a simplicial resolution of each of its objects. In particular, the simplicial object
$(Y(k)\otimes A^{cf})^{\Delta^{\bullet}}$ is a simplicial resolution of $Y(k)\otimes A^{cf}$ in $\mathcal{C}$.

Homotopy mapping spaces with cofibrant source and fibrant target does not depend, up to natural weak equivalences, on the choice of a functorial simplicial resolution (See Lemma 4.17).
We can apply this to $Mor_{\mathcal{C}}(X_i(k),(Y(k)\otimes A^{cf})^{\Delta^{\bullet}})$, because by assumption
the $I$-diagram $X_i$ is cofibrant, thus pointwise cofibrant (see Proposition 11.6.3 of \cite{Hir}).
Since $Y(k)$ is a perfect object by assumption, we can choose a more convenient simplicial resolution of
$Y(k)\otimes A^{cf}$ in $\mathcal{C}$ given by $Hom_{\mathcal{C}}(Y(k)^{\vee},(A^{cf})^{\Delta^{\bullet}})$, where
$(A^{cf})^{\Delta^{\bullet}}$ is a simplicial frame on $A^{cf}$.  The proof that this forms indeed a simplicial
resolution is postponed to Lemma 4.15. We have
\begin{eqnarray*}
Map_{T-Alg}(T(X_i),Y\otimes_e A^{cf})  & \cong & \int_{k\in ob(I)} Mor_{\mathcal{C}}(X_i(k),Hom_{\mathcal{C}}(Y(k)^{\vee},(A^{cf})^{\Delta^{\bullet}})) \\
 & \cong & \int_{k\in ob(I)} Mor_{\mathcal{C}}(X_i(k)\otimes^{\mathbb{L}}Y(k)^{\vee}, (A^{cf})^{\Delta^{\bullet}}) \\
 & \cong & \int_{k\in ob(I)} Mor_{Comm(\mathcal{C})}(Com(X_i(k)\otimes^{\mathbb{L}}Y(k)^{\vee}), (A^{cf})^{\Delta^{\bullet}})\\
 & = & \int_{k\in ob(I)} Map_{Comm(\mathcal{C})}(Com(X_i(k)\otimes^{\mathbb{L}}Y(k)^{\vee}), A^{cf}) \\
 & = & \int_{k\in ob(I)}^h Map_{Comm(\mathcal{C})}(Com(X_i(k)\otimes^{\mathbb{L}}Y(k)^{\vee}), A^{cf})
\end{eqnarray*}
where $Com(-)$ is the free commutative monoid in $\mathcal{C}$ and $\int^h$ stands for the homotopy limit.
The free functor preserves cofibrant objects,
so $Com(X_i(k)\otimes^{\mathbb{L}}Y(k)^{\vee})$ is a cofibrant algebra. This implies that the homotopy mapping spaces
$Map_{Comm(\mathcal{C})}(Com(X_i(k)\otimes^{\mathbb{L}}Y(k)^{\vee}), A^{cf})$ are Kan complexes.
The last line follows from Lemma 4.16.
We conclude by observing that the functor
\[
A\mapsto Map_{Comm(\mathcal{C})}(Com(X_i(k)\otimes^{\mathbb{L}}Y(k)^{\vee}), A^{cf})
\]
is nothing but $\mathbb{R}\underline{Spec}_{Com(X_i(k)\otimes^{\mathbb{L}}Y(k)^{\vee})}$.
\end{proof}

Here is the series of lemmas needed in the proof of Theorem 4.14:
\begin{lem}
Let $A$ be a cofibrant perfect object of $\mathcal{C}$ and $B$ be any fibrant object.
Let $B^{\Delta^{\bullet}}$ be a simplicial frame on $B$ (thus a simplicial resolution of $B$).
Then $Hom_{\mathcal{C}}(A^{\vee},B^{\Delta^{\bullet}})$ is a simplicial resolution of $A\otimes B$.
\end{lem}
\begin{proof}
The functor $Hom_{\mathcal{C}}(A^{\vee},-):\mathcal{C}\rightarrow\mathcal{C}$
preserves limits, fibrations between fibrant objects and weak equivalences between fibrant objects.
This is due to the fact that it is a right adjoint to $A^{\vee}\otimes -$, which is a left Quillen functor since
$A^{\vee}$ is cofibrant (recall that $A$ is cofibrant, hence $A^{\vee}=Hom_{\mathcal{C}}(A,R1_{\mathcal{C}})$ is cofibrant too) and $\mathcal{C}$ satisfies the pushout-product axiom.
Applying this functor pointwise defines a functor between simplicial objects
$Hom_{\mathcal{C}}(A^{\vee},-):\mathcal{C}^{\Delta^{op}}\rightarrow\mathcal{C}^{\Delta^{op}}$,
which preserves limits (limits of diagrams are built pointwise), weak equivalences between pointwise fibrant objects
(weak equivalences of diagrams are pointwise weak equivalences), and pointwise fibrations between pointwise
fibrant objects.

Recall that the simplicial resolution $B^{\Delta^{\bullet}}$ is a fibrant replacement in the Reedy model structure
of $\mathcal{C}^{\Delta^{op}}$, that is, we have a factorization
\[
cs_{\bullet}B\stackrel{\sim}{\rightarrow}B^{\Delta^{\bullet}}\twoheadrightarrow cs_{\bullet}*
\]
where $cs_{\bullet}(-)$ is the constant simplicial object, $*$ the terminal object of $\mathcal{C}$,
the first arrow is a weak equivalence and the second a Reedy fibration.
Let $f:X_{\bullet}\twoheadrightarrow Y_{\bullet}$ be any Reedy fibration of $\mathcal{C}^{\Delta^{op}}$
between two Reedy fibrant objects. This gives a simplicial map
\[
Hom_{\mathcal{C}}(A^{\vee},X_{\bullet})\rightarrow Hom_{\mathcal{C}}(A^{\vee},Y_{\bullet}).
\]
Recall that $f$ is a Reedy fibration if and only if the matching maps
\[
X_r\twoheadrightarrow Y_r\times_{M_rY}M_rX
\]
are fibrations of $\mathcal{C}$, where $M_{\bullet}(-)$ is the matching space construction.
The target is given by a pullback
\[
\xymatrix{
Y_r\times_{M_rY}M_rX \ar[r]\ar[d] & M_rX\ar[d]^{M_rf} \\
Y_r\ar[r] & M_rY.
}
\]
The fact that $f$ is a Reedy fibration implies that each $M_rf$ is a fibration (see Proposition 15.3.11 of \cite{Hir}).
Fibrations are stable under pullbacks, so $Y_r\times_{M_rY}M_rX\rightarrow Y_r$
is also a fibration. The fact that $Y_{\bullet}$ is Reedy fibrant implies that it is pointwise fibrant, that is, each $Y_r$ is fibrant. Consequently, $Y_r\times_{M_rY}M_rX$ is fibrant. Since $Hom_{\mathcal{C}}(A^{\vee},-)$ preserves fibrations between
fibrant objects, the map
\[
Hom_{\mathcal{C}}(A^{\vee},X_r)\twoheadrightarrow Hom_{\mathcal{C}}(A^{\vee},Y_r\times_{M_rY}M_rX)
\]
is a fibration. The functor $Hom_{\mathcal{C}}(A^{\vee},-)$ also preserves limits so it commutes with the matching space
construction and we finally get, for every $r\geq 0$, a fibration
\[
Hom_{\mathcal{C}}(A^{\vee},X_r)\twoheadrightarrow Hom_{\mathcal{C}}(A^{\vee},Y_r)\times_{Hom_{\mathcal{C}}(A^{\vee},M_rY)}Hom_{\mathcal{C}}(A^{\vee},M_rX).
\]
This means that
\[
Hom_{\mathcal{C}}(A^{\vee},X_{\bullet})\rightarrow Hom_{\mathcal{C}}(A^{\vee},Y_{\bullet}).
\]
is a Reedy fibration. Applying this property to the particular case of
\[
cs_{\bullet}B\stackrel{\sim}{\rightarrow}B^{\Delta^{\bullet}}\twoheadrightarrow cs_{\bullet}*,
\]
we obtain a factorization
\[
cs_{\bullet}Hom_{\mathcal{C}}(A^{\vee},B)\stackrel{\sim}{\rightarrow}Hom_{\mathcal{C}}(A^{\vee},B^{\Delta^{\bullet}})\twoheadrightarrow cs_{\bullet}*
\]
where the second arrow is a Reedy fibration according to the argument above.
Since Reedy fibrant objects are in particular pointwise fibrant and $Hom_{\mathcal{C}}(A^{\vee},-)$ preserves
weak equivalences between pointwise fibrant objects, the first arrow is a weak equivalence.
We proved that $Hom_{\mathcal{C}}(A^{\vee},B^{\Delta^{\bullet}})$ gives a simplicial resolution of
$Hom_{\mathcal{C}}(A^{\vee},B)$.
\end{proof}

\begin{lem}
Let $\mathcal{M}$ be a cofibrantly generated model category, and let $I$ be a small category.
We suppose that $\mathcal{M}^I$ is equipped with the projective model structure.
Let $X$ be a cofibrant diagram and $Y$ a diagram with values in fibrations between fibrant objects.
Then
\[
Map_{\mathcal{C}^I}(X,Y)= \int_{i\in ob(I)}^h Map_{\mathcal{C}}(X(i),Y(i)).
\]
\end{lem}
\begin{proof}
Recall that $Map_{\mathcal{C}^I}(X,Y)=Mor_{\mathcal{C}^I}(X,Y^{\Delta^{\bullet}})$ can be defined by the
coend
\begin{eqnarray*}
Map_{\mathcal{C}^I}(X,Y) & = & \int_{i\in ob(I)}Mor_{\mathcal{C}}(X(i),Y(i)^{\Delta^{\bullet}})\\
 & = & \int_{i\in ob(I)}Map_{\mathcal{C}}(X(i),Y(i))
\end{eqnarray*}
where the second line holds because, as we already noticed in the proof of Theorem 4.14, a simplicial
resolution on a diagram induces a simplicial resolution on each of its objects.
Any cofibrant diagram is pointwise cofibrant (see Proposition 11.6.3 in \cite{Hir}), and $Y$
is pointwise fibrant, so the homotopy mapping spaces $Map_{\mathcal{C}}(X(i),Y(i))$ are Kan complexes.

This end can be expressed as an equalizer in simplicial sets
\[
\int_{i\in ob(I)}Map_{\mathcal{C}}(X(i),Y(i))\rightarrow\prod_{i\in ob(I)}Map_{\mathcal{C}}(X(i),Y(i))
\rightrightarrows^{d_0}_{d_1}\prod_{i\rightarrow j\in mor(I)} Map_{\mathcal{C}}(X(i),Y(j))
\]
where $d_0$ is the product of precomposition maps $(-)\circ X(u),u:i\rightarrow j$,
and $d_1$ is the product of the postcomposition maps $Y(u)\circ (-),u:i\rightarrow j$.
Since $Y$ takes values in fibrations between fibrant objects, and homotopy mapping spaces preserve fibrations
between fibrant targets, each postcomposition map $Y(u)\circ (-)$ is a fibration of Kan complexes.
Since a product of Kan fibrations is still a Kan fibration, we obtain that $d_1$ is a fibration of Kan
complexes. The equalizer of $d_0$ and $d_1$ is a pullback of $d_1$ along $d_0$: this a pullback consisting of
fibrant objects and one of the two maps is a fibration, that is, a model for the homotopy pullback.
Thus we get
\[
Map_{\mathcal{C}^I}(X,Y) = \int_{i\in ob(I)}^h Map_{\mathcal{C}}(X(i),Y(i))
\]
as expected.
\end{proof}

\begin{lem}
Let $\mathcal{M}$ be a model category, $X$ a cofibrant object of $\mathcal{M}$ and $Y$ any fibrant object of
$\mathcal{M}$. Let $(-)^{\Delta^{\bullet}}$ and $\tilde{(-)}^{\Delta^{\bullet}}$ be two functorial simplicial frames
in $\mathcal{M}$. Then $Mor_{\mathcal{M}}(X,Y^{\Delta^{\bullet}})$ and $Mor_{\mathcal{M}}(X,\tilde{Y}^{\Delta^{\bullet}})$
are naturally weakly equivalent.
\end{lem}
\begin{proof}
Simplicial frames on fibrant objects are simplicial resolutions by Proposition 16.6.7 of \cite{Hir}.
Then one just applies Proposition 16.1.13 of \cite{Hir} combined with Corollary 16.5.5.(4) of \cite{Hir}.
\end{proof}

Theorem 4.14 admits the following variation, whose proof is completely similar:
\begin{thm}
Suppose that all the objects of $\mathcal{C}$ are cofibrant.
Let $X_{\infty}$ be a cofibrant $T$-algebra with values in fibrant objects such that $X_{\infty}\sim hocolim_iT(X_i)$, and $Y$ be a $T$-algebra with values in fibrations between fibrant dualizable objects of $\mathcal{C}$. Then
\begin{eqnarray*}
\underline{Map}(X_{\infty},Y) & \simeq & holim_i\int^h_{k\in ob(I)}\mathbb{R}\underline{Spec}_{Com(X_i(k)
\otimes Y(k)^{\vee})}\\
 & \simeq & \underline{Spec}_{hocolim_i\int^h_{k\in ob(I)}Com(X_i(k) \otimes Y(k)^{\vee})}.
\end{eqnarray*}
In particular, $\underline{Map}(X_{\infty},Y)$ is a representable stack.
\end{thm}
This holds for instance when $\mathcal{C}=sSet$ is the category of simplicial sets.
When $\mathcal{C}=Ch_{\mathbb{K}}$, all the objects are fibrant and cofibrant so we just have to suppose that
$Y$ take its values in dualizable cochain complexes, which are the bounded cochain complexes of finite dimension in each degree.

\subsection{Examples of monads}

We work in the complicial algebraic geometry context (see \cite{TV2}).
Indeed, we will need to dualize cochain complexes, thus to work in an unbounded setting.

Examples of interests for us are the following: operads, $\frac{1}{2}$-props, dioperads,
properads and props. The results readily extends to their colored versions.
Two key features of those examples are the following.
First, the base category is the category of differential graded $\Sigma$-objects for operads, the category
of dg $\Sigma$-biobjects for the others.
Secondly, the free functor can be expressed as a direct sum, over a certain class of graphs,
of tensor products representing indexation of those graphs by operations.
The adjunction between the free functor and the forgetful functor gives rise to a monad on $\Sigma$-objects,
respectively $\Sigma$-biobjects, whose algebras are operads, respectively $\frac{1}{2}$-props, dioperads,
properads and props.

Some care has to be taken concerning the transfer of model category structures.
The category of $\Sigma$-biobjects is a diagram category over cochain complexes, hence a cofibrantly generated model
category. Over a field of caracteristic zero, there is no obstruction to transfer this cofibrantly generated model category structure on operads, $\frac{1}{2}$-props, dioperads, properads and props.
For those results we refer the reader to \cite{BM}, \cite{Fre2}, \cite{Fre3} and the Appendix of \cite{MV2}.
Over a field of positive characteristic, obstructions appear and one gets a cofibrantly generated semi-model structure.
Moreover, one has to consider a restricted class of objects, for instance properads and props with non empty inputs or non empty outputs. See for instance \cite{Fre3}.
\begin{rem}
In a general setting, replacing cochain complexes by a cofibrantly generated symmetric monoidal model category
$\mathcal{C}$, there are conditions on $\mathcal{C}$ such that operads, properads and props inherit
a cofibrantly generated (semi-)model category structure from $\Sigma$-biobjects. We refer the reader to
\cite{BM}, \cite{Fre2}, and \cite{Fre3} for more details.
\end{rem}

Now we are going to explain in details how it works with props.
The method is the same for all the aforementionned structures.
Props can defined as algebras over the monad $U\circ\mathcal{F}$ corresponding to the adjunction
\[
\mathcal{F}:Ch_{\mathbb{K}}^{\mathbb{S}}\rightleftarrows \mathcal{P}:U
\]
between the free prop functor and the forgetful functor.
The free prop functor admits the following description. Let $M$ be a dg $\Sigma$-biobject,
then $\mathcal{F}(M)$ is defined by
\[
\mathcal{F}(M)(m,n)=\bigoplus_{G\in Gr(m,n)}\left( \bigotimes_{v\in Vert(G)}M(|in(v)|,|out(v)|) \right)_{Aut(G)}
\]
where
\begin{itemize}
\item the set $Gr(m,n)$ is the set of directed graphs with $m$ inputs, $n$ outputs and no directed cycles;
\item the set $Vert(G)$ is the set of vertices of $G$;
\item given a vertex $v$ of $G$, the sets $in(v)$ and $out(v)$ are respectively the set of inputs of $v$
and the set of outputs of $v$;
\item he notation $(-)_{Aut(G)}$ stands for the coinvariants under the action of the automorphism group of $G$.
\end{itemize}
Here the tensor product is the usual tensor product of cochain complexes.
For any cdga $A$, we denote by $\mathcal{F}_A$ the free prop functor on $\Sigma$-biobjects in $A$-modules.
It is defined by the same formula, replacing the tensor product of cochain complexes by
the tensor product of $A$-modules.
For any prop $Q$ in $Mod_A$ and any $\Sigma$-biobject $M$ in cochain complexes,we have
\[
\mathcal{F}_A(M\otimes_eA)=\mathcal{F}(M)\otimes_e A
\]
by direct computation with the formula above.
The fact that props are stable under external tensor product with cdgas is actually already proved in Proposition 3.4.
Consequently, we define:
\begin{defn}
The moduli space functor of $Q$-representations of $P_{\infty}$ is given by
\begin{eqnarray*}
\underline{Map}(P_{\infty},Q): & CDGA_{\mathbb{K}} \rightarrow & sSet \\
 & A \mapsto & Mor_{Prop(Ch_{\mathbb{K}})}(P_{\infty},Q\otimes_e (\Omega_{\bullet}\otimes A)).
\end{eqnarray*}
\end{defn}
By Theorem 4.18, this functor forms thus a representable stack if we suppose that each $Q(m,n)$
is a dualizable complex. In particular, for $Q=End_X$ this amounts to suppose that $X$ itself is a
dualizable complex.
We would like to present here the explicit computation.
Any homotopy mapping space with cofibrant source and fibrant target is a Kan complex \cite{Hir},
so the moduli space functor actually takes values in the subcategory of Kan complexes.
Let $A\stackrel{\sim}{\rightarrow} B$ be a weak equivalence of cdgas. The tensor product of cochain complexes preserves
quasi-isomorphisms, and weak equivalences in $CDGA_{\mathbb{K}}$ are defined by the forgetful functor, so the induced
cdga morphism $\Omega_{\bullet}\otimes A\stackrel{\sim}{\rightarrow}\Omega_{\bullet}\otimes B$ is a weak equivalence
of cdgas in each simplicial dimension. The external tensor product
$Q\otimes_e (\Omega_{\bullet}\otimes A)\stackrel{\sim}{\rightarrow}Q\otimes_e (\Omega_{\bullet}\otimes B)$
gives a componentwise quasi-isomorphism of componentwise fibrant dg props, that is, a weak equivalence
of fibrant dg props. Homotopy mapping spaces with cofibrant source preserve weak equivalences between fibrant targets \cite{Hir}, hence a weak equivalence of Kan complexes
\[
\underline{Map}(P_{\infty},Q)(A)\stackrel{\sim}{\rightarrow}\underline{Map}(P_{\infty},Q)(B).
\]
We have a well defined simplicial presheaf preserving weak equivalences, that is, a prestack.

Morphisms of $\Sigma$-biobjects are collections of equivariant morphisms of cochain complexes, that is,
for every $\Sigma$-biobjects $M$ and $N$ we have
\[
Mor_{Ch_{\mathbb{K}}^{\mathbb{S}}}(M,N)=\prod_{m,n}Mor_{Ch_{\mathbb{K}}}(M(m,n),N(m,n))^{\Sigma_m\times\Sigma_n}
\]
where $(-)^{\Sigma_m\times\Sigma_n}$ is the set of invariants under the action of the product of symmetric
groups $\Sigma_m\times\Sigma_n$.
Now let $M$ be a $\Sigma$-biobject and $Q$ be a prop. For every cdga $A$, we have
\begin{eqnarray*}
Mor_{\mathcal{P}}(\mathcal{F}(M),Q\otimes_e(A\otimes\Omega_{\bullet})) & = &
Mor_{Ch_{\mathbb{K}}^{\mathbb{S}}}(M,Q\otimes_e(A\otimes\Omega_{\bullet}))\\
 & = & \prod_{m,n}Mor_{Ch_{\mathbb{K}}}(M(m,n),Q(m,n)\otimes A\otimes\Omega_{\bullet})^{\Sigma_m\times\Sigma_n} \\
 & \cong & \prod_{m,n}Mor_{Ch_{\mathbb{K}}}((M(m,n)\otimes Q(m,n)^{\vee})^{\Sigma_m\times\Sigma_n}, A\otimes\Omega_{\bullet}) \\
 & \cong & \prod_{m,n}Mor_{CDGA_{\mathbb{K}}}(\bigwedge ((M(m,n)\otimes Q(m,n)^{\vee})^{\Sigma_m\times\Sigma_n}), A\otimes\Omega_{\bullet}) \\
 & = & \prod_{m,n}Map_{CDGA_{\mathbb{K}}}(\bigwedge ((M(m,n)\otimes Q(m,n)^{\vee})^{\Sigma_m\times\Sigma_n}), A)\\
 & = & \prod_{m,n}^h\mathbb{R}\underline{Spec}_{\bigwedge ((M(m,n)\otimes Q(m,n)^{\vee})^{\Sigma_m\times\Sigma_n})}(A).
\end{eqnarray*}
Thus for any cofibrant prop $P_{\infty}\simeq hocolim_i \mathcal{F}(M_i)$, we get
\begin{eqnarray*}
\underline{Map}(P_{\infty},Q) & \simeq & holim_i\left( \prod_{m,n}^h\mathbb{R}\underline{Spec}_{\bigwedge ((M_i(m,n)\otimes Q(m,n)^{\vee})^{\Sigma_m\times\Sigma_n})}\right) \\
 & \simeq & \mathbb{R}\underline{Spec}_{C(P_{\infty},Q)}
\end{eqnarray*}
where
\[
C(P_{\infty},Q) = hocolim_i\left( \prod_{m,n}\bigwedge ((M_i(m,n)\otimes Q(m,n)^{\vee})^{\Sigma_m\times\Sigma_n})\right).
\]

We deduce:
\begin{cor}
For any cofibrant quasi-free prop $P_{\infty}=(\mathcal{F}(s^{-1}C),\partial)$, the simplicial Maurer-Cartan
functor
\[
\underline{MC}_{\bullet}(g):CDGA_{\mathbb{K}}\rightarrow sSet
\]
associated to the dg complete $L_{\infty}$-algebra $g=Hom_{\Sigma}(C,Q)$ (see Section 2.2 for the definition of a complete $L_{\infty}$-algebra) and defined by
\[
\underline{MC}_{\bullet}(g)(A)=MC_{\bullet}(g\hat{\otimes}(A\otimes\Omega_{\bullet}))
\]
represents the same stack as $\underline{Map}(P_{\infty},Q)$. In particular, the stack
$\underline{MC}_{\bullet}(g)$ is represented by $C(P_{\infty},Q)$.
\end{cor}
The isomorphism between $MC_{\bullet}(Hom_{\Sigma}(C,Q))$ and $Map(P_{\infty},Q)$
extends to an isomorphism between $MC_{\bullet}(Hom_{\Sigma}(C,Q)\hat{\otimes}A)$ and
$Map(P_{\infty},Q\otimes_e A)$ which is natural in $A$, hence defining an isomorphism of simplicial presheaves.
Every simplicial presheaf isomorphic to a stack is a stack.
Let us note that we already knew from the results of \cite{Yal3} that the simplicial Maurer-Cartan functor forms a prestack (it preserves quasi-isomorphisms of cdgas).

\begin{rem}
Algebraic structures with scalar products and traces need more general classes of graphs, which allows loops and
wheels: cyclic operads, modular operads, wheeled operads, wheeled properads, wheeled props.
Cyclic operads have been introduced in \cite{GK1}, modular operads in \cite{GK2}.
Wheeled operads and props have been studied in \cite{MMS}.
The usual transfer of model structure from $\Sigma$-biobjects does not apply in these cases.
However, this difficulty has been recently got around in \cite{BB} via the formalism of polynomial monads.
More precisely, the authors proved a transfer of model category structure for algebras over
a tame polynomial monad in any compactly generated monoidal model category satisfying the monoid axiom.
For polynomial monads which are not tame, one can get a model structure by imposing restrictions on the
base category. For instance, when working in cochain complexes over a field of characteristic zero,
algebras over any polynomial monad possess a transferred model structure. This covers operads, cyclic operads,
modular operads, $\frac{1}{2}$-props, dioperads, properads, props, wheeled operads, wheeled properads, wheeled props.
\end{rem}

\subsection{The derived stack structure in the non-affine case}

In the derived algebraic geometry context, the derived stack $\underline{Map}(P_{\infty},Q)$ is not affine anymore wether the $Q(m,n)$ are not finite dimensional vector spaces.
However, we would like such a stack to be geometric under the appropriate assumptions for $Q$. A way to do this is to resolve it by a derived Artin $n$-hypergroupoid and apply the main result
of \cite{Pri}, which states an equivalence of $\infty$-categories between $n$-geometric derived Artin stacks and derived Artin $n$-hypergroupoids.
A derived Artin $n$-hypergroupoid is a smooth $n$-hypergroupoid object in the category of affine derived schemes. Precisely, it is a Reedy fibrant simplicial object $X_{\bullet}$ in the category $Aff$
of derived affine schemes such that the partial matching maps
\[
Hom_{sSet}(\Delta^m,X)\rightarrow Hom_{sSet}(\Lambda_k^m,X)
\]
are smooth surjections for all integers $k$, $m$, and weak equivalences for all $k$ and $m>n$.

Now let $P_{\infty}=(\mathcal{F}(s^{-1}C),\partial)$ be a cofibrant quasi-free prop in non-positively graded cochain complexes, and $Q$ be a prop such that for every integers $m$, $n$, the $Q(m,n)$ are perfect complexes with a given finite amplitude. Then, by definition of the degree in the external hom $Hom_{\Sigma}(C,Q)$, the dg complete $L_{\infty}$-algebra $g=Hom_{\Sigma}(C,Q)$
is bounded below by a certain integer $-n$. We know that the simplicial presheaves $\underline{MC}(g)$ and $\underline{Map}(P_{\infty},Q)$ are isomorphic. Moreover, according to the results of Getzler \cite{Get}
and Berglund \cite{Ber}, the natural inclusion of Kan complexes $\gamma_{\bullet}(g)\hookrightarrow MC_{\bullet}(g)$ (where $\gamma_{\bullet}(g)$ is the nerve of $g$ in the sense of Getzler) is a weak
equivalence, hence giving rise to a weak equivalence of simplicial presheaves
\[
\underline{\gamma}_{\bullet}(g)\stackrel{\sim}{\hookrightarrow} \underline{MC}_{\bullet}(g)\cong \underline{Map}(P_{\infty},Q)
\]
where $\underline{\gamma}_{\bullet}(g)(A)=\gamma_{\bullet}(g\hat{\otimes}A)$ for any cdga $A$.
According to \cite{Get}, the fact that $g$ is concentrated in degrees $[-n;+\infty[$ implies that its nerve $\gamma_{\bullet}(g)$ is an $n$-hypergroupoid in the sense of Duskin (which implies
to be an $n$-hypergroupoid in the weaker sense described above). So we expect the left-hand term of the weak equivalence above to give rise, under appropriate changes, to a derived
Artin $n$-hypergroupoid resolving $\underline{Map}(P_{\infty},Q)$.

\section{Examples}

In order to clarify our explanations, we will use the presentation of props by generators and relations.
Since there is a free prop functor as well as a notion of ideal in a prop,
such a presentation has a meaning and always exists.

\subsection{Homotopy Frobenius bialgebras and Poincaré Duality}

\begin{defn}
A (differential graded) Frobenius algebra is a unitary dg commutative associative algebra of
finite dimension $A$ endowed with a symmetric non-degenerated bilinear form $<.,.>:A\otimes A\rightarrow \mathbb{K}$
which is invariant with respect to the product, i.e $<xy,z>=<x,yz>$.
\end{defn}
A topological instance of such kind of algebra is the cohomology ring (over a field) of
a Poincaré duality space.
There is also a notion of Frobenius bialgebra:
\begin{defn}
A differential graded Frobenius bialgebra of degree $m$ is a triple $(B,\mu,\Delta)$ such that:

(i) $(B,\mu)$ is a dg commutative associative algebra;

(ii) $(B,\Delta)$ is a dg cocommutative coassociative coalgebra with $deg(\Delta)=m$;

(iii) the map $\Delta:B\rightarrow B\otimes B$ is a morphism of left $B$-module
and right $B$-module, i.e in Sweedler's notations
\begin{eqnarray*}
\sum_{(x.y)}(x.y)_{(1)}\otimes (x.y)_{(2)} & = & \sum_{(y)}x.y_{(1)}\otimes y_{(2)}\\
 & = & \sum_{(x)}(_1)^{m|x|}x_{(1)}\otimes x_{(2)}.y
\end{eqnarray*}.
This is called the Frobenius relation.
\end{defn}
We will describe the properad encoding such bialgebras in the next section.
The two definitions are strongly related. Indeed, if $A$ is a Frobenius algebra, then the pairing
$<.,.>$ induces an isomorphism of $A$-modules $A\cong A^*$, hence a map
\[
\Delta:A\stackrel{\cong}{\rightarrow} A^* \stackrel{\mu^*}{\rightarrow} (A\otimes A)^*\cong A^*\otimes
A^*\cong A\otimes A
\]
which equips $A$ with a structure of Frobenius bialgebra.
Conversely, one can prove that every unitary counitary Frobenius bialgebra gives rise to a Frobenius algebra,
which are finally two equivalent notions.

The properad $Frob^m$ of Frobenius bialgebras of degree $m$ is generated by
\[
\xymatrix @R=0.5em@C=0.75em@M=0em{
\ar@{-}[dr] & & \ar@{-}[dl] & & \ar@{-}[d] & \\
 & \ar@{-}[d] & & & \ar@{-}[dl] \ar@{-}[dr] & \\
 & & & & & }
\]
where the generating operation of arity $(1,2)$ is of degree $m$,
and the two generators are invariant under the action of $\Sigma_2$.
It is quotiented by the ideal generated by the following relations:
\begin{itemize}
\item[]Associativity and coassociativity
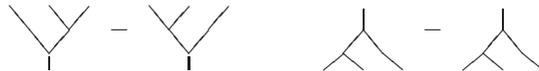
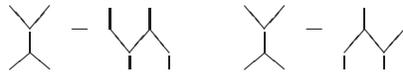
\[
\xymatrix @R=0.5em@C=0.75em@M=0em{
 \ar@{-}[dr] & & \ar@{-}[dr] & & \ar@{-}[dl] & &  \ar@{-}[dr] & & \ar@{-}[dl] & & \ar@{-}[dl]\\
 & \ar@{-}[dr] & &\ar@{-}[dl] & & - &  & \ar@{-}[dr] & &\ar@{-}[dl] & \\
 & & \ar@{-}[d] & & & & & & \ar@{-}[d] & & & \\
 & & & & & & & & & & \\
}
\hspace*{1cm}
\xymatrix @R=0.5em@C=0.75em@M=0em{
 & & \ar@{-}[d] & & & & & & \ar@{-}[d] & & & \\
 & &\ar@{-}[dl] \ar@{-}[dr]& & & - &  & &\ar@{-}[dl] \ar@{-}[dr] & & \\
 &\ar@{-}[dl] \ar@{-}[dr] & &\ar@{-}[dr] & & & & \ar@{-}[dr] \ar@{-}[dl]& & \ar@{-}[dr] & \\
 & & & & & & & & & & \\
}
\]

\item[]Frobenius relations
\[
\xymatrix @R=0.5em@C=0.75em@M=0em{
\ar@{-}[dr] & & \ar@{-}[dl] & & \ar@{-}[d] & & \ar@{-}[d] & \\
 & \ar@{-}[d] & & - & \ar@{-}[dr] & & \ar@{-}[dl] \ar@{-}[dr] & \\
 & \ar@{-}[dl] \ar@{-}[dr] & & & & \ar@{-}[d] & & \ar@{-}[d] \\
 & & & & & & & \\
}
\hspace*{1cm}
\xymatrix @R=0.5em@C=0.75em@M=0em{
\ar@{-}[dr] & & \ar@{-}[dl] & & &\ar@{-}[d] & & \ar@{-}[d] \\
 & \ar@{-}[d] & & - & & \ar@{-}[dl] \ar@{-}[dr]& & \ar@{-}[dl] \\
 & \ar@{-}[dl] \ar@{-}[dr] & & & \ar@{-}[d] & & \ar@{-}[d] & \\
 & & & & & & & \\
}
\]
\end{itemize}
In the unitary and counitary case, one adds a generator for the unit, a generator for the counit and the necessary
compatibility relations with the product and the coproduct. We note the corresponding properad $ucFrob^m$.
We refer the reader to \cite{Koc} for a detailed survey about the role of these operations
and relations in the classification of two-dimensional topological quantum field theories.
\begin{example}
Let $M$ be an oriented connected closed manifold of dimension $n$. Let $[M]\in H_n(M;\mathbb{K})\cong H^0(M;\mathbb{K})\cong \mathbb{K}$
be the fundamental class of $[M]$. Then the cohomology ring $H^*(M;\mathbb{K})$ of $M$ inherits a structure of commutative
and cocommutative Frobenius bialgebra of degree $n$ with the following data:

(i)the product is the cup product
\begin{eqnarray*}
\mu: H^kM\otimes H^lM & \rightarrow & H^{k+l}M \\
 x\otimes y & \mapsto & x\cup y
\end{eqnarray*}

(ii)the unit $\eta:\mathbb{K}\rightarrow H^0M\cong H_nM$ sends $1_{\mathbb{K}}$ on the fundamental class $[M]$;

(iii)the non-degenerate pairing is given by the Poincaré duality:
\begin{eqnarray*}
\beta: H^kM\otimes H^{n-k}M & \rightarrow & \mathbb{K}\\
 x\otimes y & \mapsto & <x\cup y,[M]>
\end{eqnarray*}
i.e the evaluation of the cup product on the fundamental class;

(iv) the coproduct $\Delta=(\mu\otimes id)\circ (id\otimes \gamma)$ where
\[
\gamma: \mathbb{K}\rightarrow \bigoplus_{k+l=n}H^kM\otimes H^lM
\]
is the dual copairing of $\beta$, which exists since $\beta$ is non-degenerate;

(v)the counit $\epsilon=<.,[M]>:H^nM\rightarrow \mathbb{K}$ i.e the evaluation on the fundamental class.
\end{example}
An important problem in the study of the Poincaré duality phenomenon is to understand what happens at the cochain level,
where there is a family of finer operations gathered into the structure of a Frobenius bialgebra up to homotopy,
i.e algebras over a resolution $ucFrob_{\infty}^n$ of $ucFrob^n$. These operations induce the Poincaré duality at the
cohomology level. The properads $Frob$ and $ucFrob$ are proved
to be Koszul respectively by Corollary 2.10.1 and Theorem 2.10.2 in \cite{CMW}, hence there are explicit resolutions.
We thus get a moduli stack of "cochain-level Poincaré duality" $ucFrob^n_{\infty}\{C^*M\}$ on an $n$-dimensional
compact oriented manifold $M$.
To conclude, let us note that a variant of Frobenius bialgebras called special Frobenius bialgebra is closely related to open-closed topological field theories \cite{LP} and conformal field theories \cite{FFRS}.

\subsection{Homotopy involutive Lie bialgebras in string topology}

Lie bialgebras originate from mathematical physics. In the study of certain dynamical systems whose phase spaces
form Poisson manifolds, certain transformation groups acting on these phase spaces are not only Lie
groups but Poisson-Lie groups. A Poisson-Lie group is a Poisson manifold with a Lie group structure such that
the group operations are morphisms of Poisson manifolds, that is, compatible with the Poisson bracket on the
ring of smooth functions. The tangent space at the neutral element is then more than a Lie algebra.
It inherits a ``cobracket'' from the Poisson bracket on smooth functions which satisfies some compatibility
relation with the Lie bracket. Lie bialgebras are used to build quantum groups and appeared for the first time
in the seminal work of Drinfeld \cite{Dri1}. We will come back to this quantization process later in 5.6.

The properad $BiLie$ encoding Lie bialgebras is generated by
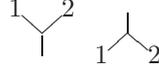
\[
\xymatrix @R=0.5em@C=0.75em@M=0em{
1\ar@{-}[dr] & & 2\ar@{-}[dl] & & \ar@{-}[d] & \\
 & \ar@{-}[d] & & & \ar@{-}[dl] \ar@{-}[dr] & \\
 & & & 1 & & 2 }
\]
where the two generators comes with the signature action of $\Sigma_2$, that is,
they are antisymmetric.
It is quotiented by the ideal generated by the following relations
\begin{itemize}
\item[]Jacobi
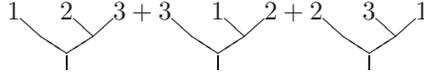
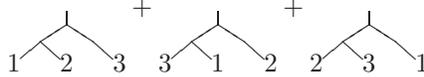
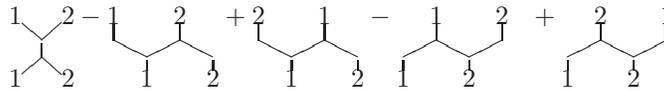
\[
\xymatrix @R=0.5em@C=0.75em@M=0em{
 1\ar@{-}[dr] & & 2\ar@{-}[dr] & & 3\ar@{-}[dl]\\
 & \ar@{-}[dr] & &\ar@{-}[dl] &  \\
 & & \ar@{-}[d] & & \\
 & & & & \\
}
+
\xymatrix @R=0.5em@C=0.75em@M=0em{
 3\ar@{-}[dr] & & 1\ar@{-}[dr] & & 2\ar@{-}[dl]\\
 & \ar@{-}[dr] & &\ar@{-}[dl] &  \\
 & & \ar@{-}[d] & & \\
 & & & & \\
}
+
\xymatrix @R=0.5em@C=0.75em@M=0em{
 2\ar@{-}[dr] & & 3\ar@{-}[dr] & & 1\ar@{-}[dl]\\
 & \ar@{-}[dr] & &\ar@{-}[dl] &  \\
 & & \ar@{-}[d] & & \\
 & & & & \\
}
\]
\item[]co-Jacobi
\[
\xymatrix @R=0.5em@C=0.75em@M=0em{
 & & \ar@{-}[d] & & \\
 & &\ar@{-}[dl] \ar@{-}[dr]& & \\
 &\ar@{-}[dl] \ar@{-}[dr] & &\ar@{-}[dr] & \\
 1 & & 2 & & 3  \\
}
+
\xymatrix @R=0.5em@C=0.75em@M=0em{
 & & \ar@{-}[d] & & \\
 & &\ar@{-}[dl] \ar@{-}[dr]& & \\
 &\ar@{-}[dl] \ar@{-}[dr] & &\ar@{-}[dr] & \\
 3 & & 1 & & 2 \\
}
+
\xymatrix @R=0.5em@C=0.75em@M=0em{
 & & \ar@{-}[d] & & \\
 & &\ar@{-}[dl] \ar@{-}[dr]& & \\
 &\ar@{-}[dl] \ar@{-}[dr] & &\ar@{-}[dr] & \\
 2 & & 3 & & 1 \\
}
\]
\item[]The cocycle relation
\[
\xymatrix @R=0.5em@C=0.75em@M=0em{
1\ar@{-}[dr] & & 2\ar@{-}[dl]  \\
 & \ar@{-}[d] & \\
 & \ar@{-}[dl] \ar@{-}[dr] &  \\
 1 & & 2 \\
}
-
\xymatrix @R=0.5em@C=0.75em@M=0em{
 1\ar@{-}[d] & & 2\ar@{-}[d] & \\
 \ar@{-}[dr] & & \ar@{-}[dl] \ar@{-}[dr] & \\
 & \ar@{-}[d] & & \ar@{-}[d] \\
 & 1 & & 2 \\
}
+
\xymatrix @R=0.5em@C=0.75em@M=0em{
 2\ar@{-}[d] & & 1\ar@{-}[d] & \\
 \ar@{-}[dr] & & \ar@{-}[dl] \ar@{-}[dr] & \\
 & \ar@{-}[d] & & \ar@{-}[d] \\
 & 1 & & 2 \\
}
-
\xymatrix @R=0.5em@C=0.75em@M=0em{
 &1\ar@{-}[d] & & 2\ar@{-}[d] \\
 & \ar@{-}[dl] \ar@{-}[dr]& & \ar@{-}[dl] \\
 \ar@{-}[d] & & \ar@{-}[d] & & \\
 1& &2 & \\
}
+
\xymatrix @R=0.5em@C=0.75em@M=0em{
 &2\ar@{-}[d] & & 1\ar@{-}[d] \\
 & \ar@{-}[dl] \ar@{-}[dr]& & \ar@{-}[dl] \\
 \ar@{-}[d] & & \ar@{-}[d] & & \\
1 & &2 & \\
}
\]
\end{itemize}
The cocycle relation means that the coLie cobracket of a Lie bialgebra $g$ is a cocycle in the Chevalley-Eilenberg
complex $C^*_{CE}(g,\Lambda^2g)$, where $\Lambda^2g$ is equipped with the structure of
$g$-module induced by the adjoint action.

Involutive Lie bialgebras appeared a few years after the birth of string topology \cite{CS1} in \cite{CS2}.
In the properad of involutive Lie bialgebras $BiLie^{\diamond}$, one adds the involutivity relation
\[
\xymatrix @R=0.5em@C=0.75em@M=0em{
 & \ar@{-}[d] & \\
 & \ar@{-}[dl] \ar@{-}[dr] & \\
 \ar@{-}[dr] & & \ar@{-}[dl] \\
 & \ar@{-}[d] & \\
 & & \\
}
\]
The string homology of a smooth manifold  $M$ is defined as the (reduced) equivariant homology of the free loop
space $LM$ of $M$, noted $H_*^{S^1}(LM)$. The word equivariant refers here to the action of the circle on loops
by reparameterization. According to \cite{CS2}, the string homology of a smooth manifold forms an involutive Lie
bialgebra. Let us note that for an $n$-dimensional manifold, the bracket and the cobracket of this
structure are of degree $2-n$. We denote this graded version by $BiLie^{\diamond,2-n}$.
It is conjectured in \cite{CS2} that such a structure is induced in homology by a family of string operations
defined at the chain level, on the complex of $S^1$-equivariant chains $C_*^{S^1}(LM)$.
Such operations should then form a structure of involutive Lie bialgebra up to homotopy, that is,
a $BiLie^{\diamond,2-n}_{\infty}$-algebra structure on $C_*^{S^1}(LM)$. If the moduli space
$BiLie^{\diamond,2-n}_{\infty}\{C_*^{S^1}(LM)\}$ is not empty, then we get a moduli stack of chain-level string operations
on $M$. Moreover, the properad of involutive Lie bialgebras is Koszul by Theorem 2.8 of \cite{CMW}, with the properad
of Frobenius bialgebras as Koszul dual, hence a small explicit model for $BiLie^{\diamond,2-n}_{\infty}$.

\subsection{Geometric structures on graded manifolds}

The category of formal $\mathbb{Z}$-graded pointed manifolds is the category opposite to the category whose objects
are isomorphism classes of completed finitely generated free graded commutative algebras equipped with the adic topology.
The morphisms are isomorphisms classes of continuous morphisms of topological commutative algebras.
The choice of a representant in a given class is a choice of a local coordinates system.
Conversely, given a graded vector space $V$, there is an associated pointed formal graded manifold
$\mathcal{V}$ whose structure sheaf $\mathcal{O}_{\mathcal{V}}$ (actually just a skyscraper with a single stalk
over the distinguished point) is isomorphic to $\hat{S(V^*)}$ (the completed symmetric algebra on the dual of $V$).
The choice of such an isomorphism is a choice of local coordinates system, and if $\{v_{\alpha}\}_{\alpha\in I}$
is a basis of $V$, then the structure sheaf can be identified with the commutative ring of formal power series $\mathbb{K}[[v_{\alpha^*}]]$. The tangent sheaf is the $\mathcal{O}_{\mathcal{V}}$-module of derivations
$\mathcal{T}_V=Der(\mathcal{O}_{\mathcal{V}})$ and the sheaf of polyvector fields is the exterior algebra
$\Lambda^{\bullet}\mathcal{T}_V$.

Important examples of graded manifolds concentrated in degree 0 are the $\mathbb{R}^d$, for any integer $d$.
A famous sort of geometric structure is the notion of Poisson manifold.
A finite dimensional Poisson manifold $M$ is a finite dimensional manifold equipped with a Poisson structure
on its algebra of smooth functions $\mathcal{C}^{\infty}(M)$.
In the case $M=\mathbb{R}^d$, there are several possible Poisson brackets, all of the form
\[
\{f,g\}=\sum_{i,j=1}^d \pi_{ij}(x)\frac{\partial f}{\partial x_i}\frac{\partial g}{\partial x_j}
\]
where the $\pi_{ij}$ are smooth functions satisfying the required relations to get a Poisson bracket.
Let us define formal variables $\psi_i=\frac{\partial}{\partial x_i}$ of degree $1$, then the graded algebra
of polyvector fields on $\mathbb{R}^d$ is defined by
\[
T_{poly}(\mathbb{R}^d)=\mathcal{C}^{\infty}(\mathbb{R}^d)[[\psi_1,...,\psi_d]]
\subset \mathbb{R}[[x_1,...,x_d,\psi_1,...,\psi_d]]
\]
where the $x_i$ are degree $0$ variables forming a basis of $\mathbb{R}^d$.
This graded space is equipped with a Gerstenhaber bracket of degree $-1$, the Schouten bracket, induced by the formulae $[x_i,x_j]=0$,
$[\psi_i,\psi_j]=0$ and $[x_i,\psi_j]=\delta_i^j$ (where $\delta_i^j$ is the Kronecker symbol).

The idea of presenting (formal germs of) geometric structures as algebras over a prop appeared in \cite{Mer1}.
For this, the author introduce Lie $n$-bialgebras, a variant of Lie bialgebras for which the Lie bracket is in degree $n$.
Such bialgebras are parameterized by a properad $BiLie^n$ which is Koszul, hence an explicit resolution.
In the case $n=1$, this resolution $BiLie^1_{\infty}$ is detailed in \cite{Mer1} and the main result is the following:
\begin{thm}(see \cite{Mer3}, Proposition 3.4.1 and Corollary 3.4.2)

(1)The $BiLie^1_{\infty}$-algebra structures on a graded vector space $V$ are in bijection with degree $1$ elements of the graded Lie algebra of polyvector fields
$\Lambda^{\bullet}\mathcal{T}_V$ vanishing at $0$ and such that $[\pi,\pi]=0$ (that is, Poisson structures on the formal graded manifold corresponding to $V$);

(2)Consequently, the $BiLie^1_{\infty}$-algebra structures on $\mathbb{R}^d$ are in bijection with Poisson structures on $\mathbb{R}^d$ vanishing at $0$.
\end{thm}
Let us note that this result actually encompass all Poisson structures, because any Poisson structure on $V$ not vanishing at $0$ corresponds to a Poisson structure on
$V\oplus\mathbb{K}$ vanishing at $0$.
We get a moduli stack of Poisson structures on a formal graded manifold $BiLie^1_{\infty}\{V\}$.

\subsection{Deformation quantization of graded Poisson manifolds}

Let $M$ be a finite dimensional Poisson manifold.
In this situation, one considers the complex of polydifferential operators $D_{poly}(M)$, which is a sub-Lie algebra of the Hochschild complex $CH^*(\mathcal{C}^{\infty}(M),\mathcal{C}^{\infty}(M))$,
and the complex of polyvector fields
\[
T_{poly}(M)=\left( \bigoplus_{k\geq 0}\bigwedge^k\Gamma T(M)[-k] \right)[1]
\]
where $\Gamma T(M)$ is the space of sections of the tangent bundle on $M$.
To define a Poisson structure on $M$ is equivalent to define a bivector $\Pi\in\bigwedge^2\Gamma T(M)$, the Poisson bivector, and set $\{f,g\}=\Pi(df,dg)$.

We recall the fundamental formality theorem of Kontsevich \cite{Ko2}:
\begin{thm}(Kontsevich \cite{Ko2})
There exists a $L_{\infty}$-quasi-isomorphism $T_{poly}(M)\stackrel{\sim}{\rightarrow} D_{poly}(M)$
realizing in particular the isomorphism of the Hochschild-Kostant-Rosenberg theorem.
\end{thm}
Note that a $L_{\infty}$-quasi-isomorphism of dg Lie algebras is actually equivalent to a chain of
quasi-isomorphisms of dg Lie algebras.
Quantization of Poisson manifolds is then obtained as a consequence of this formality theorem combined with
the following result:
\begin{thm}
A $L_{\infty}$-quasi-isomorphism of nilpotent dg Lie algebras induces a bijection between the corresponding
moduli sets of Maurer-Cartan elements.
\end{thm}
For a graded Lie algebra $g$, we note $g[[\hbar]]_+=\bigoplus\hbar g^n[[\hbar]]$.
The Maurer-Cartan set $MC(T_{poly}(M)[[\hbar]]_+)$ is the set of Poisson algebra structures on $\mathcal{C}^{\infty}(M)[[\hbar]]$.
The Maurer-Cartan set $MC(D_{poly}(M)[[\hbar]]_+)$ is the set of $*_{\hbar}$-products, which are assocative products on $\mathcal{C}^{\infty}(M)[[\hbar]]$ of the form $a.b + B_1(a,b)t+...$ (i.e these products restrict to the usual commutative associative product on $\mathcal{C}^{\infty}(M)$).
The conclusion follows:
\[
\mathcal{MC}(T_{poly}(M)[[\hbar]]_+)\cong \mathcal{MC}(D_{poly}(M)[[\hbar]]_+),
\]
that is isomorphism classes of Poisson structures on $M$ are in bijection with equivalence classes of $*_{\hbar}$-products.

An important example of Poisson manifold is $M=\mathbb{R}^d$.
Poisson structures on $\mathbb{R}^d$ are precisely the Maurer-Cartan elements of the Schouten Lie algebra of polyvector fields $T_{poly}(\mathbb{R}^d)[1]$.
Moreover, the Lie group $Diff(\mathbb{R}^d)$ of diffeomorphisms of $\mathbb{R}^d$ is the gauge group of
this Lie algebra (and thus acts canonically on Poisson structures).
Kontsevich built explicit star products quantizing these Poisson structures, with formulae involving integrals
on compactification of configuration spaces and deeply related to the theory of multi-zeta functions \cite{Ko2}.

According to \cite{Mer3} and \cite{Mer4}, one can see Kontsevich deformation quantizations, and more generally deformation quantization
of Poisson structures on formal graded manifolds, as morphisms of dg props
\[
DefQ\rightarrow \hat{BiLie^{1,\circlearrowleft}_{\infty}}.
\]
The prop $DefQ$ is built such that the $DefQ$-algebra structures on a finite dimensional graded vector space $V$ correspond to the Maurer-Cartan elements of the dg Lie algebra $D_{poly}(V)$
of polydifferential operators on $\mathcal{O}_V$ (the structure sheaf of the formal graded manifold associated to $V$).
The prop $BiLie^{1,\circlearrowleft}_{\infty}$ is actually a wheeled prop, that is, a generalization of the notion of prop for which one allows graphs with loops and wheels.
We refer the reader to \cite{MMS} for more details. It is a wheeled extension of the prop $BiLie^1_{\infty}$ encoding Poisson structures.
The prop $\hat{BiLie^{1,\circlearrowleft}_{\infty}}$ appearing as the target of the morphism above is the completion of $BiLie^1_{\infty}$ with respect to the filtration
by the number of vertices. To define a composite
\[
DefQ\rightarrow \hat{BiLie^{1,\circlearrowleft}_{\infty}}\rightarrow End_V
\]
amounts then to consider a Maurer-Cartan element of $D_{poly}(V)[[h]]$ associated to a Poisson structure on the formal graded manifold corresponding to $V$, that is,
a deformation quantization of this Poisson structure.
Let us note that a generic $BiLie^{1,\circlearrowleft}_{\infty}$-algebra structure on $V$ corresponds actually to a more general notion of wheeled Poisson structure, whose geometric
meaning is not fully understood yet. Usual Poisson structures are special cases of such algebras.

\subsection{Quantization functors}

\subsubsection{A candidate for the moduli stack of quantization functors}

The prop approach to a general theory of quantization functors has been set up in \cite{EE}.
Certain quantization problems can be formalized in the following way. Consider a "classical"
category $\mathcal{C}_{class} = P_{class} -Alg$ of algebras over a prop $P_{class}$
and a "quantum" category $\mathcal{C}_{quant} = P_{quant} -Alg$ of algebras over a prop $P_{quant}$.
Algebras are defined in $\mathbb{K}$-modules on the classical side and $\mathbb{K}[[\hbar]]$-modules
on the quantum side, where $\hbar$ is a formal parameter. The quantization problem is then to find
a right inverse $Q:\mathcal{C}_{class} \rightarrow \mathcal{C}_{quant}$ for a given
"semi-classical limit" functor $SC:\mathcal{C}_{quant} \rightarrow \mathcal{C}_{class}$.
A classical example of such a situation is the quantization of Lie bialgebras \cite{EK1}.

The functors $SC$ and $Q$ can be induced by props morphisms $\underline{SC}:P_{class} \rightarrow P_{quant}/(\hbar)$
and $\underline{Q}:P_{quant} \rightarrow P_{class}[[\hbar]]$ satisfying
$(\underline{Q}\: mod \: \hbar )\circ \underline{SC}= Id_{P_{class}}$.
The morphism $\underline{Q}$ is called a quantization morphism. If you suppose $\underline{SC}$
surjective, then $\underline{SC}$ and $\underline{Q} \: mod \: \hbar$ are isomorphisms,
and $\underline{Q}$ becomes an isomorphism too by Hensel's lemma.
This implies an equivalence of categories between $\mathcal{C}_{quant}$ and $\mathcal{C}_{class}[[\hbar]]$
called a dequantization result. The set of quantization morphisms forms a torsor
for the subgroups of props automorphisms in $Aut(P_{quant})$ and $Aut(P_{class}[[\hbar]])$
reducing to the identity modulo $\hbar$.

Since such props are not cofibrant, we do not get a well-behaved moduli space just by taking
the simplicial set of maps $Map(P_{quant},P_{class}[[\hbar]])$.
Instead of props isomorphisms, we consider isomorphisms between a cofibrant replacement of
the source and a fibrant replacement of the target in the homotopy category of props.
This incarnates in our particular situation a general principle inspired by Kontsevich's derived geometry program,
aiming to work out singularities of moduli spaces by passing to the derived category.
More precisely, we work with a cofibrant resolution $(P_{class})_{\infty}$.
Our homotopy quantization problem is the data of two morphisms
$\underline{SC}:P_{class} \rightarrow (P_{quant})_{\infty}/(\hbar)$
and $\underline{Q}:(P_{quant})_{\infty} \rightarrow P_{class}[[\hbar]]$ satisfying
$(\underline{Q}\: mod \: \hbar )\circ \underline{SC}\sim Id_{P_{class}}$
and $(\underline{SC}[[\hbar]]\circ \underline{Q}\sim Id_{(P_{quant})_{\infty}}$.
So the set of homotopy classes of quantization morphisms is a subset of
the set $\pi_0 Map((P_{quant})_{\infty},P_{class}[[\hbar]])$,
on which act the groups $\pi_0 haut((P_{quant})_{\infty})$ and $\pi_0 haut(P_{class}[[\hbar]])$.

\subsubsection{The case of Lie bialgebras}

In the special case of Lie bialgebras and homotopy Lie bialgebras, another point of view is the prop profile approach of \cite{Mer2}.
We recall the main result of \cite{Mer2}. Merkulov constructed a quasi-free prop, the dg prop of quantum homotopy bialgebra
structures $DefQ$. Let us note $BiLie$ the prop encoding Lie bialgebras and $\hat{BiLie}$ its completion with respect to
the filtration by the number of vertices. Let $\pi:BiLie_{\infty}\rightarrow BiLie$ be the minimal model of $BiLie$.
Merkulov proved in \cite{Mer2} that there exists a quantization morphism
of dg props $F:DefQ\rightarrow \hat{BiLie_{\infty}}$ such that the composite
\[
DefQ\stackrel{F}{\rightarrow} \hat{BiLie_{\infty}} \stackrel{\hat{\pi}}{\rightarrow} \hat{BiLie}
\]
is the quantization morphism of Etingof and Kazhdan \cite{EK2}.
When working over a field of characteristic zero, the prop $DefQ$ is already cofibrant
and there is no need to take a resolution of it.
Another advantage of this approach is that we can stay in $\mathbb{K}$-modules instead of working with
$\mathbb{K}[[\hbar]]$-modules.
The moduli stack of quantization functors of Lie bialgebras is then the one associated to the mapping space
$DefQ\{BiLie\}$.

\begin{rem}
The Grothendieck-Teichmüller group encodes geometric actions of the Galois
groups on curves, and are aimed to understand Galois groups through these actions.
It is a subgroup of the group of Lie formal series in two variables satisfying certain relations,
that is, a subgroup of the profinite completion of the etale fundamental group of the projective line
minus three points, in which the absolute Galois group of $\mathbb{Q}$ embeds.
In \cite{Ko1}, Kontsevich pointed out that a prounipotent version of the Grothendieck-Teichmüller
group defined by Drinfeld in \cite{Dri2} should be in some sense a universal group encoding the different
deformation quantizations of a Poisson algebra. Tamarkin's approach of the Deligne conjecture \cite{Tam}
via formality morphisms of the little disks operad, as well as the quantization of Lie bialgebra
by Etingof and Kazhdan \cite{EK1}, rely on Drinfeld's theory of associators \cite{Dri2}.
The prounipotent Grothendieck-Teichmüller group  $GT$ acts in a non trivial way on these quantizations
which is not fully understood yet.
\end{rem}

\subsection{Homotopy operads}

Operads can be seen as algebras over an $\mathbb{N}$-colored non-symmetric operad $O^{univ}$ described in Example 1.5.6
of \cite{BM3}. More generally, for any set $C$, the category of $C$-colored operads is a category of algebras over a
certain colored non-symmetric operad. This construction works in any symmetric monoidal category. The category of cochain complexes satisfies the assumptions of Theorem 2.1 of \cite{BM} and admits a coalgebra interval, which is given by
the normalization of the standard $1$-simplex, hence a model category structure on the category of dg $\mathbb{N}$-colored
operads. We thus define a meaningful notion of homotopy dg operad as an algebra over a cofibrant resolution
$O^{univ}_{\infty}$ of $O^{univ}$ in the category of $\mathbb{N}$-indexed collections of cochain complexes.
Such a resolution is provided for instance by the Boardman-Vogt style resolution of \cite{BM3}.
There is a variant of $O^{univ}$ encoding non-symmetric dg operads, for which a smaller resolution exists via
the Koszul duality for colored operads developed in \cite{vdL}.

Deformation theory of operads can then be studied via the moduli stack of homotopy operad structures, its tangent complexes
and obstruction theory.

\section{Tangent complexes, higher automorphisms and obstruction theory}

\subsection{Tangent and cotangent complexes}

Recall that we work in the complicial algebraic geometry context, whose base category is the category of unbounded cochain complexes.
We refer the reader to Chapter 1.2 of \cite{TV2} for a more general treatment of these notions in a given HAG context.

Let $F$ be a stack, $A$ a cdga and $M$ a $A$-module. Let us fix a $A$-point $x:\mathbb{R}\underline{Spec}_A
\rightarrow F$. One defines a derivation space
\[
\mathbb{D}er_F(\mathbb{R}\underline{Spec}_A,M)=Map_{\mathbb{R}\underline{Spec}_A/St}(\mathbb{R}\underline{Spec}_{A\oplus M},
F)
\]
where $\mathbb{R}\underline{Spec}_A/St$ is the model category of stacks under $\mathbb{R}\underline{Spec}_A$
and $A\oplus M$ is the square-zero extension of $A$ by $M$.
The cotangent complex $\mathbb{L}_{F,x}$ of $F$ at $x$, if it exists, satisfies
\[
\mathbb{D}er_F(\mathbb{R}\underline{Spec}_A,M)\simeq Map_{A-Mod}(\mathbb{L}_{F,x},M)
\]
where $A-Mod$ is the category of differential graded $A$-modules with the model structure induced
by the adjunction between the free $A$-module functor and the forgetful functor (see \cite{SS}).
The cotangent complex is well defined only up to quasi-isomorphism, so one usually considers it as an object
in the homotopy category of $A$-modules $Ho(A-Mod)$. The tangent complex is then the dual $A$-module
\[
\mathbb{T}_{F,x}=\underline{Hom}_A(\mathbb{L}_{F,x},A)\in Ho(A-Mod)
\]
where $\underline{Hom}_A$ is the internal hom of $A$-modules.
\begin{prop}
Let $A$ be a cdga, $M$ a $A$-module, $F$ a stack with a $A$-point $\mathbb{R}\underline{Spec}_A\rightarrow F$.
Let us suppose that either the cotangent complex of $F$ at this point is dualizable, or $M$ is a dualizable $A$-module.
Then there are group isomorphisms
\[
\pi_*\mathbb{D}er_F(\mathbb{R}\underline{Spec}_A,M)\cong H^{-*}(\mathbb{T}_{F,x}\otimes_A M).
\]
\end{prop}
\begin{proof}
The Dold-Kan correspondence between non-positively graded cochain complexes and simplicial modules induces a Quillen equivalence
between simplicial $A$-modules (simplicial objects in $A$-modules) and non-positively graded dg $A$-modules.
The mapping space $Map_{A-Mod}(M,N)$ of dg $A$-modules forms a simplicial $A$-module whose normalization $N_*Map_{A-Mod}(M,N)$
can be identified, via the aforementioned equivalence, with the smooth truncation $\tau_{\geq 0}\underline{Hom}_A(M,N)$ of the internal
hom of dg $A$-modules with reverse grading defined by
\[
\tau_{\geq 0}\underline{Hom}_A(M,N)^*= \begin{cases}
\underline{Hom}_A(M,N)^*, & \text{if $*>0$},\\
\underline{Hom}_A(M,N)/Im(\delta) & \text{if *=0} \\
0, & \text{if $*<0$}.
\end{cases}
\]
where $\delta$ is the differential of $\underline{Hom}_A(M,N)$. This kind of truncation does not change the cohomology for $*=0$.
Consequently
\begin{eqnarray*}
\pi_*\mathbb{D}er_F(\mathbb{R}\underline{Spec}_A,M) & \cong & \pi_*Map_{A-Mod}(\mathbb{L}_{F,x},M) \\
 & \cong & H_*N_*Map_{A-Mod}(\mathbb{L}_{F,x},M) \\
 & \cong & H_*\tau_{\geq 0}\underline{Hom}_A(\mathbb{L}_{F,x},M) \\
 & \cong & H_*\underline{Hom}_A(\mathbb{L}_{F,x},M) \\
 & \cong & H^{-*}(\mathbb{T}_{F,x}\otimes_A M).
\end{eqnarray*}
The third line is a consequence of the aforementioned Dold-Kan
correspondence. The fourth line holds because $*$ is non negative. The last line follows from the definition of the tangent complex as the dual of the
cotangent complex inside dg $A$-modules, and from the adjunction relation between the internal hom and the tensor
product.
\end{proof}

In the sequel we work with representable stacks, which all admit a cotangent complex and an obstruction theory.
The cotangent complex of the derived spectrum of a given cdga coincides with the usual notion of cotangent complex
of this cdga. We actually only need the usual notions of derivations,
square-zero extensions, tangent and cotangent complexes in the category of commutative differential graded algebras $CDGA_{\mathbb{K}}$.
Moreover, the modules we will use for obstruction theory are dualizable, so that Proposition 6.1 applies.

\subsection{Tangent complexes and higher automorphisms}

First, since homotopy classes of stack morphisms between
$\mathbb{R}\underline{Spec}_A$ and $\mathbb{R}\underline{Spec}_{C(P_{\infty},Q)}$ are in bijection with $[C,A]_{CDGA_{\mathbb{K}}}$, to fix a $A$-point of $\mathbb{R}\underline{Spec}_{C(P_{\infty},Q)}$ is equivalent
to fix a cdga morphism $x:C(P_{\infty},Q)\rightarrow A$.

Now let us fix a prop morphism $\varphi:P_{\infty}\rightarrow Q$. It gives a morphism of props in $A$-modules
$\varphi\otimes_e id_A:P_{\infty}\otimes_e id_A\rightarrow Q\otimes_e id_A$, which fixes in turn a dg prop morphism
$[x\mapsto \varphi(x)\otimes_e 1_A]:P_{\infty}\rightarrow Q\otimes_e A$. This is a vertex in the Kan complex
$\underline{Map}(P_{\infty},Q)(A)$, and we know that
\[
\underline{Map}(P_{\infty},Q)(A) \simeq \mathbb{R}\underline{Spec}_{C(P_{\infty},Q)}(A),
\]
which implies
\[
[P_{\infty},Q\otimes_eA]_{Ho(Prop)}\cong [C(P_{\infty},Q),A]_{CDGA_{\mathbb{K}}}.
\]
This bijection sends the homotopy class of $[x\mapsto \varphi(x)\otimes_e 1_A]$ to a homotopy class
of cgdas morphisms from which we pick up a $A$-point $x_{\varphi}:\mathbb{R}\underline{Spec}_A\rightarrow
\mathbb{R}\underline{Spec}_{C(P_{\infty},Q)}$.

We then state the main theorem of this section:
\begin{thm}
Let $A$ be a perfect cdga (a cdga which is perfect as a cochain complex) and $n$ be a natural integer.

(1) We have isomorphisms
\[
\pi_*\mathbb{D}er_{\underline{Map}(P_{\infty},Q)}(\mathbb{R}\underline{Spec}_A,A[n]) \cong
H^{-*-n+1}(Hom_{\Sigma}(\overline{C},Q)^{\varphi}\otimes A)
\]
for every natural integer $*$ such that $*+n\geq 1$.

(2) We have isomorphisms
\[
H^{-*}(\mathbb{T}_{\underline{Map}(P_{\infty},Q),x_{\varphi}}) \cong
H^{-*+1}(Hom_{\Sigma}(\overline{C},Q)^{\varphi}\otimes A)
\]
for every $*\geq 1$.
\end{thm}

To prove this theorem, we need the following preliminary results:
\begin{lem}
Let $F:\mathcal{C}\rightleftarrows\mathcal{G}:G$ be a Quillen adjunction (respectively a Quillen equivalence).
Then for every object $C$
of $\mathcal{C}$, the induced adjunction between the comma categories
\[
F:C/\mathcal{C}\rightleftarrows F(C)/\mathcal{D}:G
\]
is also a Quillen adjunction (respectively a Quillen equivalence).
\end{lem}
\begin{cor}
(1) The geometric realization functor and the singular complex functor induce a Quillen equivalence
\[
|-|:sSet_*\rightleftarrows Top_*:Sing_{\bullet}(-)
\]
between pointed topological spaces and pointed simplicial sets.

(2) Sullivan's realization functor and functor of piecewise linear forms \cite{Sul} induce a Quillen adjunction
\[
A_{PL}:sSet_*\rightleftarrows (CDGA_{\mathbb{K}}^{aug})^{op}:<->
\]
between pointed simplicial sets and the opposite category of augmented commutative differential graded algebras.
\end{cor}
\begin{proof}
(1) The geometric realization sends the standard $0$-simplex to the point.

(2) In rational homotopy theory, there is a Quillen adjunction
\[
A_{PL}:sSet\rightleftarrows CDGA_{\mathbb{K}}^{op}:<->
\]
(see Section 8 of \cite{BG}). The functor $A_{PL}$ sends the standard $0$-simplex to the base field $\mathbb{K}$,
so we get a Quillen adjunction
\[
A_{PL}:sSet_*\rightleftarrows CDGA_{\mathbb{K}}^{op}/\mathbb{K}:<->.
\]
The right-side category is nothing but the opposite category of the category of augmented
cdgas.
\end{proof}

Using (1) we get:
\begin{lem}
For every pointed simplicial set $K$ and every integer $n$, the topological spaces $|Map_{sSet_*}(\partial\Delta^{n+1},
K)|$ and $\Omega^n|K|$ (the iterated loop space of the pointed topological space $|K|$) have the same homotopy type.
\end{lem}
\begin{proof}
We have
\begin{eqnarray*}
Map_{sSet_*}(\partial\Delta^{n+1},K) & \simeq & Map_{sSet_*}(\partial\Delta^{n+1},Sing_{\bullet}(|K|)) \\
 & \simeq & Map_{Top_*}(|\partial\Delta^{n+1}|,|K|) \\
 & \simeq & Map_{Top_*}(S^n,|K|) \\
\end{eqnarray*}
hence
\[
|Map_{sSet_*}(\partial\Delta^{n+1},K)|\simeq |Map_{Top_*}(S^n,|K|)|\sim \Omega^n|K|.
\]
\end{proof}

Using (2) we get:
\begin{prop}
Let $A$ be a perfect cdga and $x:C\rightarrow A$ be a cdga over $A$.
For every integer $n$ there is a homotopy equivalence
\[
\Omega_x^n|Map_{CDGA_{\mathbb{K}}}(C,A)| \simeq |Map_{CDGA_{\mathbb{K}}/A}(C,A\oplus A[n])|
\]
where $A\oplus A[n]$ is the square-zero extension of $A$ by the $A$-module $A[n]$.
\end{prop}
\begin{proof}
We have
\[
\Omega_x^n|Map_{CDGA_{\mathbb{K}}}(C,A)| \simeq |Map_{sSet_*}(\partial\Delta^{n+1},(Map_{CDGA_{\mathbb{K}}}(C,A),x))|
\]
and
\begin{eqnarray*}
Map_{CDGA_{\mathbb{K}}}(C,A) & \simeq & Map_{CDGA_{\mathbb{K}}}(C\otimes A^{\vee},\mathbb{K}) \\
 & = & <C\otimes A^{\vee}>
\end{eqnarray*}
since we assume $A$ to be perfect. Moreover, the cdga $C\otimes A^{\vee}$ is endowed with an augmentation
$C\otimes A^{\vee}\rightarrow \mathbb{K}$ induced by  the base point $x:C\rightarrow A$.
We deduce
\begin{eqnarray*}
Map_{sSet_*}(\partial\Delta^{n+1},(Map_{CDGA_{\mathbb{K}}}(C,A),x)) & \simeq &
Map_{sSet_*}(\partial\Delta^{n+1},<C\otimes A^{\vee}>) \\
 & \cong & Map_{CDGA_{\mathbb{K}}^{aug}}(C\otimes A^{\vee}, A_{PL}(\partial\Delta^{n+1})).
\end{eqnarray*}
The space $\partial\Delta^{n+1}$ is formal (see Definition 2.85 in \cite{FOT}), so there is a zigzag of quasi-isomorphisms of cdgas
\[
A_{PL}(\partial\Delta^{n+1})\stackrel{\sim}{\leftarrow}M_{\partial\Delta^{n+1}}\stackrel{\sim}{\rightarrow} H^*(|\partial\Delta^{n+1}|,\mathbb{K})
\]
where $M_{\partial\Delta^{n+1}}$ is the minimal model of $\partial\Delta^{n+1}$, that is, the minimal model of $A_{PL}(\partial\Delta^{n+1})$ in $CDGA_{\mathbb{K}}$
(see Definition 2.30 in \cite{FOT}).
Alternatively, one could use that the manifold $S^n$ is formal, and then that the homotopy equivalence between $|\partial\Delta^{n+1}|$ and $S^n$ induces a homotopy
equivalence of cdgas between $A_{PL}(\partial\Delta^{n+1})$ and $A_{PL}(S^n)$.
Hence we get
\[
Map_{sSet_*}(\partial\Delta^{n+1},(Map_{CDGA_{\mathbb{K}}}(C,A),x))
\sim Map_{CDGA_{\mathbb{K}}^{aug}}(C\otimes A^{\vee}, H^*S^n).
\]
The fact that $A$ is perfect implies an adjunction
\[
-\otimes A^{\vee}:CDGA_{\mathbb{K}}/A\rightleftarrows CDGA_{\mathbb{K}}^{aug}:-\otimes  A.
\]
Weak equivalences and fibrations of augmented cdgas are quasi-isomorphisms and surjections at the level
of cochain complexes, and the dg tensor product preserves such kind of maps so $-\otimes A$ is a right
Quillen functor. Thus we have a Quillen adjunction $(-\otimes A^{\vee},-\otimes A)$ and natural isomorphisms
of simplicial sets
\[
Map_{CDGA_{\mathbb{K}}^{aug}}(C\otimes A^{\vee}, H^*S^n)\cong Map_{CDGA_{\mathbb{K}}/A}(C, H^*S^n\otimes A).
\]
Finally let us observe that $H^*S^n=\mathbb{K}[t]/(t^2)$ with $|t|=n$, hence $H^*S^n\otimes A=A\oplus A[n]$
is the square-zero extension of $A$ by $A[n]$.
\end{proof}

\begin{proof}[Proof of Theorem 6.2]
(1) We use the previous results to get the following sequence of isomorphisms:
\begin{eqnarray*}
\pi_*\mathbb{D}er_{\underline{Map}(P_{\infty},Q)}(\mathbb{R}\underline{Spec}_A,A[n]) & = & \pi_*Map_{\mathbb{R}\underline{Spec}_A/St}
(\mathbb{R}\underline{Spec}_{A\oplus A[n]},\underline{Map}(P_{\infty},Q)) \\
 & \cong & \pi_*Map_{\mathbb{R}\underline{Spec}_A/St}
(\mathbb{R}\underline{Spec}_{A\oplus A[n]},\mathbb{R}\underline{Spec}_{C(P_{\infty},Q)}) \\
& \cong & \pi_*Map_{CDGA_{\mathbb{K}}/A}(C(P_{\infty},Q),A\oplus A[n]) \\
& \cong & \pi_*\Omega_{x_{\varphi}}^n|Map_{CDGA_{\mathbb{K}}}(C(P_{\infty},Q),A)| \\
& \cong & \pi_{*+n}(Map_{CDGA_{\mathbb{K}}}(C(P_{\infty},Q),A),x_{\varphi}) \\
& \cong & \pi_{*+n}(\underline{Map}(P_{\infty},Q)(A),\varphi\otimes_e id_A) \\
&\cong & H^{-*-n+1}Hom_{\Sigma}(\overline{C},Q\otimes_e A)^{\varphi\otimes_e id_A} \\
& \cong & H^{-*-n+1}(Hom_{\Sigma}(\overline{C},Q)^{\varphi}\otimes A).
\end{eqnarray*}

(2) We have
\begin{eqnarray*}
\pi_*\mathbb{D}er(\mathbb{R}\underline{Spec}_A,A[n]) & \cong & H^{-*}(\mathbb{T}_{\underline{Map}(P_{\infty},Q),x_{\varphi}}
\otimes_A A[n]) \\
 & \cong & H^{-*+n}(\mathbb{T}_{\underline{Map}(P_{\infty},Q),x_{\varphi}})
\end{eqnarray*}
hence
\[
H^{-*}(\mathbb{T}_{\underline{Map}(P_{\infty},Q),x_{\varphi}}) \cong
H^{-*+1}(Hom_{\Sigma}(\overline{C},Q)^{\varphi}\otimes A)
\]
for $n=0$ and $*\geq 1$.
\end{proof}

\subsection{Obstruction theory}

For every integer $n$, there is a homotopy pullback
\[
\xymatrix{
\mathbb{K}[t]/(t^{n+1})\ar[r]\ar[d] & \mathbb{K}[t]/(t^n)\ar[d] \\
\mathbb{K} \ar[r] & \mathbb{K}\oplus\mathbb{K}[1]
}
\]
inducing a fiber sequence of pointed simplicial sets
\begin{eqnarray*}
Map_{\mathbb{R}\underline{Spec}_A/St}(\mathbb{R}\underline{Spec}_{\mathbb{K}[t]/(t^{n+1})},\underline{Map}(X_{\infty},Y)) & \rightarrow  & \\
Map_{\mathbb{R}\underline{Spec}_A/St}(\mathbb{R}\underline{Spec}_{\mathbb{K}[t]/(t^n)},\underline{Map}(X_{\infty},Y)) & \rightarrow  & \\
Map_{\mathbb{R}\underline{Spec}_A/St}(\mathbb{R}\underline{Spec}_{\mathbb{K}\oplus \mathbb{K}[1]},\underline{Map}(X_{\infty},Y)) & & \\
\end{eqnarray*}
hence a long exact sequence
\begin{eqnarray*}
... & \rightarrow & \pi_0Map_{\mathbb{R}\underline{Spec}_A/St}(\mathbb{R}\underline{Spec}_{\mathbb{K}[t]/(t^{n+1})},\underline{Map}(X_{\infty},Y)) \\
 & \stackrel{p_n}{\rightarrow} & \pi_0Map_{\mathbb{R}\underline{Spec}_A/St}(\mathbb{R}\underline{Spec}_{\mathbb{K}[t]/(t^n)},\underline{Map}(X_{\infty},Y)) \\
 & \rightarrow & \pi_0Map_{\mathbb{R}\underline{Spec}_A/St}(\mathbb{R}\underline{Spec}_{\mathbb{K}\oplus \mathbb{K}[1]},\underline{Map}(X_{\infty},Y)) \\
  & \cong & H^1(\mathbb{T}_{\underline{Map}(X_{\infty},Y),x_{\varphi}}).
\end{eqnarray*}
where the last isomorphism follows from Proposition 6.1.
Consequently the first cohomology group of the deformation complex is the obstruction group for formal
deformations:
\begin{prop}
If $H^1(\mathbb{T}_{\underline{Map}(X_{\infty},Y),x_{\varphi}})=0$ then for every integer $n$, the
map $p_n$ is surjective, that is, every deformation of order $n$ lifts to a deformation of order $n+1$.
Thus any infinitesimal deformation of $\varphi$ can be extended to a formal deformation.
\end{prop}
Unicity results for such lifts follows from Proposition 1.4.2.5 of \cite{TV2}:
\begin{prop}
Suppose that a given deformation of order $n$ lifts to a deformation of order $n+1$. Then the set of such lifts forms a torsor under the action
of the cohomology group $H^0(\mathbb{T}_{\underline{Map}(X_{\infty},Y),x_{\varphi}})$. In particular, if
$H^0(\mathbb{T}_{\underline{Map}(X_{\infty},Y),x_{\varphi}})=0$ then such a lift is unique up to equivalence.
\end{prop}

\begin{rem}
These results work as well for $A$-deformations, where $A$ is a perfect cdga.
\end{rem}

\subsection{Zariski open immersions}

In classical algebraic geometry, Zariski open immersions are open immersions for the Zariski topology on the category of affine schemes $Aff$.
A morphism of affine schemes $Spec_B\rightarrow Spec_A$ is a Zariski open immersion if the corresponding morphism of commutative algebras $A\rightarrow B$
is a flat morphism of finite presentation such that $B$ is canonically isomorphic to $B\otimes_A B$.
Schemes are then defined as sheaves over $Aff^{op}$ for the Zariski topology, that is, by gluing affine schemes along Zariski open immersions.
In derived algebraic geometry, a morphism of simplicial algebras $A\rightarrow B$ is a Zariski open immersion if the induced map
\[
Spec_{\pi_0B}\rightarrow Spec_{\pi_0A}
\]
is a Zariski open immersion in the classical sense and the canonical map
\[
\pi_*A\otimes_{\pi_0A}\pi_0B\rightarrow \pi_*B
\]
is an isomorphism.

There is a general notion of Zariski open immersion for affine stacks in any HAG context. However, an explicit description can be more difficult to obtain than
in the derived setting: this is what happens for instance in the complicial context, where there is no simple description in terms of homotopy groups anymore.
We refer the reader to Chapter 1.2 of \cite{TV2} for the definitions of finitely presented morphisms, flat morphisms, epimorphisms, formal Zariski open immersions
and Zariski open immersions for monoids in a given HA context.
An interesting feature of these open immersions is that, intuitively, local properties in a neighborhood of a point in the source are transposed to local
properties in a neighborhood of its image in the target. In particular, if $f:X\rightarrow Y$ is an open immersion and $x:\underline{\mathbb{R}Spec}_R\rightarrow X$
an $R$-point of $X$ for a given commutative monoid $R$, then the cotangent complexes $\mathbb{L}_{X,x}$ and $\mathbb{L}_{Y,f(x)}$ are quasi-isomorphic.

In complicial algebraic geometry, we provide the following conditions under which a morphism of cofibrant cdgas is a formal Zariski open immersion:
\begin{prop}
Let $f:C_2\rightarrow C_1$ be a morphism of cofibrant cdgas. Let us suppose that

(1) For every cdga $A$, the induced map
\[
\pi_0f^*:\pi_0\underline{\mathbb{R}Spec}_{C_1}(A)\hookrightarrow \pi_0\underline{\mathbb{R}Spec}_{C_2}(A)
\]
is injective;

(2) For every cdga morphism $\varphi:C_1\rightarrow A$, the induced map
\[
\pi_{*+1}(\underline{\mathbb{R}Spec}_{C_1}(A),\varphi)\stackrel{\cong}{\rightarrow}\pi_{*+1}(\underline{\mathbb{R}Spec}_{C_1}(A),\varphi\circ f)
\]
is an isomorphism.

(3) The map $f$ is finitely presented.

Then $f$ is a Zariski open immersion.
\end{prop}
\begin{proof}
Let $f:C_2\rightarrow C_1$ be a morphism of cdgas. It induces a morphism of Kan complexes
\[
f^*=(-\circ f):\underline{\mathbb{R}Spec}_{C_1}(A)=Map_{CDGA_{\mathbb{K}}}(C_1,A)\rightarrow Map_{CDGA_{\mathbb{K}}}(C_2,A)=\underline{\mathbb{R}Spec}_{C_2}(A).
\]
Let us fix a morphism $\varphi:C_2\rightarrow A$ and consider the homotopy fiber
\[
hofib_{\varphi}(f^*)\rightarrow Map_{CDGA_{\mathbb{K}}}(C_1,A)\rightarrow Map_{CDGA_{\mathbb{K}}}(C_2,A).
\]
This homotopy fiber is the Kan subcomplex of $Map_{CDGA_{\mathbb{K}}}(C_1,A)$ whose vertices are the maps $\psi:C_1\rightarrow A$ such that $\psi\circ f=\varphi$,
that is, the mapping space $Map_{C_2\setminus CDGA_{\mathbb{K}}}(C_1,A)$ of morphisms between $f:C_2\rightarrow C_1$ and $\varphi:C_2\rightarrow A$.
This is nothing but the mapping space of $C_2$-algebra morphisms
\[
hofib_{\varphi}(f^*)=Map_{C_2-Alg}(C_1,A).
\]
Using the long exact sequence of homotopy groups associated to this homotopy fiber sequence, assumptions (1) and (2) imply that the homotopy fiber
$hofib_{\varphi}(f^*)$ is contractible. According to Definition 1.2.6.1 of \cite{TV2}, the fact that $Map_{C_2-Alg}(C_1,A)$ is contractible means that
$f$ is an epimorphism. Unbounded cochain complexes form a stable model category, so every morphism is flat and we proved that $f$ is
a formal Zariski open immersion (see Corollary 1.2.6.6 of \cite{TV2}).
By definition, a finitely presented formal Zariski open immersion is a Zariski open immersion (Definition 1.2.6.7 of \cite{TV2}).
\end{proof}
This applies to the affine stacks $\underline{Map}_{T-Alg}(X_{\infty},Y)$: under finitess conditions on a $T$-algebra morphism $f:X_{\infty}\rightarrow X_{\infty}'$,
conditions (1) and (2) of Proposition 6.11 about homotopy groups of mapping spaces imply that the deformation complex of any $\varphi:X_{\infty}\rightarrow Y$
is quasi-isomorphic to the deformation complex of $\varphi\circ f$.
For instance, let $T$ be the monad encoding props and $Y=End_X$ be the endomorphism prop of a cochain complex $X$.
Let $O_{\infty}\rightarrow P_{\infty}$ be a morphism of cofibrant props.
Suppose that $X$ is a $P_{\infty}$-algebra. Proposition 6.11 gives homotopical conditions under which
the deformation complex of $X$ as $P_{\infty}$-algebra is quasi-isomorphic to the deformation complex of $X$ as $O_{\infty}$-algebra.
This can be seen as a partial generalization of Theorem 0.1 (homotopical deformation theory controls algebraic deformation theory).

\section{Another description}

This short section is devoted to the proof of Theorem 0.14. For this, we use the following general results about Lie theory
of profinite complete $L_{\infty}$-algebras. The first is about the affineness of such stacks:
\begin{lem}
Let $g$ be a profinite complete $L_{\infty}$-algebra. Then the Maurer-Cartan stack $\underline{MC}_{\bullet}(g)$
is isomorphic to $\mathbb{R}Spec_{C^*(g)}$, where $C^*(g)$ is the Chevalley-Eilenberg algebra of $g$.
\end{lem}
\begin{proof}
According to Lemma 2.3 of \cite{Ber}, for every cdga $A$ there is an isomorphism
\[
Mor_{cdga}(C^*(g),A) \cong MC(g\hat{\otimes} A)
\]
which is natural in $A$, hence
\begin{eqnarray*}
\underline{MC}_{\bullet}(g)(A) & = & MC_{\bullet}(g\hat{\otimes}A) \\
 & = & MC((g\hat{\otimes}A)\hat{\otimes}\Omega_{\bullet}) \\
 & = & MC(g\hat{\otimes}(A\otimes\Omega_{\bullet})) \\
 & \cong & Mor_{cdga}(C^*(g),A\otimes\Omega_{\bullet}) \\
 & = & \mathbb{R}Spec_{C^*(g)}(A)
\end{eqnarray*}
where the isomorphism preserves the simplicial structure by naturality.
\end{proof}
The second is about the cohomology of their tangent complexes:
\begin{lem}
Let $g$ be a profinite complete $L_{\infty}$-algebra and $\varphi$ be a Maurer-Cartan element of $g$, that is,
a $\mathbb{K}$-point of $\mathbb{R}Spec_{C^*(g)}$. Then the tangent cohomology of the Maurer-Cartan stack
is given by
\[
H^*(\mathbb{T}_{\underline{MC}_{\bullet}(g),\varphi}[-1]) \cong H^*g^{\varphi}.
\]
\end{lem}
\begin{proof}
Recall that we have a bijection between $MC(g)$ and $Mor_{cdga}(C^*(g),\mathbb{K})$, so that
Maurer-Cartan elements of $g$ correspond to augmentations of $C^*(g)$. We still note $\varphi$ the augmentation associated
to the Maurer-Cartan element $\varphi$. Lemma 7.1 gives the isomorphisms
\[
\mathbb{L}_{\underline{MC}_{\bullet}(g),\varphi} \cong \mathbb{L}_{C^*(g),\varphi} \cong
\mathbb{L}_{C^*(g)/\mathbb{K}_{\varphi}}.
\]
The last complex is the relative cotangent complex of $C^*(g)$ with respect to $\mathbb{K}_{\varphi}$,
which denotes $\mathbb{K}$ equipped with the $C^*(g)$-algebra structure defined by the augmentation $\varphi$.
Theorem 4.1. of \cite{Laz} states a Quillen equivalence between the category $CGDA^{c,f}$ of connected formal cdgas
and the category $DGLA^{pro}$ of dg Lie algebras (dgla for short) with pronilpotent (i.e. complete) homology, for which the Chevalley-Eilenberg algebra functor $C^*(-)$ is the right adjoint. The left adjoint is noted $L$. This is a dualized version of the classical bar-cobar equivalence between augmented cdgas and conilpotent dg Lie coalgebras. Then, given a dgla $g$ and a formal cdga $A$,
Lazarev defines a relative Maurer-Cartan set $MC(g,A)$ such that
\[
MC(L(A),B) \cong Mor_{CDGA^{c,f}}(A,B)
\]
for every formal cdgas $A$ and $B$.
Let us note $\xi$ the Maurer-Cartan element of $MC(LC^*(g),\mathbb{K})$ corresponding to the augmentation $\varphi$.
Then, according to Section 7 of \cite{Laz}, the Harrison complex of $C^*(g)$ with coefficients in $\mathbb{K}_{\varphi}$
is given by
\[
C^*_{Har}(C^*(g),\mathbb{K}_{\varphi}) = LC^*(g)^{\xi}
\]
where the dg Lie algebra $LC^*(g)$ is twisted by $\xi$. Since $L$ and $C$ define a Quillen equivalence, the counit of the
adjunction $\eta:LC\rightarrow Id$ gives a quasi-isomorphism of complete dg Lie algebras $\eta(g):LC^*(g)\stackrel{\sim}{\rightarrow}g$.
By Proposition 3.8 of \cite{Yal3} this remains a quasi-isomorphism under twisting
\[
LC^*(g)^{\xi}\stackrel{\sim}{\rightarrow}g^{\eta(g)(\xi)}.
\]
By construction, the last twisting is nothing but $\eta(g)(\xi)=\varphi$. This induces an isomorphism between
the Harrison cohomology $Har^*(C^*(g),\mathbb{K}_{\varphi})$ and $H^*g^{\varphi}$.
The Harrison cohomology with degree shifted by one $Har^{*+1}(C^*(g),\mathbb{K}_{\varphi})$, in turn,
is isomorphic to the André-Quillen cohomology of $C^*(g)$ with coefficients in $\mathbb{K}_{\varphi}$
(see Chapter 12 of \cite{LV} for instance),
which is nothing but the cohomology of the tangent complex of $C^*(g)$ at $\varphi$. We finally get
\[
H^*(\mathbb{T}_{\underline{MC}_{\bullet}(g),\varphi}) \cong H^{*+1}g^{varphi}.
\]
\end{proof}
\begin{rem}
Let us note that an alternative proof consists in comparing $Har^*(C^*(g),\mathbb{K}_{\varphi})$ with
$Har^*(C^*(g^{\varphi}),\mathbb{K})$ (the Harrison complex of $C^*(g^{\varphi})$ with trivial coefficients)
by spectral sequence arguments for filtered complexes, then using the usual bar-cobar Quillen equivalence
which gives a quasi-isomorphism between $Har^*(C^*(g^{\varphi}),\mathbb{K})$ and $g^{\varphi}$.
\end{rem}

The proof of Theorem 0.14 then follows by using again the filtration defined in the proof of Proposition 15 of \cite{Mar3}.
We already established in the proof of Theorem 2.17 that such a filtration makes $Hom_{\Sigma}(\overline{C},Q)$
into a complete dg Lie algebra. Recall that this filtration is defined by a decomposition
\[
Hom_{\Sigma}(\overline{C},Q) = \prod_{l\geq 1} Hom_{\Sigma}(\overline{C},Q)_l
\]
and
\[
F_rHom_{\Sigma}(\overline{C},Q) = \prod_{l\geq r} Hom_{\Sigma}(\overline{C},Q)_l.
\]
The assumptions of Theorem 0.14 implies that the product indexed by $m,n\in \mathbb{N}$
defining each $Hom_{\Sigma}(\overline{C},Q)_l$ (see Proposition 15 of \cite{Mar3}) is finite (hence a finite direct sum)
and that its components are of finite dimension. Then each
\[
Hom_{\Sigma}(\overline{C},Q)/F_rHom_{\Sigma}(\overline{C},Q)\cong \prod_{l=1}^{r-1}Hom_{\Sigma}(\overline{C},Q)_l
\]
is a finite dimensional dg Lie algebra, making $Hom_{\Sigma}(\overline{C},Q)$ into a profinite complete dg Lie algebra.

\section{Perspectives}

\subsection{Deformation theory for morphisms of algebras over polynomial monads}

Under the assumptions of Theorem 4.14 we have
\[
\underline{Map}_{T-Alg}(X_{\infty},Y) \sim \mathbb{R}\underline{Spec}_{C(X_{\infty},Y)},
\]
and for a given $A$-point $x_{\varphi}:\mathbb{R}\underline{Spec}_A\rightarrow \underline{Map}_{T-Alg}(X_{\infty},Y)$
corresponding to a morphism $\varphi\otimes A:X_{\infty}\otimes_e A\rightarrow Y\otimes_e A$,
the tangent complex is
\begin{eqnarray*}
\mathbb{T}_{\underline{Map}(X_{\infty},Y),x_{\varphi}} & = & \mathbb{R}\underline{Hom}_{A-Mod}
(\mathbb{L}_{C(X_{\infty},Y)/A},A) \\
& \cong & Der_A(C(X_{\infty},Y),A)
\end{eqnarray*}
where $\mathbb{R}\underline{Hom}_{A-Mod}$ is the derived hom of the internal hom in the category of
dg $A$-modules.
The complex $\mathbb{T}_{\underline{Map}(X_{\infty},Y),x_{\varphi}}$ is actually a dg Lie algebra, hence
the following definition:
\begin{defn}
The $A$-deformation complex of the $T$-algebras morphism $X_{\infty}\rightarrow Y$ is
$\mathbb{T}_{\underline{Map}(X_{\infty},Y),x_{\varphi}}$.
\end{defn}
In the particular case where $I=\mathbb{S}$, we consider $T$-algebras in the category of $\Sigma$-biobjects of
$\mathcal{C}$. When all the objects of $\mathcal{C}$ are fibrant, all the assumptions of Theorem 4.18 are
satisfied. Consequently, we have a well-defined and meaningful notion of deformation complex of any morphism
of $T$-algebras with cofibrant source. This gives a deformation complex, for instance, for morphisms of
cyclic operads, modular operads, wheeled properads...more generally, morphisms of poynomial monads as defined
in \cite{BB}. In the situation where it makes sense to define an algebraic structure via a morphism towards
an "endomorphism object" (operads, properads, props, cyclic and modular operads, wheeled prop...) this gives
a deformation complex of algebraic structures.

As a perspective for a future work, a better understanding of model categories of algebras over polynomial monads could be used
to obtain results similar to those proven in this paper for properads:
\begin{itemize}
\item for every $\varphi:X_{\infty}\rightarrow Y$, a quasi-isomorphism
\[
\mathbb{T}_{\underline{Map}(X_{\infty},Y),x_{\varphi}} \sim Der_{\varphi}(X_{\infty},Y)
\]
inducing, by transfer, a $L_{\infty}$ algebra structure on the derivations complex;
\item A homotopy equivalence
\[
Map(X_{\infty},Y) \sim MC_{\bullet}(Der_{\varphi}(X_{\infty},Y))
\]
giving as corollaries the group isomorphisms
\[
\pi_{n+1}(Map(X_{\infty},Y),\varphi)\cong H^nDer_{\varphi}(X_{\infty},Y)^{\varphi}
\]
for every integer $n>1$ and the isomorphism of pronilpotent Lie algebras
\[
H^0Der_{\varphi}(X_{\infty},Y)^{\varphi} \cong Lie(\pi_1(Map(X_{\infty},Y),\varphi));
\]
\item if $T$-algebras encode a certain kind of algebraic structure, and the notion of "endomorphism object" of a complex $C$ makes sense,
study the cohomology theory given by $Der_{\varphi}(X_{\infty},End_C)$;
\item the obstruction theory of the previous section.
\end{itemize}
The proofs of this paper are quite general, and the author thinks two essential ingredients to transpose them to this even more general setting are
a good understanding of the notion of cofibrant resolution for algebras over polynomial monads and the subsequent analysis of Maurer-Cartan elements in
$Der_{\varphi}(X_{\infty},Y)$.

\subsection{Shifted symplectic structures}

We describe briefly how shifted symplectic and Poisson structures are defined in the affine case, and refer
the reader to \cite{PTVV} for more details about these notions for derived Artin stacks.
Let $A$ be a cdga. For simplicity we suppose that $A$ is cofibrant (the derived spectrum is defined with a cofibrant
resolution). We denote its $A$ dg module of Kähler differentials by $\Omega^1_A$.
For every integer $i\geq 0$, the exterior power in $A$-modules $\Omega_A^i=\Lambda_A^i\Omega_A^1$ has a de Rham
differential $d_R:\Omega_A^i\rightarrow\Omega_A^{i+1}$. The complex of closed $p$-forms is defined by
\[
\mathcal{A}^{p,cl}(A)=\prod_{i\geq 0}\Omega^{p+i}_A[-i]
\]
with the differential of degree $1$ given by $D(\{\omega_i\})=\{d_R(\omega_{i-1})+d(\omega_i)\}$, where $d$ is the
differential of $\Omega^{p+i}_A$. A closed $p$-form of degree $n$ is then a cohomology class of degree $n$ in
$H^n\mathcal{A}^{p,cl}(A)$. This is a series of elements $\omega_i$ satisfying relations corresponding to the fact
of being a cocycle. To any closed $p$-form $\omega=\{\omega_i\}$
of degree $n$ corresponds a $p$-form, actually the term $\omega_0$. For $p=2$,
it is equivalent to a morphism $\theta_{\omega}:\mathbb{T}_A\rightarrow\mathbb{L}_A[n]$ in the homotopy
category of $A$-modules $Ho(Mod_A)$, where
$\mathbb{L}_A$ is the cotangent complex of $A$ and $\mathbb{T}_A=Hom_{Mod_A}(\mathbb{L}_A,A)$ its tangent complex.
An $n$-shifted symplectic structure on the affine stack $\underline{\mathbb{R}Spec}_A$ is a closed $2$-form $\omega$
of degree $n$ such that the associated map $\theta_{\omega}$ is an isomorphism in $Ho(Mod_A)$.

One can define an algebra of shifted polyvector fields for any derived Artin stack locally of finite presentation.
In the affine case, it is defined by
\[
Pol(A,n)=\Lambda^{\bullet}_A(\mathbb{T}_A[-1-n])
\]
and equipped with a Schouten-Nijenhuis bracket, so that $Pol(A,n)[n+1]$ is a graded dg Lie algebra, with a
cohomological grading and a weight grading (associated to the exterior powers). The space of $n$-shifted Poisson structures
on $\underline{\mathbb{R}Spec}_A$ is the mapping space of graded dg Lie algebra morphisms
\[
Pois(A,n)=Map_{dgLie^{gr}}(\mathbb{K}(2)[1],Pol(A,n)[n+1])
\]
where $\mathbb{K}(2)[1]$ is concentrated in cohomological degree $1$ and weight $2$ and equipped with the zero bracket.
Equivalently, Poisson structures are Maurer-Cartan elements of weight $2$ (``bivectors'') in
$Pol(A,n)[n+1]\otimes t\mathbb{K}[[t]]$.

It is strongly conjectured, but not proved in full generality yet, that an $n$-shifted symplectic structure
induces an $n$-shifted Poisson structure. The $n$-quantization process then consists in a one-parameter formal deformation of
the dg category of quasi-coherent complexes on a derived Artin stack, viewed as an $E_n$-monoidal dg category.
This relies on a higher formality conjecture, providing a quasi-isomorphism of dg Lie algebras between
shifted polyvector fields on a derived Artin stack and higher Hochschild homology of its dg category of quasi-coherent complexes.

In a future work we intend to study the existence of shifted symplectic structures on $\underline{Map}_{T-Alg}(X_{\infty},Y)$
for a monad $T$ general enough. We would like to include in particular the case of polynomial monads \cite{BB},
whose algebras are operads, properads, props and their cyclic, modular, wheeled versions. When a $X_{\infty}$-algebra
structure on a complex $C$ can be encoded by morphisms of the form $X_{\infty}\rightarrow End_C$, this gives
a shifted symplectic structure on the moduli stack of $X_{\infty}$-algebra structures on $C$.
By this way, the program started in \cite{PTVV} could lead to a general theory of deformation quantization of algebraic
structures. We hope to discover new deformation quantizations as well as recover the known cases (like quantization
morphisms of Lie bialgebras), and use this approach to get new properties of these.
New deformation quantizations include, for instance, the quantization of the moduli stack of complex structures on a formal
manifold, which are described in \cite{Mil2} as algebras over a certain cofibrant operad.

\end{document}